\documentclass[11pt]{amsart}

\usepackage{amsmath,amssymb,amsthm}
\usepackage{eucal}
\usepackage[colorlinks=true,backref=page]{hyperref}
\usepackage[all]{xy}
\usepackage{seqsplit}
\usepackage{autobreak}
\usepackage{here}

\allowdisplaybreaks

\numberwithin{equation}{section}

\SelectTips{eu}{12}

\newtheorem{theorem}{Theorem}[section]
\newtheorem{proposition}[theorem]{Proposition}
\newtheorem{lemma}[theorem]{Lemma}
\newtheorem{corollary}[theorem]{Corollary}

\theoremstyle{definition}

\theoremstyle{remark}

\numberwithin{equation}{section}

\newcommand{\Z}{\mathbb{Z}}
\newcommand{\Q}{\mathbb{Q}}

\newcommand{\Hom}{\mathrm{Hom}}
\newcommand{\map}{\mathrm{map}}

\SelectTips{cm}{}

\title[The space of commuting elements and classifying spaces]{The space of commuting elements in an exceptional Lie group and maps between classifying spaces}

\author[M. Takeda]{Masahiro Takeda}
\address{Institute for Liberal Arts and Sciences, Kyoto University, Kyoto, 606-8316, Japan}
\email{takeda.masahiro.87u@kyoto-u.ac.jp}

\date{\today}



\begin{document}

\maketitle

\begin{abstract}
  Let $\pi$ be a discrete group, and let $G$ be a compact connected Lie group.
  $\Hom(\pi,G)_0$ denotes the null-component of the space of homomorphisms from $\pi$ to $G$, and $\map_*(B\pi,BG)_0$ denotes the null-component of the space of maps from $B\pi$ to $BG$.
  Since the classifying space functor is continuous, there is a continuous map $\Theta\colon\Hom(\pi,G)_0\to\map_*(B\pi,BG)_0$. 
  Atiyah and Bott studied this map when $\pi$ is a surface group, and proved surjectivity in rational cohomology. 
  In this paper, we obtain the condition that the map $\Theta$ is surjective or not in rational cohomology when $\pi$ is $\Z^m$ for $m\geq 3$ and $G$ is a compact connected Lie group. 
\end{abstract}


\section{Introduction}\label{sec:introduction}

For two pointed spaces $X,Y$, let $\map_*(X,Y)$ denotes the space of based maps from $X$ to $Y$.
For two topological groups $H,G$, let $\Hom(H,G)$ be the space of homomorphisms from $H$ to $G$.
Since the classifying space functor is continuous, there is a continuous map
\[
  \widehat{\Theta}\colon \Hom(H,G)\rightarrow \map_*(BH,BG).
\]
Let $\Hom(H,G)_0, \map_*(BH,BG)_0$ be the null-component of $\Hom(H,G), \map_*(BH,BG)$ respectively.
Since the trivial homomorphism in $\Hom(H,G)_0$ is sent to the constant map to the base point by $\widehat{\Theta}$, there is a following map as a restriction of $\widehat{\Theta}$
\[
  \Theta\colon \Hom(H,G)_0\rightarrow \map_*(BH,BG)_0.
\]

These maps $\widehat{\Theta},\Theta$ are of interest from several points of view.
Here we will show two points of view.
The first relates to $A_\infty$-maps.
The $A_\infty$-map is introduced by Stasheff \cite{St}.
$A_\infty$-map is like a homomorphism between topological monoids replacing the strict associativity by the associativity up to coherent higher homotopies.
For two topological groups $H,G$, let $A_\infty(H,G)$ denotes the space of $A_\infty$-maps from $H$ to $G$.
It is proved in \cite{F,T}, there is a weak homotopy equivalence
\[
  A_\infty(H,G)\rightarrow \map_*(BH,BG).
\]
Since $\widehat{\Theta}$ factors this weak homotopy equivalence, $\widehat{\Theta}$ can be interpreted to the relation between homomorphism, which has strict associativity, and $A_\infty$-map, which has homotopical associativity.
The second relates to the bundle theory. 
When $H$ is a discrete group and $BH$ has the homotopy type of a manifold, $\Hom(H,G)_0$ is identified with the based moduli space of flat connections of trivial $G$-bundle over $BH$ and $\map_*(BH,BG)_0$ is identified with the based moduli space of all connections.
Thus $\Theta$ can be interpreted to the relation between based moduli spaces of flat connections and all connections.
When $H$ is a surface group, Atiyah and Bott \cite{AB} studied this map $\Theta$ by using Morse theory, and in particular they proved that $\Theta$ is surjective in rational cohomology.

After here, let $G$ be a compact connected Lie group.
In this paper we will consider the following map
\[
  \Theta\colon \Hom(B\Z^m,BG)_0\rightarrow \map_*(B\Z^m,BG)_0,
\]
for $m\geq 3$.
Now since there is a natural homeomorphism
\[
  \Hom(B\Z^m,BG)\cong \{(g_1,\dots g_m)\in G^m\mid g_ig_j=g_jg_i \text{ for }1\leq i,j \leq m\},
\]
$\Hom(\Z^m,G)_0$ is called the space of commuting elements in $G$.
Recently, the space of commuting elements in $G$ is studied in algebraic topology, for example \cite{ACT-G,AGG,Ba,BJS,BR,C,GPS,KT1,KT2,KTT,RS1,RS2}.

We will state the main theorem.
For a compact connected Lie group $G$, there exist simple Lie groups $K_1,K_2,\dots K_n$ such that $K_1\times K_2\times \dots \times K_n\times (S^1)^i$ for some $i$ is the finite cover of $G$.
Such a Lie groups $K_1,K_2,\dots K_k$ is called simple factors of $G$.

\begin{theorem}\label{main1}
    Let $G$ be a compact connected Lie group with simple factors $K_1,K_2,\dots K_k$, and $m\geq 3$.
    Then the map 
    \[
        \Theta\colon \Hom(\Z^m,G)_0\rightarrow \map_*(B\Z^m,BG)_0
    \]
    induces surjection in rational cohomology if and only if $K_i$ is equal to $SU(n)$, $Sp(n)$, $Spin(2n+1)$ or $G_2$ for each $i$.
\end{theorem}
It is shown in Section \ref{sec:reduction} we can reduce the main theorem to the case that $G$ is a simple Lie group, and we have obtained such a surjectivity for classical group $G$ in \cite{KTT}.
Thus in this paper we will consider the exceptional Lie group cases.

In Section \ref{sec:reduction}, we review the previous result and we reduce the main theorem to the case that $G$ is a simple Lie group.
In Section \ref{sec:cohomology}, we will show some properties about rational cohomology of $\Hom(\Z^m,G)_0$ and $\map_*(B\Z^m,BG)_0$.
In Section \ref{sec:exceptional_cases}, we prove the main theorem for the case of exceptional Lie groups other than $G_2$ by using computation of Hilbert-Poincar\'e series in Section \ref{sec:computation_of_Poincare}.
In Section \ref{sec:G_2_case} and \ref{sec:G_2_generator}, we prove the main theorem for the case of $G_2$.

Unless otherwise stated, the coefficient of cohomology is $\Q$.

\subsection{Acknowledgements}
The author was supported by JSPS KAKENHI JP24KJ1758.


\section{Reduction to the simple Lie group case}\label{sec:reduction}

In this section we review the previous result and we reduce the main theorem to the case that $G$ is a simple Lie group.
At first, we remind the computation of $H^*(\Hom(\Z^m,G)_0)$ given by Baird \cite{Ba}.
Let $T$ be a maximal torus of $G$ and $W$ be the Weyl group of $G$.
We define a map
\[
  \Phi\colon G/T\times T^m\rightarrow \Hom(\Z^m,G)_0
\]
by $\Phi(gT,(t_1,\dots t_m))=(gt_1g^{-1},\dots gt_mg^{-1})$ for $g\in G$ and $t_1,\dots ,t_m\in T$.
Now $T^m$ is the direct product of the $m$-copies of the maximal torus $T$.
The $G$ action on $G/T\times T^m$ is defined as 
\[
  w\cdot (gT,a_1,\dots ,a_m)=(gwT,w^{-1}a_1w,\dots , w^{-1}a_mw).
\]
Then the rational cohomology of $\Hom(\Z^m,G)_0$ is computed as follows.
\begin{theorem}[Theorem 4.3 in \cite{Ba}]\label{Thm_Baird}
  The map 
  \[
    \Phi^*\colon H^*(\Hom(\Z^m,G)_0) \rightarrow H^*(G/T\times T^m)
  \]
  is an injection in rational cohomology, and the image of $\Phi^*$ is the invariant subring
  \[
    \mathrm{Im}(\Phi^*)=H^*(G/T\times T^m)^W.
  \]
\end{theorem}
This proposition says that $H^*(\Hom(\Z^m,G)_0)$ is isomorphic to $H^*(G/T\times T^m)^W$.

We look at three lemmas.
\begin{lemma}[ref. Lemma 4.2 in \cite{KT1}]\label{Hom_surjective}
  Let $G$ and $H$ are compact connected Lie groups.
  If there is a finite covering homomorphism $H\rightarrow G$, then the map induced by the covering 
  \[
    \Hom(\Z^m,H)_0\rightarrow \Hom(\Z^m,G)_0
  \] 
  is an isomorphism in rational cohomology.
\end{lemma}
This lemma can be proved by Theorem \ref{Thm_Baird}, because a finite covering between Lie groups preserves the Weyl group.
\begin{lemma}\label{map_surjective}
  Let $G$ and $H$ are compact connected Lie groups.
  If there is a finite covering homomorphism $H\rightarrow G$, then the maps induced by the covering 
  \[
    \map_*(B\Z^m,BH)_0\rightarrow \map_*(B\Z^m,BG)_0
  \] 
  is isomorphism in rational cohomology.
\end{lemma}
\begin{proof}
  Let $K$ be the kernel of the finite covering of Lie groups $H\rightarrow G$.
  It is easy to see that $K$ in center of $H$, $K$ is finite abelian group. 
  So the rationalization of $BK$ is contractible.
  The finite covering induces the fibration $BK\rightarrow BH\rightarrow BG$.
  By considering the rationalization of this fibration, we obtain the homotopy equivalence of the rationalization of $BG$ and $BH$.
  Therefore the rationalization of the map $\map_*(B\Z^m,BH)_0\rightarrow \map_*(B\Z^m,BG)_0$ is also homotopy equivalence, and this map induces isomorphism in cohomology.
\end{proof}
\begin{lemma}[ref. Lemma 3.4 in \cite{KTT}]\label{S^1_hom_eq}
  For $n\geq 1$, the map
  \[
    \Theta\colon \Hom(\Z^m,(S^1)^n)_0\rightarrow \map_*(B\Z^m,B(S^1)^n)_0.
  \]
  is a homotopy equivalence.
\end{lemma}

By using these lemmas, we obtain the following proposition.
\begin{proposition}\label{simple_factor_surjective}
  Let $G$ be a compact connected Lie group, and $K_1,\dots K_k$ be the simple factors of $G$.
  Then 
  \[
    \Theta\colon \Hom(\Z^m,G)_0\rightarrow \map_*(B\Z^m,BG)_0
  \]
  is surjection in rational cohomology if and only if for all $1\leq j \leq k$
  \[
    \Theta\colon \Hom(\Z^m,K_i)_0\rightarrow \map_*(B\Z^m,BK_j)_0
  \]
  is surjection in rational cohomology.
\end{proposition}
\begin{proof}
  There is a commuting diagram
  \[
    \xymatrix{
      \Hom(\Z^m,G_1\times \dots G_n\times (S^1)^i)_0 \ar^-{\Theta}[r] \ar[d] & \map_*(B\Z^m,BG_1\times \dots BG_n\times (BS^1)^i)_0 \ar[d] \\
      \Hom(\Z^m,G)_0 \ar^{\Theta}[r] & \map_*(B\Z^m,BG)_0,
    }
  \]
  and by Lemma \ref{map_surjective} and \ref{Hom_surjective} the vertical arrows are isomorphisms.
  Thus the bottom map is surjective if and only if the top map is surjective.
  Since for two Lie groups $H$ and $K$, there is a natural homeomorphism
  \[
    \Hom(\Z^m, H\times K)_0\cong \Hom(\Z^m,H)_0 \times (\Hom(\Z^m,K)_0)
  \]
  and a natural homotopy equivalence
  \[
    \map_*(B\Z^m,BH\times BK)_0\simeq \map_*(B\Z^m,BH)_0 \times \map_*(B\Z^m,BK)_0,
  \]
  the top map is surjective if and only if the maps
  \begin{align*}
    &\Theta \colon \Hom(\Z^m,G_j)_0\rightarrow \map_*(B\Z^m,BG_j)_0 \quad (1\leq j \leq k)\\
    &\Theta \colon \Hom(\Z^m, S^1)_0\rightarrow \map_*(B\Z^m,BS^1)_0
  \end{align*}
  are surjective.
  By Theorem \ref{S^1_hom_eq}, $\Theta \colon \Hom(\Z^m, S^1)_0\rightarrow \map_*(B\Z^m,BS^1)_0$ is a surjection in rational cohomology. 
  Complete the proof.
\end{proof}
By this proposition, to obtain the main theorem it is enough to prove the surjectivity of $\Theta$ for each simple factors.

When $G$ is classical group, in \cite{KTT} we have obtained that the map $\Theta$ is surjective or not as follows.
\begin{theorem}[Theorem 1.2 and 1.5 and Lemma 3.3 in \cite{KTT}]\label{classical_surjective}
  If $G$ is $SU(n+1),Sp(n)$ or $SO(2n+1)$ for $n\geq 1$, then the map for $m\geq 3$
  \[
    \Theta\colon \Hom(\Z^m,G)_0\rightarrow \map_*(B\Z^m,BG)_0
  \]
  is surjective.
  If $G$ is $SO(2n)$ for $n\geq 4$, then $\Theta$ for $m\geq 3$ is not surjective.
\end{theorem}
In this paper, we will consider the exceptional Lie group cases
To prove the following theorems is the main part of this paper.
\begin{theorem}\label{G_2_surjective}
  The map $\Theta\colon \Hom(\Z^m,G_2)_0\rightarrow \map_*(B\Z^m,BG_2)_0$ for $m\geq 1$ is surjection in rational cohomology.
\end{theorem}
\begin{theorem}\label{exceptional_non_surjective}
  When $G$ is $F_4,E_6,E_7$ or $E_8$, the map $\Theta\colon \Hom(\Z^m,G)_0\rightarrow \map_*(B\Z^m,BG)_0$ for $m\geq 3$ is not surjection in rational cohomology.
\end{theorem}
Now we prove the main theorem assuming these two theorems.
\begin{proof}[Proof of Theorem \ref{main1}]
  By combining Theorem \ref{G_2_surjective}, \ref{exceptional_non_surjective} and Proposition \ref{simple_factor_surjective}, we obtain this theorem.
\end{proof}

\section{The rational cohomology of $\Hom(\Z^m,G)$ and $\map_*(B\Z^m,BG)$}\label{sec:cohomology}
In this section we define a bigraded algebra structure of the rational cohomology of $\Hom(\Z^m,G)_0$ and $\map_*(B\Z^m,BG)_0$.
Most part of this section is based on \cite{RS1,KT1}.
\subsection{Bigraded algebra structure of $H^*(\Hom(\Z^m,G))$}

We will consider a bigraded algebra structure of $H^*(\Hom(\Z^m,G)_0)$ and its Hilbert-Poincar\'{e} series.
For $x\in H^i(G/T)\otimes H^j(T^m)$ let the bidegree of $x$, $\deg(x)$, be $(i,j)$.
By this bidegree $H^*(G/T\times T^m)\cong H^*(G/T)\otimes H^*(T^m)$ becomes a bigraded algebra.
By Theorem \ref{Thm_Baird} $H^*(\Hom(\Z^m,G)_0)$ can be regarded as a subring in $H^*(G/T\times T^m)$, so $H^*(\Hom(\Z^m,G)_0)$ has the bigraded algebra structure induced by $H^*(G/T\times T^m)$.

For a bigraded algebra over a field $M=\oplus M_{i,j}$ where $M_{i,j}$ is $(i,j)$ degree part of $M$, we define the Hilbert-Poincar\'{e} series of $M$ as
\[
  P(M,s,t)=\sum_{i,j} \dim(M_{i,j})s^it^j,
\]
where $\dim(M_{i,j})$ is the dimension of $M_{i,j}$ as a linear space.
When $M$ is a cohomology of a space $X$, we sometimes write $P(M,s,t)$ as $P(X,s,t)$.

The Hilbert-Poincar\'{e} series of $H^*(\Hom(\Z^m,G)_0)$ is computed in \cite{RS1} as follows.
\begin{theorem}\label{bigraded_Poincare_formula}
  Let $d_1,\dots d_r$ be the characteristic degree of $W$.
  The Hilbert-Poincar\'{e} series of $H^*(\Hom(\Z^m,G)_0)$ is 
  \[
    P(\Hom(\Z^m,G)_0,s,t)=\frac{1}{|W|}\prod_{i=1}^{r}(1-s^{d_i})\sum_{w\in W}\frac{(\det(1+tw))^m}{\det(1-s^2w)},
  \]
  where the determinants are taken as the reflection group structure of $W$.
\end{theorem}

\subsection{Bigraded ring structure of $H^*(\map_*(B\Z^m,BG))$}
Next we will give a bigraded algebra structure of $H^*(\map_*(B\Z^m,BG)_0)$ compatible with that of $H^*(\Hom(\Z^m,G)_0)$.

Let $r$ be the rank of $G$.
It is well known that the rational cohomology of $BG$ is isomorphic to a polynomial ring generated by $r$ elements.
We denote 
\[
  H^*(BG)=\Q[z_1,z_2,\dots z_r].
\]
Since $B\Z^m$ is homotopy equivalent to the direct product of $m$ copies of $S^1$, 
\[
  H^*(B\Z^m)=\Lambda(t_1,\dots t_m),
\]
where $\Lambda(S)$ is the exterior algebra generated by a set $S$.
For $I=\{i_1<i_2<\dots <i_k\}\subset \{1,\dots m\}$, we denote $t_I=t_{i_1}t_{i_2}\dots t_{i_k}$.
Let $\omega$ be the evaluation map 
\[
  \omega\colon \map_*(B\Z^m,BG)_0\times B\Z^m\rightarrow BG.
\]
For $z_i\in H^*(BG)$ and $I\subset \{1,\dots m\}$, there exists $z_{i,I}\in H^*(\map_*(B\Z^m,BG)_0)$ such that
\[
  \omega^*(z_i)=\sum_{\emptyset\ne I\subset \{1,\dots m\}}z_{i,I}\times t_I.
\]
By degree reason $z_{i,I}=0$ for $|z_i|<|I|$.
The rational cohomology of $\map_*(B\Z^m,BG)$ is computed as follows.
\begin{proposition}\label{generator_map_BG}
  The cohomology $H^*(\map_*(B\Z^m,BG))$ is isomorphic to the free commutative graded algebra generated by 
  \[
    \{z_{i,I}\mid1\leq i\leq r, \emptyset\ne I\subset \{1,\dots m\}, |z_{i}|>|I|\}.
  \]
\end{proposition}
We define bidegree of the generators of $H^*(\map_*(B\Z^m,BG))$, $z_{i,I}$, as 
\[
  \deg(z_{i,I})=\begin{cases}
    (|z_i|-2|I|,|I|)\quad &(|z_i|\geq 2|I|)\\
    (0,|z_i|-|I|)\quad &(|z_i|< |I|),
  \end{cases}
\]
and bigraded algebra structure of $H^*(\map_*(B\Z^m,BG))$ is induced by this bidegree.
The bidegree of $z_{k,I}$ for $|z_i|< |I|$ is not essential, because such classes are in the kernel of $\Theta$.(ref. Proposition \ref{theta_bigraded})

\subsection{The map $\Theta$}
In this subsection we prove that the map 
\[
  \Theta\colon \Hom(\Z^m,G)_0\rightarrow \map_*(B\Z^m,BG)_0
\]
preserves bigraded algebra structure.
To prove this, the following map 
\[
  \widehat{\Phi}\colon G/T\times \map_*(B\Z^m,BT)_0\rightarrow \map_*(B\Z^m,BG)_0
\]
has an important role, and to define this map a construction of the classifying space is needed.

We use the Milnor construction \cite{M} as a construction for the classifying space.
Let $EG$ be the colimit of the sequence of $i$-join of $G$,
\[
  EG=\lim_{i\rightarrow \infty}G*G*\dots *G.
\]
As in \cite{M}, a point of $EG$ is denoted by
\[
  t_1g_1\oplus t_2g_2 \oplus\dots 
\]
such that $g_i \in G$, $t_i\geq 0$ with except for finite $i$ $t_i=0$, $\sum_{i\geq 1} t_i=1$ and 
\[
  t_1g_1\oplus t_2g_2 \oplus\dots=s_1h_1\oplus s_2h_2 \oplus\dots 
\]
if and only if for each $i$ one of the following condition holds;
\begin{enumerate}
  \item $t_i=s_i=0$.
  \item $t_i=s_i$ and $g_i=h_i$.
\end{enumerate}
$G$ acts on $EG$ by 
\[
  (t_1g_1\oplus t_2g_2 \oplus\dots)\cdot g = t_1g_1g \oplus t_2g_2g \oplus\dots
\]
for $g\in G$.
This action is free and the quotient space by this action is the classifying space of $G$, $BG=EG/G$.
The map $\iota\colon BT\rightarrow EG/T$ induced by the inclusion $ET\rightarrow EG$ is homotopy equivalence.
Now we fix a homotopy inverse of $\iota$, $\iota^{-1}$, and define the map
\[
  \widehat{\phi} \colon G/T\times BT \rightarrow EG/T \xrightarrow{\iota^{-1}} BT
\]
as $\widehat{\phi}((g,t_1g_1\oplus t_2g_2 \oplus \dots))=\iota^{-1}(t_1gg_1\oplus t_2gg_2 \oplus\dots)$ for $g \in G$.
And we define the map $\phi$ as the composition of the following maps
\[
  \phi \colon G/T\times BT \rightarrow EG/T \rightarrow BG,
\]
where the first map is same as the first map of the definition in $\widehat{\phi}$, and the second map is the quotient map.
Since $\iota^{-1}$ is a homotopy equivalence, there is a homotopy commutative diagram
\[
    \xymatrix{
      G/T\times BT \ar^{\widehat{\phi}}[r]   \ar_{\phi}[rd] & BT \ar^{\Theta}[d]\\
      & BG.
    }
  \]
This $\phi$ induces the map 
\[
  \widehat{\Phi}\colon G/T\times \map_*(B\pi,BT)_0\rightarrow \map_*(B\pi,BG)_0.
\]
Since $T$ is commutative, $\Hom(\Z^m,T)_0 = T^m$ in the same way as $\Hom(\Z^m,G)_0 \subset G^m$.
The following lemma holds.
\begin{lemma}[Lemma 2.4 in \cite{KTT}]\label{commutative_theta_phi}
  There is a commutative diagram
  \[
    \xymatrix{
      G/T\times T^m \ar^{\Phi}[r]   \ar_{1\times \Theta}[d] & \Hom(\Z^m,G)_0 \ar^{\Theta}[d]\\
      G/T\times \map_*(B\Z^m,BT)_0 \ar^-{\widehat{\Phi}}[r] & \map_*(B\Z^m,BG)_0.
    }
  \]
\end{lemma}

Next, we review some fact about the cohomology. 
The cohomology of the classifying space of the maximal torus, $BT$, is polynomial ring generated by $r$ degree $2$ generators.
So we denote 
\[
  H^*(BT)\cong \Q[x_1,\dots x_r].
\]
We take polynomials $f_1(x_1,\dots x_r),\dots f_r(x_1,\dots x_r)\in \Q[x_1,\dots x_r]$ as an image of $z_1,\dots z_r$ by the inclusion $BT\rightarrow BG$, then there is an isomorphism
\[
  H^*(BG)\cong \Q[f_1(x_1,\dots x_r),\dots f_r(x_1,\dots x_r)].
\]
Now the cohomology of $G/T$ is
\[
  H^*(G/T)\cong \Q[x_1,\dots x_r]/(f_1(x_1,\dots x_r),\dots f_r(x_1,\dots x_r)).
\]
The following lemma holds.
\begin{lemma}[Lemma 3.2 \cite{KTT}]\label{phi_hat_cohomology}
  For each $x_i\in H^2(BT)$,
  \[
    \widehat{\phi}(x_i)=x_i \times 1 + 1\times x_i \in H^*(G/T) \otimes H^*(BT)
  \]
\end{lemma}

We will compute the evaluation map $\omega\colon \map_*(B\Z^m,BT)_0 \times B\Z^m\rightarrow BT$ in rational cohomology.
We denote $H^*(T^m)=H^*(\Hom(\Z^m,T))=\Lambda_{1\leq j \leq m}(y_{1}^{j},\dots ,y_{r}^{j})$.
By Lemma \ref{S^1_hom_eq}, $H^*(\map_*(B\Z^m,BT)_0)$ can also be written as $\Lambda_{1\leq j \leq m}(y_{1}^{j},\dots ,y_{r}^{j})$.
The following lemma holds.
\begin{lemma}[Lemma 3.5 in \cite{KTT}]\label{eval_torus}
  For $x_i \in H^*(BT)$, there is an equation
  \[
    \omega^*(x_i)=y_i^1\times t_1+\dots +y_i^m\times t_m.
  \]
\end{lemma}

Now we prove the following proposition.
\begin{proposition}\label{theta_bigraded}
  The map $\Theta^*\colon H^*(\map_*(B\Z^m,BG)_0)\rightarrow H^*(\Hom(\Z^m,G)_0)$ is a bigraded algebra homomorphism, and $z_{i,I}$ is in the kernel of $\Theta$ for $|z_i|<2|I|$.
\end{proposition}
\begin{proof}
  By Lemma \ref{commutative_theta_phi}, there is a commutative diagram
  \[
    \xymatrix{
      G/T\times T^m \ar^{\Phi}[r]   \ar_{1\times \Theta}[d] & \Hom(\Z^m,G)_0 \ar^{\Theta}[d]\\
      G/T\times \map_*(B\Z^m,BT)_0 \ar^-{\widehat{\Phi}}[r] & \map_*(B\Z^m,BG)_0.
    }
  \]
  By Lemma \ref{S^1_hom_eq}, the left map induces isomorphism in cohomology.
  The bidegree on $H^*(G/T\times \map_*(B\Z^m,BT)_0)$ can be defined by this isomorphism.
  By the definition of the bigraded algebra structure of $\Hom(\Z^m,G)_0$, $\Phi^*$ is an inclusion of bigraded algebras.
  Therefore if the bottom map $\widehat{\Phi}$ preserves the bigraded algebra structure, the left map $\Theta$ also preserves this.

  There is a commuting diagram
  \[
    \xymatrix{
      G/T\times \map_*(B\Z^m,BT)_0 \times B\Z^m \ar^-{\widehat{\Phi}\times 1}[r] \ar_{1\times \omega}[d] & \map_*(B\Z^m,BG)\times B\Z^m \ar^{\omega}[d]\\
      G/T \times BT \ar^{\phi}[r] &BG.\\
    }
  \]
  By Lemma \ref{phi_hat_cohomology} and \ref{eval_torus}, we obtain 
  \begin{align*}
    &(1\times \omega)^*\circ \phi^*(z_i)\\
    =&(1\times \omega)^*\circ \widehat{\phi}^*(f_i(x_1,\dots ,x_r))\\
    =&(1\times \omega)^*(f_i(x_1\times 1+1\times x_1,\dots ))\\
    =&f_i(x_1\times 1\times 1+1\times y_1^1\times t_1+\dots +1\times y_1^m\times t_m,\dots ),\\
  \end{align*}
  Since 
  \[
    x_1\times 1\times 1\in H^2(G/T)\otimes H^0(map_*(B\Z^m,BT)_0) \otimes H^0(B\Z^m)
  \] 
  and 
  \[
    1\times y_1^1\times t_1\in H^0(G/T)\otimes H^1(map_*(B\Z^m,BT)_0) \otimes H^1(B\Z^m),
  \]
  each term of $(1\times \omega)^*\circ \widehat{\phi}^*(z_i)$ satisfies that the degree of $H^*(\map_*(B\Z^m,BT)_0)$ component is same as of $H^*(B\Z^m)$ component, in other words $(1\times \omega)^*\circ \widehat{\phi}^*(z_i)$ is in 
  \[
    \bigoplus_{i,j} H^i(G/T)\otimes H^j(map_*(B\Z^m,BT)_0) \otimes H^j(B\Z^m).
  \]
  Since
  \begin{align*}
    &(\widehat{\Phi}\times 1)^*\circ \omega^*(z_i)\\
    =&(\widehat{\Phi}\times 1)^*\left(\sum_{\emptyset\ne I\subset \{1,\dots m\}}z_{i,I}\times t_I\right)\\
    =&\sum_{\emptyset\ne I\subset \{1,\dots m\}}\widehat{\Phi}(z_{i,I})\times t_I,\\
  \end{align*}
  and $t_I\in H^0(G/T)\otimes H^0(map_*(B\Z^m,BT)_0) \otimes H^{|I|}(B\Z^m)$, by the commuting diagram $\widehat{\Phi}(z_{i,I})$ is in $H^{|z_{i,I}|-2|I|}(G/T)\otimes H^{|I|}(T^m)$.

  For $|z_i|\geq 2|I|$ the bidegree of $\widehat{\Phi}(z_{i,I})$ is $(|z_i|-2|I|,|I|)$ in $H^{*}(G/T)\otimes H^{*}(T^m)$, and it is same as the bidegree of $z_{i,I}$.
  For $|z_i|<2|I|$, $\widehat{\Phi}(z_{i,I})$ must be $0$, because there is no $(|z_i|-2|I|,|I|)$ degree part in $H^{*}(G/T)\otimes H^{*}(T^m)$.
  So $\widehat{\Phi}$ preserves bidegree for all generators in $H^*(map_*(B\Z^m,BG)_0)$, and we obtain that $\theta$ is a bigraded algebra homomorphism.
  We also obtain $\Theta^*(z_{i,I})=0$ for $|z_i|<2|I|$ by the equation $\widehat{\Phi}(z_{i,I})=0$.
  By combining them, we obtain this theorem.
\end{proof}

Let $\mathcal{H}(G,m)=\langle z(i+1-|I|,I)\mid |z_i|<2|I|\rangle$ be the sub-bigraded algebra of $H^*(\map_*(B\Z^m,BG)_0)$ generated by $z_{i,I}$ with $|z_i|<2|I|$.

\begin{corollary}\label{surjective_condition}
  
  Let 
  \begin{align*}
    P(\Hom(\Z^m,G)_0,s,t)&=\sum a_{i,j}s^it^j\\
    P(\mathcal{H}(G,m),s,t)&=\sum b_{i,j}s^it^j.
  \end{align*}
  If $\Theta^*\colon H^*(\map_*(B\Z^m,BG)_0)\rightarrow H^*(\Hom(\Z^m,G)_0)$ is a surjection, then $a_{i,j}\leq b_{i,j}$ for all $i,j$.
\end{corollary}
\begin{proof}
  By Proposition \ref{generator_map_BG} and \ref{theta_bigraded}, the map $\Theta^*$ factors as follows
  \[
    H^*(\map_*(B\Z^m,BG)_0)\rightarrow \mathcal{H}(G,m) \rightarrow H^*(\Hom(\Z^m,G)_0).
  \]
  If $\Theta^*$ is surjective, the second map must be surjection.
  The surjectivity of second map induces the inequation in the statement.
\end{proof}

\section{$F_4,E_6,E_7$ and $E_8$ case}\label{sec:exceptional_cases}
We will prove Theorem \ref{exceptional_non_surjective} by using the computation of Hilbert-Poincar\'e series of $\Hom(\Z^2,G)_1$ in Section \ref{sec:computation_of_Poincare}.
\begin{lemma}\label{Poincare_hom_key_term}
  About the Hilbert-Poincar\'e series of $\mathcal{H}(G,3)$, the followings hold.
  \begin{enumerate}
    \item The coefficient of $s^{18}t^3$ in $P(\mathcal{H}(F_4,3),s,t)$ is $19$.
    \item The coefficient of $s^{10}t^4$ in $P(\mathcal{H}(E_6,3),s,t)$ is $36$.
    \item The coefficient of $s^{6}t^3$ in $P(\mathcal{H}(E_7,3),s,t)$ is $1$.
    \item The coefficient of $s^{30}t^3$ in $P(\mathcal{H}(E_8,3),s,t)$ is $19$.
  \end{enumerate}
\end{lemma}
\begin{proof}
  $H^*(BF_4)$ is generated by $4$ generators $z_1,z_2,z_3,z_4$ with $|z_1|=4,|z_2|=12,|z_3|=16,|z_4|=20$.
  Since $\mathcal{H}(F_4,3)$ is free graded algebra generated by 
  \begin{align*}
    \{&z_{1,\{1\}},z_{1,\{2\}},z_{1,\{3\}},z_{2,\{1\}},z_{2,\{2\}},z_{2,\{3\}},z_{3,\{1\}},z_{3,\{2\}},z_{3,\{3\}},z_{4,\{1\}},z_{4,\{2\}},z_{4,\{3\}},\\
    &z_{1,\{1,2\}},z_{1,\{1,3\}},z_{1,\{2,3\}},z_{2,\{1,2\}},z_{2,\{1,3\}},z_{2,\{2,3\}},z_{3,\{1,2\}},z_{3,\{1,3\}},z_{3,\{2,3\}},z_{4,\{1,2\}},\\
    &z_{4,\{1,3\}},z_{4,\{2,3\}},z_{2,\{1,2,3\}},z_{3,\{1,2,3\}},z_{4,\{1,2,3\}}\},
  \end{align*}
  we obtain 
  \begin{align*}
    &P(\mathcal{H}(F_4,3),s,t)\\
    =&(1+s^2t)^3(1+s^{10}t)^3(1+s^{14}t)^3(1+s^{22}t)^3\\
    \times&\left(\frac{1}{1-t^2}\right)^3\left(\frac{1}{1-s^{8}t^2}\right)^3\left(\frac{1}{1-s^{12}t^2}\right)^3\left(\frac{1}{1-s^{20}t^2}\right)^3\\
    \times&(1+s^{6}t^3)(1+s^{10}t^3)(1+s^{18}t^3).
  \end{align*}
  Thus the coefficient of $s^{18}t^3$ is $19$.

  By similar calculation, we obtain the other cases.
\end{proof}

On the other hands, by looking at the underlined part in the Hilbert-Poincar\'e series of $\Hom(\Z^3,G)_0$ in Section \ref{sec:computation_of_Poincare}, we obtain the following lemma.
\begin{lemma}\label{Poincare_map_key_term}
  About the Hilbert-Poincar\'e series of $\Hom(\Z^3,G)_1$, the followings hold.
  \begin{enumerate}
    \item The coefficient of $s^{18}t^3$ in $P(\Hom(\Z^3,F_4)_1,s,t)$ is $20$.
    \item The coefficient of $s^{10}t^4$ in $P(\Hom(\Z^3,E_6)_1,s,t)$ is $39$.
    \item The coefficient of $s^{6}t^3$ in $P(\Hom(\Z^3,E_7)_1,s,t)$ is $2$.
    \item The coefficient of $s^{30}t^3$ in $P(\Hom(\Z^3,E_8)_1,s,t)$ is $20$.
  \end{enumerate}
\end{lemma}

\begin{proof}[Proof of Theorem \ref{exceptional_non_surjective}]
  By combining the Corollary \ref{surjective_condition} and Lemma \ref{Poincare_hom_key_term} and \ref{Poincare_map_key_term}, we obtain this.
\end{proof}


\section{$G_2$ case}\label{sec:G_2_case}
In this section we consider the case $G=G_2$.
To compute $H^*(\Hom(\Z^m,G_2))$, we will use notation of $SU(3)\subset G_2$.
Since the rank of $SU(3)$ and $G_2$ are two, a maximal torus of $SU(3)$ is a maximal torus of $G_2$.
We take a maximal torus of $SU(3)$, $T$, as a restriction of a maximal torus of $U(3)$.
Then the cohomology of the maximal torus of $SU(3)$ and its classifying space can be defined as follows
\begin{align*}
  H^*(T^m)&\cong \bigotimes_{1\leq j\leq m} \Lambda(\alpha_j, \beta_j, \gamma_j)/K_j\\
  H^*(BT)&\cong \Q[x,y,w]/(x+y+w),
\end{align*}
where $K_j=(\alpha_j+\beta_j+\gamma_j)$.
In previous section the generators of $H^*(T^m)$ and $H^*(BT)$ are denoted by $y_j^i$ and $x_i$, but after here we will use these notations for visibility.

The Weyl group of $G_2$ is the dihedral group $D_6=\langle a,b \mid a^6=b^2=abab=1 \rangle$.
The subgroup generated by $a^2,b$ is the Weyl groups of $SU(3)$, $\mathfrak{S}_3$, and $D_6$ is generated by $\mathfrak{S}_3$ and $a^3$.
The action of $a^3$ on $H^1(T^m)$ and $H^2(BT)$ are the map that multiplies by $-1$.

Let $H^*(BT)=\Q[x,y,w]/(x+y+w)$ and we regard $H^*(BG_2)=H^*(BT)^{D_6}\subset H^*(BT)$ by the inclusion $BT\rightarrow BG_2$.
It is well known that $H^*(BG_2)$ is a polynomial ring, and by comparing $H^*(BSU(3))$ it is easy to see that $x^2+y^2+w^2$ and $x^2y^2w^2$ are the generators of $H^*(BG_2)$.
Since there is an equation 
\[
  x^6+y^6+w^6=3x^2y^2w^2+(x^2+y^2+w^2)Q(x,y,w)
\]
for a polynomial $Q(x,y,w)$, the following lemma holds.
\begin{lemma}\label{BG_2_generator}
  Let $z_1=x^2+y^2+w^2,z_2=x^6+y^6+w^6$.
  Then $H^*(BG_2)$ is the polynomial ring generated by $z_1,z_2$.
\end{lemma}

Let 
\[
  {\mathcal{K}(m)}=\Q[x,y,w]/J\otimes \left(\bigotimes_{1\leq j\leq m} \Lambda(\alpha_j, \beta_j, \gamma_j)/K_j\right),
\]
where $J=(x+y+w,x^2+y^2+w^2,x^6+y^6+w^6)$.
Since there is an isomorphism 
  \[
    H^*(G_2/T)\cong \Q[x,y,w]/J,
  \]
by Theorem \ref{Thm_Baird} the following statement holds.
\begin{lemma}
  There is a isomorphism
  \[
    H^*(\Hom(\Z^m,G_2))\cong{\mathcal{K}(m)}^{D_6}.
  \]
\end{lemma}

For $d>0$, let 
\[
  z(d+1,I)=x^{d}\alpha_I+y^{d}\beta_I+w^{d}\gamma_I\in {\mathcal{K}(m)}.
\]
By the action of $D_6$, $z(d+1,I)\in {\mathcal{K}(m)}^{D_6}$ if and only if $d+|I|$ is odd.
To prove Theorem \ref{G_2_surjective} we need the next theorem.
\begin{theorem}\label{G_2_generator}
  $H^*(\Hom(\Z^m,G_2))$ is generated by 
  \[
    \{z(3-|I|,I)\mid I\subset \{1,\dots m\},|I|\leq 2\}\cup \{z(7-|I|,I)\mid I\subset \{1,\dots m\},|I|\leq 6\}
  \]
\end{theorem}
To prove this theorem, much calculation is needed.
Thus we will prove this theorem in the next section, and in this section we will prove Theorem \ref{G_2_surjective} by assuming this theorem.

\begin{proof}[Proof of Theorem \ref{G_2_surjective}]
  By Lemma \ref{BG_2_generator} $z_1=x^2+y^2+w^2$ and $z_2=x^6+y^6+w^6$ are generators of $H^*(BG_2)$.
  By Lemma \ref{phi_hat_cohomology},
  \begin{align*}
    &(1\times \omega)^*\phi^*(z_2)\\
    =&(1\times \omega)^*((x\times 1+1\times x)^6+(y\times 1+1\times y)^6+(w\times 1+1\times w)^6)\\
    =&\left(x\times 1+1\times \left(\sum_{i=1}^{m}\alpha_i \times t_i\right)\right)^6+\left(y\times 1+1\times \left(\sum_{i=1}^{m}\beta_i \times t_i\right)\right)^6\\
    +&\left(w\times 1+1\times \left(\sum_{i=1}^{m}\gamma_i \times t_i\right)\right)^6\\
    =&\sum_{I\subset \{1,\dots ,m\}}\frac{6!}{(6-|I|)!}(x^{6-|I|}\alpha_I+y^{6-|I|}\beta_I+w^{6-|I|}\gamma_I)t_{I}\\
    =&\sum_{I\subset \{1,\dots ,m\}}\frac{6!}{(6-|I|)!}z(7-|I|,I)t_{I}
  \end{align*}
  and similarly
  \begin{align*}
    &(1\times \omega)^*\phi^*(z_1)\\
    =&\sum_{I\subset \{1,\dots ,m\}}\frac{2!}{(2-|I|)!}z(3-|I|,I)t_{I}.
  \end{align*}
  By the commutative diagram
  \[
    \xymatrix{
      G_2/T\times \map_*(B\Z^m,BT)_0 \times B\Z^m \ar^-{\widehat{\Phi}\times 1}[r] \ar_{1\times \omega}[d] & \map_*(B\Z^m,BG_2)\times B\Z^m \ar^{\omega}[d]\\
      G_2/T \times BT \ar^{\phi}[r] &BG_2,\\
    }
  \]
  we obtain
  \begin{align*}
    \widehat{\Phi}^*(z_{1,I})=&\frac{2!}{(2-|I|)!}z(3-|I|,I)\\
    \widehat{\Phi}^*(z_{2,I})=&\frac{6!}{(6-|I|)!}z(7-|I|,I).\\
  \end{align*}
  By Theorem \ref{G_2_generator} and the following commuting diagram,
  \[
    \xymatrix{
      G_2/T\times T^m \ar^{\Phi}[r] \ar_{1\times \Theta}[d]  & \Hom(\Z^m,G_2)_0 \ar^{\Theta}[d]\\
      G_2/T\times \map_*(B\Z^m,BT)_0 \ar^-{\widehat{\Phi}}[r] &  \map_*(B\Z^m,BG_2)_0,
    }
  \]
  $\Theta^*(z_{1,I}),\Theta^*(z_{2,I})$ correspond to the generators of $H^*(\Hom(\Z^m,G_2)_0)$.
  Complete the proof.
\end{proof}

\section{Proof of Theorem \ref{G_2_generator}}\label{sec:G_2_generator}

In this section we will prove Theorem \ref{G_2_generator}.
To prove this we will consider the map ${\mathcal{K}(m)}^{D_6}\rightarrow {\mathcal{K}(m)}^{\mathfrak{S}_3}$.
\begin{lemma}\label{S_3_generator}
  ${\mathcal{K}(m)}^{\mathfrak{S}_3}$ is generated by 
  \[
    \{z(i,I)\mid  1\leq i\leq 6,I\subset \{1,\dots m\},|I|\leq 3\}.
  \]
\end{lemma}
The following proof is essentially based on Section 6 in \cite{KT1}.
This lemma can be shown in a more general case.
\begin{proof}
  At first we prove that the set
  \[
    \overline{S}=\{z(i,I)\mid 1\leq i\leq 6,I\subset \{1,\dots m\}\},
  \]
  generates ${\mathcal{K}(m)}^{\mathfrak{S}_3}$.
  Now $x^6=y^6=w^6=0$ in $\Q[x,y,w]/J$, $z(i,I)$ is $0$ for $i\geq 7$ in $H^*(\Hom(\Z^m,G_2))$.
  Let 
  \[
    \rho\colon {\mathcal{K}(m)}\rightarrow {\mathcal{K}(m)}^{\mathfrak{S}_3}
  \]
  be the linear map defined by $\rho(z)=\sum_{\sigma\in \mathfrak{S}_3} \sigma(z)$ for $z\in {\mathcal{K}(m)}$.
  Since $\rho$ restricted to ${\mathcal{K}(m)}^{\mathfrak{S}_3}$ is $|\mathcal{S}_3|$-fold map, $\rho$ is surjection. 
  So ${\mathcal{K}(m)}^{\mathfrak{S}_3}$ is spanned by the image of monic monomials $\rho(x^iy^jw^k\alpha_I\beta_J\gamma_K)$ for $i,j,k\geq 0,I,J,K\subset \{1,\dots m\}$.
  It is enough to prove that $\rho(x^iy^jw^k\alpha_I\beta_J\gamma_K)$ is in the subring generated by $\overline{S}$.

  When the monomial is $x^i\alpha_I$, $y^i\beta_I$, or $w^i\gamma_I$, 
  \[
    \rho(x^i\alpha_I)=\rho(y^i\beta_I)=\rho(w^i\gamma_I)=2z(i+1,I).
  \]
  Thus these monomials are in the subring generated by $\overline{S}$.
  
  When the monomial is $x^iy^j\alpha_I\beta_J$ for $i+|I|\ne 0\ne j+|J|$,
  \[
    \rho(x^iy^j\alpha_I\beta_J)=z(i+1,I)z(j+1,J)-(x^{i+j}\alpha_I\alpha_J+y^{i+j}\beta_I\beta_J+w^{i+j}\gamma_I\gamma_J).
  \]
  The last term $x^{i+j}\alpha_I\alpha_J+y^{i+j}\beta_I\beta_J+w^{i+j}\gamma_I\gamma_J$ is $0$ when $I\cap J\ne \emptyset$ and $\pm z(i+j+1,I\cup J)$ for $I\cap J= \emptyset$.
  Thus $\rho(x^iy^j\alpha_I\beta_J)$ is in the subring generated by $\overline{S}$.
  We can prove $\rho(x^iw^j\alpha_I\gamma_J)$ or $\rho(y^iw^j\beta_I\gamma_J)$ are in the subring generated by $\overline{S}$ by same way.  

  When the monomial is $x^iy^jw^k\alpha_I\beta_J\gamma_K$ for $i+|I|\ne 0,j+|J|\ne 0$ and $k+|K|\ne 0$, 
  \[
    \rho(x^iy^jw^k\alpha_I\beta_J\gamma_K)-z(i+1,I)z(j+1,J)z(k+1,K)
  \]
  is the linear combination of $\rho(x^{i+j}y^k\alpha_{I\cup J}\beta_K)$, $\rho(x^{i+k}y^j\alpha_{I\cup K}\beta_J)$, $\rho(x^iy^{j+k}\alpha_I\beta_{J\cup K})$ and $\rho(x^{i+j+k}\alpha_{I \cup J\cup K})$.
  By previous arguments these are in the subring generated by $\overline{S}$.
  Thus $\rho(x^iy^jw^k\alpha_I\beta_J\gamma_K)$ is also in the subring generated by $\overline{S}$.
  By combining these, we obtain that ${\mathcal{K}(m)}^{\mathfrak{S}_3}$ is generated by $\overline{S}$.

  Next, we prove that $z(i+1,I)$ with $|I|\geq 4$ is in the ring generated by
  \[
    S=\{z(i,I)\mid  1\leq i\leq 6,I\subset \{1,\dots m\},|I|\leq 3\}.
  \]
  For $1\leq i_1<i_2<\dots i_n\leq m$ Let
  \begin{align*}
    &e(k,\{i_1\})=x^k\alpha_{i_1}+y^k\beta_{i_1}+z^k\gamma_{i_1}\\
    &e(k,\{i_1,i_2\})=x^k\alpha_{i_1}(\beta_{i_2}+\gamma_{i_2})+y^k\beta_{i_1}(\alpha_{i_2}+\gamma_{i_2})+z^k\gamma_{i_1}(\alpha_{i_2}+\beta_{i_2})\\
    &e(k,\{i_1,i_2,i_3\})\\
    =&x^k\alpha_{i_1}(\beta_{i_2}\gamma_{i_3}+\gamma_{i_2}\beta_{i_3})+y^k\beta_{i_1}(\alpha_{i_2}\gamma_{i_3}+\gamma_{i_2}\alpha_{i_3})+z^k\gamma_{i_1}(\alpha_{i_2}\beta_{i_3}+\beta_{i_2}\alpha_{i_3}),
  \end{align*}
  and for $|I|\geq 4$ $e(k,I)=0$.
  By Lemma 6.13 in \cite{KT1}, for $I =\{i_1<i_2< \dots <i_n\}$ there are equations
  \begin{align*}
    &\sum_{l=0}^{n-2}\sum_{\substack{i_1\in J\subset I\\|J|=l+1}}(-1)^{l+d(J)}l!z(k,J)e(0,I-J)+(-1)^{n+1}(n-1)!z(k,I)\\
    &=
    \begin{cases}
      e(k-1,I)&(n\le 3)\\
      0&(n>3),
    \end{cases}
  \end{align*}
  where $d(J)=\left(\sum_{i_j\in J}j\right)-\frac{|J|(|J|-1)}{2}$.
  By these equations, we obtain that $e(d,I)$ is in the subring generated by $S$, and $z(i,I)$ for $|I|>3$ is also in this subring.
  By combining these argument, we obtain this lemma.
\end{proof}

\begin{lemma}\label{D6_generator}
  ${\mathcal{K}(m)}^{D_6}$ is generated by the union of
  \[
    \{z(i,I)\mid 1\leq i\leq 6,I\subset \{1,\dots m\},i+|I|\text{ is even},|I|\leq 3\}
  \]
  and 
  \[
    \{z(i,I)z(j,J)\mid 1\leq i,j\leq 6,I,J\subset \{1,\dots m\},|I|,|J|\leq 3,(i+|I|)(j+|I|)\text{ is odd}\}.
  \]
\end{lemma}
\begin{proof}
  $D_6$ is generated by $\mathfrak{S}_3\subset D_6$ and an order 2 element $a^3\in D_6$.
  Since the action of $a^3$ is the map that multiplies by $-1$ on the generators of $\mathcal{K}(m)$ and 
  \[
    a^3(z(i+1,I))=
    \begin{cases}
      z(i+1,I)\quad &(i+|I|\text{ is even})\\
      -z(i+1,I)\quad &(i+|I|\text{ is odd}),\\
    \end{cases}
  \] 
  by Lemma \ref{S_3_generator}, we obtain this lemma.
\end{proof}

Let $\mathcal{L}(m)$ be the subring of $\mathcal{K}(m)^{D_6}$ generated by 
\[
  \{z(3-|I|,I)\mid I\subset \{1,\dots m\},|I|\leq 2\}\cup \{z(7-|I|,I)\mid I\subset \{1,\dots m\},|I|\leq 6\}.
\]
To prove Theorem \ref{G_2_generator} we will show that the generators of ${\mathcal{K}(m)}^{D_6}$ in Lemma \ref{D6_generator} are in $\mathcal{L}(m)$.
Since each bijection $\{1,\dots m\}\rightarrow \{1,\dots m\}$ induces a $D_6$-equivariant isomorphism $\mathcal{K}(m)\rightarrow \mathcal{K}(m)$, it is enough to consider $z(i,\{1,\dots,|I|\})$ and $z(i,\{1,\dots |I|\})z(j,\{|I|-|I\cap J|+1,\dots,|I|+|J|-|I\cap J|\})$ instead of $z(i,I)$ and $z(i,I)z(j,J)$.
We have to prove that the set
\[
    \{z(2,\{1,2,3\}),z(3,\{1,2\}),z(4,\{1\}),z(4,\{1,2,3\}),z(6,\{1,2,3\})\}
\]
are in $\mathcal{L}(m)$ for suitable $m$, and the product of two of 
\begin{align*}
  &\{z(1,I),z(3,I),z(5,I)\mid |I|=3\}\cup\{z(2,I),z(4,I),z(6,I)\mid |I|=2\}\\
  \cup&\{z(3,I),z(5,I)\mid |I|=1\}\cup\{z(4,\emptyset),z(6,\emptyset)\}
\end{align*}
are in $\mathcal{L}(m)$ for suitable $m$.
Remark that $z(1,\{1,2\}),z(2,\{1\}),z(5,\{1,2\}),z(6,\{1\})$ are in $\mathcal{L}(m)$ as generators, and by definition of $\mathcal{K}(m)$ $z(1,\{1\}),z(2,\emptyset),z(3,\emptyset),z(5,\emptyset)$ are $0$ in $\mathcal{K}(m)$.
So we don't have to consider the cases related to these elements.

These are too much to check all of them directly.
To reduce the amount of calculation, we use some properties.
At first, we compute the relation between them.
\begin{lemma}\label{relation_between_z_G2}
  There are equations
  \begin{align*}
    9z&(3,\{1,2,3\})\\
    &=z(2,\{1,2\})z(2,\{3\})-z(2,\{1,3\})z(2,\{2\})+z(2,\{2,3\})z(2,\{1\})\\
    &+z(1,\{1,2\})z(3,\{3\})-z(1,\{1,3\})z(3,\{2\})+z(1,\{2,3\})z(3,\{1\}),\\
    9z&(5,\{1,2,3\})\\
    &=z(4,\{1,2\})z(2,\{3\})-z(4,\{1,3\})z(2,\{2\})+z(4,\{2,3\})z(2,\{1\})\\
    &+z(3,\{1,2\})z(3,\{3\})-z(3,\{1,3\})z(3,\{2\})+z(3,\{2,3\})z(3,\{1\}),\\
    3z&(4,\{1,2\})=z(3,\{1\})z(2,\{2\})+z(2,\{1\})z(3,\{2\}), \\
    3z&(6,\{1,2\})=z(5,\{1\})z(2,\{2\})+z(4,\{1\})z(3,\{2\}), \\
    z(&6,\emptyset)=0\\
  \end{align*}
\end{lemma}
\begin{proof}
  In $\Q[x,y,w]/J$, there is an equation 
  \[
    x^2+xy+y^2=y^2+yw+w^2=w^2+wx+x^2=0.
  \]
  There is an equation
  \begin{align*}
    &z(2,\{1,2\})z(2,\{3\})-z(2,\{1,3\})z(2,\{2\})+z(2,\{2,3\})z(2,\{1\})\\
    +&z(1,\{1,2\})z(3,\{3\})-z(1,\{1,3\})z(3,\{2\})+z(1,\{2,3\})z(3,\{1\})\\
    =&6x^2\alpha_1 \alpha_2 \alpha_3+ 6y^2\beta_1 \beta_2 \beta_3 +6w^2\gamma_1 \gamma_2 \gamma_3\\
    +& (xy +y^2) \alpha_1 \alpha_2 \beta_3 + (xw +w^2) \alpha_1 \alpha_2 \gamma_3\\
    +& (yx +x^2) \beta_1 \beta_2 \alpha_3 + (yw +w^2) \beta_1 \beta_2 \gamma_3\\
    +& (wx+x^2) \gamma_1 \gamma_2 \alpha_3 + (wy +y^2) \gamma_1 \gamma_2 \beta_3\\
    +& (xy +y^2) \alpha_1 \beta_2 \alpha_3 + (xw +w^2) \alpha_1 \gamma_2 \alpha_3\\
    +& (yx +x^2) \beta_1 \alpha_2 \beta_3 + (yw +w^2) \beta_1 \gamma_2 \beta_3\\
    +& (wx+x^2) \gamma_1 \alpha_2 \gamma_3 + (wy +y^2) \gamma_1 \beta_2 \gamma_3\\
    +& (xy +y^2) \beta_1 \alpha_2 \alpha_3 + (xw +w^2) \gamma_1 \alpha_2 \alpha_3\\
    +& (yx +x^2) \alpha_1 \beta_2 \beta_3 + (yw +w^2) \gamma_1 \beta_2 \beta_3\\
    +& (wx+x^2) \alpha_1 \gamma_2 \gamma_3 + (wy +y^2) \beta_1 \gamma_2 \gamma_3\\
    =&6x^2\alpha_1 \alpha_2 \alpha_3+ 6y^2\beta_1 \beta_2 \beta_3 +6w^2\gamma_1 \gamma_2 \gamma_3\\
    -& x^2 \alpha_1 \alpha_2 \beta_3 + x^2 \alpha_1 \alpha_2 (\alpha_3 + \beta_3)\\
    -& y^2 \beta_1 \beta_2 \alpha_3 + y^2 \beta_1 \beta_2 (\alpha_3 + \beta_3)\\
    -& w^2 \gamma_1 \gamma_2 \alpha_3 + w^2 \gamma_1 \gamma_2 (\alpha_3 + \gamma_3)\\
    -& x^2 \alpha_1 \beta_2 \alpha_3 + x^2 \alpha_1 (\alpha_3 + \beta_3) \alpha_3\\
    -& y^2 \beta_1 \alpha_2 \beta_3 + y^2 \beta_1 (\alpha_3 + \beta_3) \beta_3\\
    -& w^2 \gamma_1 \alpha_2 \gamma_3 + w^2 \gamma_1 (\alpha_3 + \gamma_3) \gamma_3\\
    -& x^2 \beta_1 \alpha_2 \alpha_3 + x^2 (\alpha_3 + \beta_3) \alpha_2 \alpha_3\\
    -& y^2 \alpha_1 \beta_2 \beta_3 + y^2 (\alpha_3 + \beta_3) \beta_2 \beta_3\\
    -& w^2 \alpha_1 \gamma_2 \gamma_3 + w^2 (\alpha_3 + \gamma_3) \gamma_2 \gamma_3\\
    =&9x^2\alpha_1 \alpha_2 \alpha_3+ 9y^2\beta_1 \beta_2 \beta_3 +9w^2\gamma_1 \gamma_2 \gamma_3\\
    =&9z(3,\{1,2,3\}),
  \end{align*}
  and we obtain the first equation.
  There is an equation
  \begin{align*}
    &z(4,\{1,2\})z(2,\{3\})-z(4,\{1,3\})z(2,\{2\})+z(4,\{2,3\})z(2,\{1\})\\
    +&z(3,\{1,2\})z(3,\{3\})-z(3,\{1,3\})z(3,\{2\})+z(3,\{2,3\})z(3,\{1\})\\
    =&6x^4\alpha_1 \alpha_2 \alpha_3+ 6y^4\beta_1 \beta_2 \beta_3 +6w^4\gamma_1 \gamma_2 \gamma_3\\
    +& (x^3y +x^2y^2) \alpha_1 \alpha_2 \beta_3 + (x^3w +x^2w^2) \alpha_1 \alpha_2 \gamma_3\\
    +& (y^3x +y^2x^2) \beta_1 \beta_2 \alpha_3 + (y^3w +y^2w^2) \beta_1 \beta_2 \gamma_3\\
    +& (w^3x+w^2x^2) \gamma_1 \gamma_2 \alpha_3 + (w^3y +w^2y^2) \gamma_1 \gamma_2 \beta_3\\
    +& (x^3y +x^2y^2) \alpha_1 \beta_2 \alpha_3 + (x^3w +x^2w^2) \alpha_1 \gamma_2 \alpha_3\\
    +& (y^3x +y^2x^2) \beta_1 \alpha_2 \beta_3 + (y^3w +y^2w^2) \beta_1 \gamma_2 \beta_3\\
    +& (w^3x+w^2x^2) \gamma_1 \alpha_2 \gamma_3 + (w^3y +w^2y^2) \gamma_1 \beta_2 \gamma_3\\
    +& (x^3y +x^2y^2) \beta_1 \alpha_2 \alpha_3 + (x^3w +x^2w^2) \gamma_1 \alpha_2 \alpha_3\\
    +& (y^3x +y^2x^2) \alpha_1 \beta_2 \beta_3 + (y^3w +y^2w^2) \gamma_1 \beta_2 \beta_3\\
    +& (w^3x+w^2x^2) \alpha_1 \gamma_2 \gamma_3 + (w^3y +w^2y^2) \beta_1 \gamma_2 \gamma_3\\
    =&6x^4\alpha_1 \alpha_2 \alpha_3+ 6y^4\beta_1 \beta_2 \beta_3 +6w^4\gamma_1 \gamma_2 \gamma_3\\
    -& x^4 \alpha_1 \alpha_2 \beta_3 + x^4 \alpha_1 \alpha_2 (\alpha_3 + \beta_3)\\
    -& y^4 \beta_1 \beta_2 \alpha_3 + y^4 \beta_1 \beta_2 (\alpha_3 + \beta_3)\\
    -& w^4 \gamma_1 \gamma_2 \alpha_3 + w^4 \gamma_1 \gamma_2 (\alpha_3 + \gamma_3)\\
    -& x^4 \alpha_1 \beta_2 \alpha_3 + x^4 \alpha_1 (\alpha_3 + \beta_3) \alpha_3\\
    -& y^4 \beta_1 \alpha_2 \beta_3 + y^4 \beta_1 (\alpha_3 + \beta_3) \beta_3\\
    -& w^4 \gamma_1 \alpha_2 \gamma_3 + w^4 \gamma_1 (\alpha_3 + \gamma_3) \gamma_3\\
    -& x^4 \beta_1 \alpha_2 \alpha_3 + x^4 (\alpha_3 + \beta_3) \alpha_2 \alpha_3\\
    -& y^4 \alpha_1 \beta_2 \beta_3 + y^4 (\alpha_3 + \beta_3) \beta_2 \beta_3\\
    -& w^4 \alpha_1 \gamma_2 \gamma_3 + w^4 (\alpha_3 + \gamma_3) \gamma_2 \gamma_3\\
    =&9x^4\alpha_1 \alpha_2 \alpha_3+ 9y^4\beta_1 \beta_2 \beta_3 +9w^4\gamma_1 \gamma_2 \gamma_3\\
    =&9z(5,\{1,2,3\}),
  \end{align*}
  and we obtain the second equation.
  There is an equation
  \begin{align*}
    &z(3,\{1\})z(2,\{2\})+z(2,\{1\})z(3,\{2\}) \\
    =&2x^3\alpha_1\alpha_2 + 2y^3\beta_1 \beta_2 + 2w^3\gamma_1\gamma_2 \\
    +& (x^2y+xy^2)\alpha_1 \beta_2 + (x^2w+xw^2)\alpha_1 \gamma_2  \\
    +& (x^2y+xy^2)\beta_1 \alpha_2 + (y^2w+yw^2)\beta_1 \gamma_2  \\
    +& (x^2w+xw^2)\gamma_1 \alpha_2 + (y^2w+yw^2)\gamma_1 \beta_2  \\
    =&2x^3\alpha_1\alpha_2 + 2y^3\beta_1 \beta_2 + 2w^3\gamma_1\gamma_2 \\
    -& x^3\alpha_1 \beta_2 + x^3\alpha_1 (\alpha_2 + \beta_2)  \\
    -& y^3\beta_1 \alpha_2 + y^3\beta_1 (\alpha_2+ \beta_2)  \\
    -& w^3\gamma_1 \alpha_2 + w^3\gamma_1 (\alpha_2+ \gamma_2)  \\
    =& 3 x^3\alpha_1\alpha_2 + 3y^3\beta_1 \beta_2 + 3w^3\gamma_1\gamma_2 \\
    =&3 z(4,\{1,2\}),
  \end{align*}
  and we obtain the third equation.
  There is an equation
  \begin{align*}
    &z(4,\{1\})z(1,\{2\})+z(3,\{1\})z(2,\{2\}) \\
    =&2x^5\alpha_1\alpha_2 + 2y^5\beta_1 \beta_2 + 2w^5\gamma_1\gamma_2 \\
    +& (x^4y+x^3y^2)\alpha_1 \beta_2 + (x^4w+x^3w^2)\alpha_1 \gamma_2  \\
    +& (x^4y+x^3y^2)\beta_1 \alpha_2 + (y^4w+y^3w^2)\beta_1 \gamma_2  \\
    +& (x^4w+x^3w^2)\gamma_1 \alpha_2 + (y^4w+y^3w^2)\gamma_1 \beta_2  \\
    =&2x^5\alpha_1\alpha_2 + 2y^5\beta_1 \beta_2 + 2w^5\gamma_1\gamma_2 \\
    -& x^5\alpha_1 \beta_2 + x^5\alpha_1 (\alpha_2 + \beta_2)  \\
    -& y^5\beta_1 \alpha_2 + y^5\beta_1 (\alpha_2+ \beta_2)  \\
    -& w^5\gamma_1 \alpha_2 + w^5\gamma_1 (\alpha_2+ \gamma_2)  \\
    =& 3 x^5\alpha_1\alpha_2 + 3y^5\beta_1 \beta_2 + 3w^5\gamma_1\gamma_2 \\
    =& 3 z(6,\{1,2\}),
  \end{align*}
  and we obtain the forth equation.
  There is an equation
  \begin{align*}
    &z(6,\emptyset) \\
    =&x^5+y^5+w^5 \\
    =& (x^4+y^4+w^4)(x+y+w)-(x^3+y^3+w^3)(x^2+y^2+w^2)\\
    +&(x^2+y^2+w^2)xyw  \\
    =&-\frac{1}{2}(x^3+y^3+w^3)((x+y+w)^2-(x^2+y^2+w^2))\\
    =&0,
  \end{align*}
  and we obtain the fifth equation.
\end{proof}

The following is Theorem 4.5 in \cite{Ta}.
\begin{theorem}\label{Theorem_2G2_generator}
  $\mathcal{K}(2)^{D_6}$ is generated by 
  \[
    \{z(2,\{1\}),z(2,\{2\}),z(6,\{1\}),z(6,\{2\}),z(1,\{1,2\}),z(5,\{1,2\})\}.
  \]
\end{theorem}

By Lemma \ref{relation_between_z_G2}, we don't have to consider the cases related to $\{z(3,I),z(5,I)\mid |I|=3\}\cup\{z(4,I)\mid |I|=2\}\cup\{z(6,\emptyset)\}$.
By Theorem \ref{Theorem_2G2_generator} and the isomorphism $\mathcal{K}(1)^{D_6}\cong H^*(G_2)$, the following elements are in $\mathcal{L}(1)$ or $\mathcal{L}(2)$;
\begin{align*}
  &z(3,\{1,2\}),z(4,\{1\}),z(2,\{1,2\})^2,z(2,\{1,2\})z(3,\{2\}),z(2,\{1,2\})z(4,\emptyset),\\
  &z(2,\{1,2\})z(5,\{2\}),z(3,\{1\})z(3,\{2\}),z(3,\{1\})z(4,\emptyset),z(3,\{1\})z(5,\{1\}),\\
  &z(3,\{1\})z(5,\{2\}),z(4,\emptyset)z(5,\{1\}),z(5,\{1\})z(5,\{2\}).
\end{align*}
Therefore we obtain the next lemma.
\begin{lemma}\label{basis_to_check}
  If the following elements are in $\mathcal{L}(m)$ for suitable $m$;
  \begin{align*}
    &z(2,\{1,2,3\}),z(4,\{1,2,3\}),z(6,\{1,2,3\})\\
    &z(1,\{1,2,3\})z(1,\{2,3,4\}),z(1,\{1,2,3\})z(1,\{3,4,5\}),\\
    &z(1,\{1,2,3\})z(1,\{4,5,6\}),z(1,\{1,2,3\})z(2,\{2,3\}),z(1,\{1,2,3\})z(2,\{3,4\}),\\
    &z(1,\{1,2,3\})z(2,\{4,5\}),z(1,\{1,2,3\})z(3,\{3\}),z(1,\{1,2,3\})z(3,\{4\}),\\
    &z(1,\{1,2,3\})z(4,\emptyset),z(1,\{1,2,3\})z(5,\{3\}),z(1,\{1,2,3\})z(5,\{4\}),\\
    &z(2,\{1,2\})z(2,\{2,3\}),z(2,\{1,2\})z(2,\{3,4\}),z(2,\{1,2\})z(3,\{3\}),\\
    &z(2,\{1,2\})z(5,\{3\}), 
  \end{align*}
  then $\mathcal{L}(m)=\mathcal{K}(m)$ for all $m$.
\end{lemma}
Remark that by the property of commutative graded algebra, the square of an odd graded element in $\mathcal{K}$ is $0$, so $z(1,\{1,2,3\})^2,z(3,\{1\})^2$ are not in this list.
And by degree reason $z(4,\emptyset)^2$ is $0$, $z(4,\emptyset)^2$ is also not in this list.

Because there are still many cases, so we use another property, filtration of ${\mathcal{K}(m)}^{D_6}$.
For $0 \leq i\leq m$ let $\sigma_i\colon \{1,\dots m\}\rightarrow \{1,\dots m+1\}$ be the map defined as 
\[
  \sigma_i(j)=\begin{cases}
    j\quad &(j\leq i)\\
    j+1\quad &(i+1\leq j),
  \end{cases}
\]
and let $\sigma_i\colon\mathcal{K}(m)\rightarrow \mathcal{K}(m+1)$ for $0\leq i \leq m$ be the inclusion defined as $\sigma_i(\alpha_j)=\alpha_{\sigma_i(j)}$.
We denote this map as same symbol $\sigma_i$.
Let 
\[
  F_{n}\mathcal{K}(m)=\bigcup_{i_1,\dots i_{m-n}}\sigma_{i_1}\circ \dots \circ \sigma_{i_{m-n}}(\mathcal{K}(n))\subset \mathcal{K}(m).
\]
Then $F_1\mathcal{K}(m)\subset F_2\mathcal{K}(m)\subset \dots \subset F_m\mathcal{K}(m)=\mathcal{K}(m)$ is a sequence of linear subspaces.
$F_n\mathcal{K}(m)$ is spanned by $\{x^{i}y^{j}w^k\alpha_{I}\beta_J\gamma_K\mid |I\cup J\cup K|\leq n\}$.
Thus we regard $\mathcal{K}(m)/F_{m-1}\mathcal{K}(m)$ the linear subspace spanned by $\{x^{i}y^{j}w^k\alpha_{I}\beta_J\gamma_K\mid |I\cup J\cup K|= m\}$.
Since $\sigma_i$ preserves the $D_6$-action and the bigraded module structure, $F_{n}(\mathcal{K}(m)^{D_6})=(F_{n}\mathcal{K}(m))^{D_6}$ is a bigraded linear subspace of $\mathcal{K}(m)^{D_6}$.
We compute the Hilbert-Poincar\'e series of $\mathcal{K}(m)^{D_6}/F_{m-1}\mathcal{K}(m)^{D_6}$.
\begin{lemma}\label{Poincare_G2_filter}
  There are equations
  \begin{align*}
    &P(\mathcal{K}(3)^{D_6}/F_{2}\mathcal{K}(3)^{D_6};s,t)\\
    =&s^{12}t^6 + 3s^{10}t^5 + 3s^2t^5 + 3s^{12}t^4 + 3s^8t^4 + 3s^4t^4 + 3t^4 + 3s^{10}t^3 + 2s^6t^3 + 3s^2t^3\\
    &P(\mathcal{K}(4)^{D_6}/F_{3}\mathcal{K}(4)^{D_6};s,t)\\
    =&t^8 + 4s^{10}t^7 + 4s^2t^7 + 6s^{12}t^6 + 6s^8t^6 + 6s^4t^6 + 6t^6 + 12s^{10}t^5 + 8s^6t^5 + 12s^2t^5\\
    +& 3s^{12}t^4 + 5s^8t^4+ 5s^4t^4 + 3t^4\\
    &P(\mathcal{K}(5)^{D_6}/F_{4}\mathcal{K}(5)^{D_6};s,t)\\
    =&s^{12}t^{10} + 5s^{10}t^9 + 5s^2t^9 + 10s^{12}t^8 + 10s^8t^8 + 10s^4t^8 + 10t^8 + 30s^{10}t^7 + 20s^6t^7\\
    +& 30s^2t^7+ 15s^{12}t^6 + 25s^8t^6 + 25s^4t^6 + 15t^6 + 11s^{10}t^5 + 10s^6t^5 + 11s^2t^5\\
    &P(\mathcal{K}(6)^{D_6}/F_{5}\mathcal{K}(6)^{D_6};s,t)\\
    =&t^{12} + 6s^{10}t^{11} + 6s^2t^{11} + 15s^{12}t^{10} + 15s^8t^{10} + 15s^4t^{10} + 15t^{10} + 60s^{10}t^9 + 40s^6t^9\\
    +& 60s^2t^9 + 45s^{12}t^8 +75s^8t^8 + 75s^4t^8 + 45t^8 + 66s^{10}t^7 + 60s^6t^7 + 66s^2t^7 + 11s^{12}t^6\\
    +& 21s^8t^6 + 21s^4t^6 + 11t^6
  \end{align*}
\end{lemma}
\begin{proof}
  By Theorem \ref{bigraded_Poincare_formula}, we obtain the following equations
  \begin{align*}
    &P(\mathcal{K}(1)^{D_6};s,t)=s^{12}t^2 + s^{10}t + s^2t + 1\\
    &P(\mathcal{K}(2)^{D_6};s,t)\\
    =&t^4 + 2s^{10}t^3 + 2s^2t^3 + 3s^{12}t^2 + s^8t^2 + s^4t^2 + t^2 + 2s^{10}t + 2s^2t + 1\\
    &P(\mathcal{K}(3)^{D_6};s,t)\\
    =&s^{12}t^6 + 3s^{10}t^5 + 3s^2t^5 + 3s^{12}t^4 + 3s^8t^4 + 3s^4t^4 + 6t^4 + 9s^{10}t^3+ 2s^6t^3 + 9s^2t^3\\
    +& 6s^{12}t^2 + 3s^8t^2 + 3s^4t^2 + 3t^2+ 3s^{10}t + 3s^2t + 1\\
    &P(\mathcal{K}(4)^{D_6};s,t)\\
    =&t^8 + 4s^{10}t^7 + 4s^2t^7 + 10s^{12}t^6 + 6s^8t^6 + 6s^4t^6 + 6t^6 + 24s^{10}t^5 + 8s^6t^5 + 24s^2t^5\\
    +& 15s^{12}t^4 + 17s^8t^4 + 17s^4t^4 + 21t^4+ 24s^{10}t^3+ 8s^6t^3 + 24s^2t^3 + 10s^{12}t^2 + 6s^8t^2\\
    +& 6s^4t^2 + 6t^2 + 4s^{10}t + 4s^2t + 1\\
    &P(\mathcal{K}(5)^{D_6};s,t)\\
    =&s^{12}t^{10} + 5s^{10}t^9 + 5s^2t^9 + 10s^{12}t^8 + 10s^8t^8 + 10s^4t^8 + 15t^8+ 50s^{10}t^7 + 20s^6t^7\\
    +& 50s^2t^7 + 55s^{12}t^6 + 55s^8t^6 + 55s^4t^6 + 45t^6 + 101s^{10}t^5 + 50s^6t^5 + 101s^2t^5\\
    +& 45s^{12}t^4 + 55s^8t^4 + 55s^4t^4 + 55t^4 + 50s^{10}t^3 + 20s^6t^3 + 50s^2t^3 + 15s^{12}t^2\\
    +& 10s^8t^2+ 10s^4t^2 + 10t^2 + 5s^{10}t + 5s^2t + 1\\
    &P(\mathcal{K}(6)^{D_6};s,t)\\
    =&t^{12} + 6s^{10}t^{11} + 6s^2t^{11} + 21s^{12}t^10 + 15s^8t^{10} + 15s^4t^{10} + 15t^{10} + 90s^{10} t^9+ 40s^6t^9\\
    +& 90s^2t^9 + 105s^{12}t^8 + 135s^8t^8+ 135s^4t^8 + 120t^8 + 306s^{10}t^7 + 180s^6t^7 + 306s^2t^7\\
    +& 211s^{12}t^6 + 261s^8t^6 + 261s^4t^6 + 191t^6 + 306s^{10}t^5 + 180s^6t^5 + 306s^2t^5 + 105s^{12}t^4\\
    +& 135s^8t^4 + 135s^4t^4 + 120t^4 + 90s^{10}t^3 + 40s^6t^3 + 90s^2 t^3+ 21s^{12}t^2 + 15s^8t^2\\
    +& 15s^4t^2 + 15t^2 + 6s^{10}t + 6s^2t + 1.\\
  \end{align*}
  Since $F_n\mathcal{K}(m)^{D_6}$ is the linear subspace of $\mathcal{K}(m)^{D_6}$ spanned by linear combinations of $x^{i}y^{j}w^k\alpha_{I}\beta_J\gamma_K$ with $|I\cup J\cup K|\leq n$, there are equations
  \begin{align*}
    &P(\mathcal{K}(2)^{D_6}/F_{1}\mathcal{K}(2)^{D_6};s,t)=P(\mathcal{K}(2)^{D_6};s,t)-2P(\mathcal{K}(1)^{D_6};s,t)+1\\
    &P(\mathcal{K}(3)^{D_6}/F_{2}\mathcal{K}(3)^{D_6};s,t)\\
    =&P(\mathcal{K}(3)^{D_6};s,t)-3P(\mathcal{K}(2)^{D_6};s,t)+3P(\mathcal{K}(1)^{D_6};s,t)-1\\
    &P(\mathcal{K}(4)^{D_6}/F_{3}\mathcal{K}(4)^{D_6};s,t)\\
    =&P(\mathcal{K}(4)^{D_6};s,t)-4P(\mathcal{K}(3)^{D_6};s,t)+6P(\mathcal{K}(2)^{D_6};s,t)-4P(\mathcal{K}(1)^{D_6};s,t)+1\\
    &P(\mathcal{K}(5)^{D_6}/F_{4}\mathcal{K}(5)^{D_6};s,t)\\
    =&P(\mathcal{K}(5)^{D_6};s,t)-5P(\mathcal{K}(4)^{D_6};s,t)+10P(\mathcal{K}(3)^{D_6};s,t)-10P(\mathcal{K}(2)^{D_6};s,t)\\
    +&5P(\mathcal{K}(1)^{D_6};s,t)-1\\
    &P(\mathcal{K}(6)^{D_6}/F_{5}\mathcal{K}(6)^{D_6};s,t)\\
    =&P(\mathcal{K}(6)^{D_6};s,t)-6P(\mathcal{K}(5)^{D_6};s,t)+15P(\mathcal{K}(4)^{D_6};s,t)-20P(\mathcal{K}(3)^{D_6};s,t)\\
    +&15P(\mathcal{K}(2)^{D_6};s,t)-6P(\mathcal{K}(1)^{D_6};s,t)+1.\\
  \end{align*}
  By combining them, we obtain this lemma. 
\end{proof}

We consider how to use this filtration.
For example we consider the $z(2,\{1,2,3\})$ case.
$z(2,\{1,2,3\})$ is in $(2,3)$ degree part of $\mathcal{K}(3)^{D_6}/F_{2}\mathcal{K}(3)^{D_6}$, $(\mathcal{K}(3)^{D_6}/F_{2}\mathcal{K}(3)^{D_6})^{(i,j)}$.
By Lemma \ref{Poincare_G2_filter}, the dimension of this part is $3$.
If this part is spanned by $z(2,\{1\})z(1,\{2,3\}), z(2,\{2\})z(1,\{1,3\}), z(2,\{3\})z(1,\{1,2\}) \in \mathcal{L}(3)$, we obtain that $z(2,\{1,2,3\})$ is linear combination of them and in $\mathcal{L}(3)$.

To confirm the assumption of Lemma \ref{basis_to_check}, we will show that for pairs $(i,j, m)$ that corresponds to the elements in the assumption, $(i,j)$ degree part of $\mathcal{K}(m)^{D_6}/F_{m-1}\mathcal{K}(m)^{D_6}$ is spanned by the elements in $L(m)$.
Here is the correspondence table for the generators in Lemma \ref{basis_to_check} and its degree.
\begin{table}[H]
  \centering
  \begin{tabular}{p{4.5cm}p{4.5cm}l}
    \hline
    Element&Corresponding part&$(m,i,j)$\\\hline
    $z(2,\{1,2,3\})$&$(\mathcal{K}(3)^{D_6}/F_{2}\mathcal{K}(3)^{D_6})^{(2,3)}$&$(3,2,3)$\\
    $z(4,\{1,2,3\})$&$(\mathcal{K}(3)^{D_6}/F_{2}\mathcal{K}(3)^{D_6})^{(6,3)}$&$(3,6,3)$\\
    $z(1,\{1,2,3\})z(4,\emptyset)$&$(\mathcal{K}(3)^{D_6}/F_{2}\mathcal{K}(3)^{D_6})^{(6,3)}$&$(3,6,3)$\\
    $z(2,\{1,2\})z(3,\{3\})$&$(\mathcal{K}(3)^{D_6}/F_{2}\mathcal{K}(3)^{D_6})^{(6,3)}$&$(3,6,3)$\\
    $z(6,\{1,2,3\})$&$(\mathcal{K}(3)^{D_6}/F_{2}\mathcal{K}(3)^{D_6})^{(10,3)}$&$(3,10,3)$\\
    $z(2,\{1,2\})z(5,\{3\})$&$(\mathcal{K}(3)^{D_6}/F_{2}\mathcal{K}(3)^{D_6})^{(10,3)}$&$(3,10,3)$\\
    $z(1,\{1,2,3\})z(1,\{2,3,4\})$&$(\mathcal{K}(4)^{D_6}/F_{3}\mathcal{K}(4)^{D_6})^{(0,6)}$&$(4,0,6)$\\
    $z(1,\{1,2,3\})z(1,\{3,4,5\})$&$(\mathcal{K}(5)^{D_6}/F_{4}\mathcal{K}(5)^{D_6})^{(0,6)}$&$(5,0,6)$\\
    $z(1,\{1,2,3\})z(1,\{4,5,6\})$&$(\mathcal{K}(6)^{D_6}/F_{5}\mathcal{K}(6)^{D_6})^{(0,6)}$&$(6,0,6)$\\
    $z(1,\{1,2,3\})z(2,\{2,3\})$&$(\mathcal{K}(3)^{D_6}/F_{2}\mathcal{K}(3)^{D_6})^{(2,5)}$&$(3,2,5)$\\
    $z(1,\{1,2,3\})z(2,\{3,4\})$&$(\mathcal{K}(4)^{D_6}/F_{3}\mathcal{K}(4)^{D_6})^{(2,5)}$&$(4,2,5)$\\
    $z(1,\{1,2,3\})z(2,\{4,5\})$&$(\mathcal{K}(4)^{D_6}/F_{3}\mathcal{K}(4)^{D_6})^{(2,5)}$&$(5,2,5)$\\
    $z(1,\{1,2,3\})z(3,\{3\})$&$(\mathcal{K}(3)^{D_6}/F_{2}\mathcal{K}(3)^{D_6})^{(4,4)}$&$(3,4,4)$\\
    $z(2,\{1,2\})z(2,\{2,3\})$&$(\mathcal{K}(3)^{D_6}/F_{2}\mathcal{K}(3)^{D_6})^{(4,4)}$&$(3,4,4)$\\
    $z(1,\{1,2,3\})z(3,\{4\})$&$(\mathcal{K}(4)^{D_6}/F_{3}\mathcal{K}(4)^{D_6})^{(4,4)}$&$(4,4,4)$\\
    $z(2,\{1,2\})z(2,\{3,4\})$&$(\mathcal{K}(4)^{D_6}/F_{3}\mathcal{K}(4)^{D_6})^{(4,4)}$&$(4,4,4)$\\
    $z(1,\{1,2,3\})z(5,\{3\})$&$(\mathcal{K}(3)^{D_6}/F_{2}\mathcal{K}(3)^{D_6})^{(8,4)}$&$(3,8,4)$\\
    $z(1,\{1,2,3\})z(5,\{4\})$&$(\mathcal{K}(4)^{D_6}/F_{3}\mathcal{K}(4)^{D_6})^{(8,4)}$&$(4,8,4)$\\\hline
  \end{tabular}
\end{table}

To prove that the dimension of these part coincides with the dimension of the same degree part of $\mathcal{L}(m)$, we use the following lemma.
\begin{lemma}\label{basis_K}
  The set
  \[
    \{x^iy^j\alpha_I\beta_J\mid 0\leq i\leq 5,0\leq j \leq 1,I,J\subset \{1,\dots m\}\}
  \]
  is a basis of $\mathcal{K}(m)$.
\end{lemma}
\begin{proof}
  The set 
  \[
    \{x^iy^j\mid 0\leq i\leq 5,0\leq j \leq 1\}
  \]
  is a basis of $H^*(G_2/T)\cong \Q[x,y,z]/J$.
  The set $\{\alpha_I\beta_I \mid I,J\subset \{1,\dots m\}\}$ is a basis of 
  \[
    \bigotimes_{1\leq j\leq m} \Lambda(\alpha_j, \beta_j, \gamma_j)/K_j.
  \]
  Since $\mathcal{K}(m)$ is the tensor product of these two algebra, we obtain this lemma.
\end{proof}
By using this basis, the generators of $\mathcal{L}(m)$ can be rewrite as follows; 
\begin{align*}
  &z(2,\{1\})=(2x+y)\alpha_1 +  (x+2y) \beta_1,\\
  &z(1,\{1,2\})=2\alpha_1\alpha_2 +  2\beta_1 \beta_2 + \alpha_1 \beta_2 +\beta_1 \alpha_2, \\
  &z(6,\{1\})=(x^5-x^4y) \alpha_1 +  (-x^5-2x^4y) \beta_1 , \\
  &z(5,\{1,2\})=-x^3y \alpha_1\alpha_2 -x^4 \beta_1 \beta_2 - (x^4+ x^3y) (\alpha_1 \beta_2 +\beta_1 \alpha_2), \\
  &z(4,\{1,2,3\})\\
  =&-x^3\alpha_1 \alpha_2 \beta_3 -x^3\alpha_1 \beta_2 \alpha_3 -x^3\beta_1 \alpha_2 \alpha_3-x^3\alpha_1 \beta_2 \beta_3 -x^3\beta_1 \alpha_2 \beta_3\\
  -&x^3\beta_1 \beta_2 \alpha_3, \\ 
  &z(3,\{1,2,3,4\})\\
   =& (x^2+xy) \alpha_1 \alpha_2 \alpha_3 \alpha_4 - x^2 \beta_1 \beta_2 \beta_3 \beta_4 + xy \alpha_1 \alpha_2 \alpha_3 \beta_4 + xy \alpha_1 \alpha_2 \beta_3 \alpha_4 \\
  +& xy \alpha_1 \beta_2 \alpha_3 \alpha_4 +xy \beta_1 \alpha_2 \alpha_3 \alpha_4 + xy \alpha_1 \alpha_2 \beta_3 \beta_4 + xy \alpha_1 \beta_2 \alpha_3 \beta_4 \\
  +& xy \beta_1 \alpha_2 \alpha_3 \beta_4 +xy \alpha_1 \beta_2 \beta_3 \alpha_4 +xy \beta_1 \alpha_2 \beta_3 \alpha_4 + xy \beta_1 \beta_2 \alpha_3 \alpha_4 \\
  +& xy \alpha_1 \beta_2 \beta_3 \beta_4 + xy \beta_1 \alpha_2 \beta_3 \beta_4 + xy \beta_1 \beta_2 \alpha_3 \beta_4 +xy \beta_1 \beta_2 \beta_3 \alpha_4, \\
  &z(2,\{1,2,3,4,5\}) \\
  =& (2x+y) \alpha_1 \alpha_2 \alpha_3 \alpha_4 \alpha_5 + (x+2y) \beta_1 \beta_2 \beta_3 \beta_4 \beta_5 + (x+y) \alpha_1 \alpha_2 \alpha_3 \alpha_4 \beta_5 \\
  +& (x+y) \alpha_1 \alpha_2 \alpha_3 \beta_4 \alpha_5 + (x+y) \alpha_1 \alpha_2 \beta_3 \alpha_4 \alpha_5 + (x+y) \alpha_1 \beta_2 \alpha_3 \alpha_4 \alpha_5 \\
  +& (x+y) \beta_1 \alpha_2 \alpha_3 \alpha_4 \alpha_5 + (x+y) \alpha_1 \alpha_2 \alpha_3 \beta_4 \beta_5 + (x+y) \alpha_1 \alpha_2 \beta_3 \alpha_4 \beta_5 \\
  +& (x+y) \alpha_1 \beta_2 \alpha_3 \alpha_4 \beta_5 + (x+y) \beta_1 \alpha_2 \alpha_3 \alpha_4 \beta_5 + (x+y) \alpha_1 \alpha_2 \beta_3 \beta_4 \alpha_5 \\
  +& (x+y) \alpha_1 \beta_2 \alpha_3 \beta_4 \alpha_5 + (x+y) \beta_1 \alpha_2 \alpha_3 \beta_4 \alpha_5 + (x+y) \alpha_1 \beta_2 \beta_3 \alpha_4 \alpha_5 \\
  +& (x+y) \beta_1 \alpha_2 \beta_3 \alpha_4 \alpha_5 + (x+y) \beta_1 \beta_2 \alpha_3 \alpha_4 \alpha_5 + (x+y) \alpha_1 \alpha_2 \beta_3 \beta_4 \beta_5 \\
  +& (x+y) \alpha_1 \beta_2 \alpha_3 \beta_4 \beta_5 + (x+y) \beta_1 \alpha_2 \alpha_3 \beta_4 \beta_5 + (x+y) \alpha_1 \beta_2 \beta_3 \alpha_4 \beta_5 \\
  +& (x+y) \beta_1 \alpha_2 \beta_3 \alpha_4 \beta_5 + (x+y) \beta_1 \beta_2 \alpha_3 \alpha_4 \beta_5 + (x+y) \alpha_1 \beta_2 \beta_3 \beta_4 \alpha_5 \\
  +& (x+y) \beta_1 \alpha_2 \beta_3 \beta_4 \alpha_5 + (x+y) \beta_1 \beta_2 \alpha_3 \beta_4 \alpha_5 + (x+y) \beta_1 \beta_2 \beta_3 \alpha_4 \alpha_5 \\
  +& (x+y) \alpha_1 \beta_2 \beta_3 \beta_4 \beta_5 + (x+y) \beta_1 \alpha_2 \beta_3 \beta_4 \beta_5 + (x+y) \beta_1 \beta_2 \alpha_3 \beta_4 \beta_5 \\
  +& (x+y) \beta_1 \beta_2 \beta_3 \alpha_4 \beta_5 + (x+y) \beta_1 \beta_2 \beta_3 \beta_4 \alpha_5, \\
  &z(1,\{1,2,3,4,5,6\})\\
  =& \alpha_1 \alpha_2 \alpha_3  \alpha_4 \alpha_5 \alpha_6 + \beta_1 \beta_2 \beta_3 \beta_4\beta_5 \beta_6\\
  +& (\alpha_1+\beta_1)(\alpha_2+\beta_2)(\alpha_3+\beta_3)(\alpha_4+\beta_4)(\alpha_5+\beta_5)(\alpha_6+\beta_6).
\end{align*}

\begin{lemma}\label{other_cases_generator}
  For following pairs of $m,i,j$, $(\mathcal{K}(m)^{D_6}/F_{m-1}\mathcal{K}(m)^{D_6})^{(i,j)}$ is spanned by the products of the generators of $\mathcal{L}(m)$.
  \begin{enumerate}
    \item $(m,i,j)=(3,2,3)$
    \item $(m,i,j)=(3,6,3)$
    \item $(m,i,j)=(3,10,3)$
    \item $(m,i,j)=(3,4,4)$
    \item $(m,i,j)=(4,4,4)$
    \item $(m,i,j)=(3,2,5)$
    \item $(m,i,j)=(4,2,5)$
    \item $(m,i,j)=(5,2,5)$
    \item $(m,i,j)=(4,0,6)$
    \item $(m,i,j)=(5,0,6)$
    \item $(m,i,j)=(6,0,6)$
    \item $(m,i,j)=(3,8,4)$
    \item $(m,i,j)=(4,8,4)$.
  \end{enumerate}
\end{lemma}
We will prove each case of this lemma.
\begin{proof}[Proof of the case $(m,i,j)=(3,2,3)$]

There is an equation
\begin{align*}
  &z(2,\{1\})z(1,\{2,3\})\\
  =&((2x+y)\alpha_1 +  (x+2y) \beta_1)(2\alpha_2\alpha_3 +  2\beta_2 \beta_3 + \alpha_2\beta_3 +\beta_2 \alpha_3 )\\
  =&(4x+2y)\alpha_1 \alpha_2 \alpha_3  + (4x+2y)\alpha_1 \beta_2 \beta_3 + (2x+y)\alpha_1 \alpha_2\beta_3+ (2x+y)\alpha_1 \beta_2 \alpha_3 \\
  +&(2x+4y) \beta_1 \beta_2 \beta_3 + (2x+4y) \beta_1 \alpha_2\alpha_3 + (x+2y) \beta_1 \alpha_2\beta_3+ (x+2y) \beta_1 \beta_2 \alpha_3.
\end{align*}
Similarly,
\begin{align*}
  &-z(2,\{2\})z(1,\{1,3\})\\
  =&(4x+2y)\alpha_1 \alpha_2 \alpha_3 + (x+2y) \alpha_1 \beta_2 \beta_3  +(2x+y)\alpha_1 \alpha_2\beta_3+  (2x+4y) \alpha_1 \beta_2 \alpha_3  \\
  +&(2x+4y) \beta_1 \beta_2 \beta_3 + (2x+y)\beta_1 \alpha_2\alpha_3 +(4x+2y)\beta_1 \alpha_2\beta_3 + (x+2y) \beta_1 \beta_2 \alpha_3,
\end{align*}
\begin{align*}
  &z(2,\{3\})z(1,\{1,2\})\\
  =&(4x+2y)\alpha_1 \alpha_2 \alpha_3  + (x+2y) \alpha_1 \beta_2 \beta_3  + (2x+4y) \alpha_1 \alpha_2\beta_3 + (2x+y)\alpha_1 \beta_2 \alpha_3 \\
  +&(2x+4y) \beta_1 \beta_2 \beta_3  + (2x+y)\beta_1 \alpha_2\alpha_3 + (x+2y) \beta_1 \alpha_2\beta_3  + (4x+2y)\beta_1 \beta_2 \alpha_3.
\end{align*}
We assume that $a_1z(2,\{1\})z(1,\{2,3\}) -a_2 z(2,\{2\})z(1,\{1,3\}) + a_3z(2,\{3\})z(1,\{1,2\})=0 $ for some $a_1,a_2,a_3 \in \Q$.
By looking at the coefficient of $\alpha_1 \beta_2 \beta_3 ,\alpha_1 \alpha_2\beta_3$, we obtain the equations
\begin{align*}
  (4x+2y)a_1 + (x+2y)a_2 + (x+2y)a_3&=0\\
  (2x+y)a_1 + (2x+y)a_2 + (2x+4y)a_3&=0.
\end{align*} 
Thus $a_1=a_2=a_3=0$ and we obtain that these are linearly independent.
Since the dimension of $(\mathcal{K}(3)^{D_6}/F_{2}\mathcal{K}(3)^{D_6})^{(2,3)}$ is $3$, this pair is a basis of this part.
\end{proof}

\begin{proof}[Proof of the case $(m,i,j)=(3,10,3)$]
  
  There is an equation
\begin{align*}
  &z(6,\{1\})z(1,\{2,3\})\\
  =&((x^5-x^4y)\alpha_1 +  (-x^5-2x^4y) \beta_1)(2\alpha_2\alpha_3 +  2\beta_2 \beta_3 + \alpha_2\beta_3 +\beta_2 \alpha_3 )\\
  =&2(x^5-x^4y)\alpha_1 \alpha_2 \alpha_3  + 2(x^5-x^4y)\alpha_1 \beta_2 \beta_3 + (x^5-x^4y)\alpha_1 \alpha_2\beta_3\\
  +& (x^5-x^4y)\alpha_1 \beta_2 \alpha_3 +2(-x^5-2x^4y)\beta_1 \beta_2 \beta_3 + 2(-x^5-2x^4y) \beta_1 \alpha_2\alpha_3\\
  +& (-x^5-2x^4y) \beta_1 \alpha_2\beta_3+ (-x^5-2x^4y) \beta_1 \beta_2 \alpha_3.
\end{align*}
Similarly, 
\begin{align*}
  &-z(6,\{2\})z(1,\{1,3\})\\
  =&2(x^5-x^4y)\alpha_1 \alpha_2 \alpha_3 + (-x^5-2x^4y) \alpha_1 \beta_2 \beta_3  +(x^5-x^4y)\alpha_1 \alpha_2\beta_3\\
  +&2(-x^5-2x^4y) \alpha_1 \beta_2 \alpha_3  +2(-x^5-2x^4y) \beta_1 \beta_2 \beta_3 + (x^5-x^4y)\beta_1 \alpha_2\alpha_3\\
  +&2(x^5-x^4y)\beta_1 \alpha_2\beta_3 +(-x^5-2x^4y) \beta_1 \beta_2 \alpha_3 
\end{align*}
\begin{align*}
  &z(6,\{3\})z(1,\{1,2\})\\
  =&2(x^5-x^4y)\alpha_1 \alpha_2 \alpha_3  + (-x^5-2x^4y) \alpha_1 \beta_2 \beta_3  + 2(-x^5-2x^4y) \alpha_1 \alpha_2\beta_3\\
  +& (x^5-x^4y)\alpha_1 \beta_2 \alpha_3 +2(-x^5-2x^4y) \beta_1 \beta_2 \beta_3  + (x^5-x^4y)\beta_1 \alpha_2\alpha_3\\
  +& (-x^5-2x^4y) \beta_1 \alpha_2\beta_3  + 2(x^5-x^4y)\beta_1 \beta_2 \alpha_3.
\end{align*}
We assume that 
\[
  a_1z(6,\{1\})z(1,\{2,3\}) -a_2 z(6,\{2\})z(1,\{1,3\}) + a_3z(6,\{3\})z(1,\{1,2\})=0 
\] 
for some $a_1,a_2,a_3\in \Q$.
By the coeffitient of $\alpha_1 \beta_2 \beta_3 ,\alpha_1 \alpha_2\beta_3$, there are equations
\begin{align*}
  2(x^5-x^4y)a_1 + (-x^5-2x^4y)a_2 + (-x^5-2x^4y)a_3&=0\\
  (x^5-x^4y)a_1 + (x^5-x^4y)a_2 + 2(-x^5-2x^4y)a_3&=0.
\end{align*}
Thus $a_1=a_2=a_3=0$ and we obtain that these are linearly independent.
Since the dimension of $(\mathcal{K}(3)^{D_6}/F_{2}\mathcal{K}(3)^{D_6})^{(10,3)}$ is $3$, this pair is a basis of this part.
\end{proof}
  
\begin{proof}[Proof of the case $(m,i,j)=(3,6,3)$]

There are equations
\begin{align*}
  z(2,\{1\})z(2,\{2\})z(2,\{3\})=& ((2x+y)\alpha_1 +  (x+2y) \beta_1)((2x+y)\alpha_2\\
  +&  (x+2y) \beta_2)((2x+y)\alpha_3 +  (x+2y) \beta_3)\\
  =& (2x+y)^3 \alpha_1 \alpha_2 \alpha_3 +\dots \\
  =& (3x^3 + 6x^2y) \alpha_1 \alpha_2 \alpha_3 +\dots,
\end{align*}
\begin{align*}
  z(4,\{1,2,3\})=&-x^3\alpha_1 \alpha_2 \beta_3 -x^3\alpha_1 \beta_2 \alpha_3 -x^3\beta_1 \alpha_2 \alpha_3 \\
  &-x^3\alpha_1 \beta_2 \beta_3 -x^3\beta_1 \alpha_2 \beta_3 -x^3\beta_1 \beta_2 \alpha_3. \\ 
\end{align*}
By comparing the coefficient of $\alpha_1 \alpha_2 \alpha_3$, we obtain that these are linearly independent.
Since the dimension of $(\mathcal{K}(3)^{D_6}/F_{2}\mathcal{K}(3)^{D_6})^{(6,3)}$ is $2$, this pair is a basis of this part.

\end{proof}

\begin{proof}[Proof of the case $(m,i,j)=(3,4,4)$]

There is an equation
\begin{align*}
  &z(2,\{1\})z(2,\{2\})z(1,\{2,3\})\\
  =& ((2x+y)\alpha_1 +  (x+2y) \beta_1)((2x+y)\alpha_2 +  (x+2y) \beta_2)(\alpha_2\alpha_3 +  \beta_2 \beta_3 + 2\alpha_2\beta_3 +2\beta_2 \alpha_3 )\\
  =& (2x+y)(2x+y-(x+2y)) \alpha_1 \alpha_2 \beta_2 \alpha_3 + \dots \\
  =& 3x^2 \alpha_1 \alpha_2 \beta_2 \alpha_3 + 0\alpha_1 \beta_1 \alpha_2 \alpha_3+0\alpha_1 \alpha_2 \alpha_3 \beta_3\dots 
\end{align*}
Similarly,
\begin{align*}
  &z(2,\{1\})z(2,\{2\})z(1,\{1,3\})\\
  =&0\alpha_1 \alpha_2 \beta_2 \alpha_3 - 3x^2 \alpha_1 \beta_1 \alpha_2 \alpha_3 + 0\alpha_1 \alpha_2 \alpha_3 \beta_3 + \dots 
\end{align*}
\begin{align*}
  &z(2,\{1\})z(2,\{3\})z(1,\{2,3\})\\
  =& 0\alpha_1 \alpha_2 \beta_2 \alpha_3+0 \alpha_1 \beta_1 \alpha_2 \alpha_3- 3x^2 \alpha_1 \alpha_2 \alpha_3 \beta_3 + \dots 
\end{align*}
By comparing the coefficient of $\alpha_1 \alpha_2 \beta_2 \alpha_3, \alpha_1 \beta_1 \alpha_2 \alpha_3,\alpha_1 \alpha_2 \alpha_3 \beta_3$, we obtain that these are linearly independent.
Since the dimension of $(\mathcal{K}(3)^{D_6}/F_{2}\mathcal{K}(3)^{D_6})^{(4,4)}$ is $3$, this pair is a basis of this part.
\end{proof}

\begin{proof}[Proof of the case $(m,i,j)=(4,4,4)$]
  
  There is an equation
  \begin{align*}
      &z(2,\{1\})z(2,\{2\})z(1,\{3,4\})\\
      =& ((2x+y)\alpha_1 +  (x+2y) \beta_1)((2x+y)\alpha_2\\
      +&  (x+2y) \beta_2)(2\alpha_3\alpha_4 +  2\beta_3 \beta_4 + \alpha_3 \beta_4 +\beta_3 \alpha_4 )\\
      =& 2(2x+y)(2x+y) \alpha_1 \alpha_2 \alpha_3 \alpha_4 + (2x+y)(2x+y) \alpha_1 \alpha_2 \beta_3 \beta_4\\
      +& 2(2x+y)(2x+y) \alpha_1 \alpha_2 \alpha_3 \beta_4 + (2x+y)(2x+y) \alpha_1 \alpha_2 \beta_3 \alpha_4\\
      +& 2(2x+y)(x+2y) \alpha_1 \beta_2 \alpha_3 \alpha_4 + 2(2x+y)(x+2y) \alpha_1 \beta_2 \beta_3 \beta_4 \\
      +& (2x+y)(x+2y) \alpha_1 \beta_2 \alpha_3 \beta_4 + (2x+y)(x+2y) \alpha_1 \beta_2 \beta_3 \alpha_4\\
      +& 2(2x+y)(x+2y) \beta_1 \alpha_2 \alpha_3 \alpha_4 + 2(2x+y)(x+2y) \beta_1 \alpha_2 \beta_3 \beta_4 \\
      +& (2x+y)(x+2y) \beta_1 \alpha_2 \alpha_3 \beta_4 + (2x+y)(x+2y) \beta_1 \alpha_2 \beta_3 \alpha_4\\
      +& 2(x+2y)(x+2y) \beta_1 \beta_2 \alpha_3 \alpha_4+ 2(x+2y)(x+2y) \beta_1 \beta_2 \beta_3 \beta_4\\
      +& (x+2y)(x+2y) \beta_1 \beta_2 \alpha_3 \beta_4 + (x+2y)(x+2y) \beta_1 \beta_2 \beta_3 \alpha_4\\
      =& (6x^2+6xy) \alpha_1 \alpha_2 \alpha_3 \alpha_4 + (6x^2+6xy) \alpha_1 \alpha_2 \beta_3 \beta_4 + (3x^2+3xy) \alpha_1 \alpha_2 \alpha_3 \beta_4\\
      +& (3x^2+3xy) \alpha_1 \alpha_2 \beta_3 \alpha_4+ 6xy \alpha_1 \beta_2 \alpha_3 \alpha_4 + 6xy \alpha_1 \beta_2 \beta_3 \beta_4 + 3xy \alpha_1 \beta_2 \alpha_3 \beta_4\\
      +& 3xy \alpha_1 \beta_2 \beta_3 \alpha_4+ 6xy \beta_1 \alpha_2 \alpha_3 \alpha_4 + 6xy \beta_1 \alpha_2 \beta_3 \beta_4 + 3xy \beta_1 \alpha_2 \alpha_3 \beta_4\\
      +& 3xy \beta_1 \alpha_2 \beta_3 \alpha_4- 6x^2 \beta_1 \beta_2 \alpha_3 \alpha_4 - 6x^2 \beta_1 \beta_2 \beta_3 \beta_4 -3x^2 \beta_1 \beta_2 \alpha_3 \beta_4\\
      -& 3x^2 \beta_1 \beta_2 \beta_3 \alpha_4.\\
  \end{align*}
  Similarly
  \begin{align*}
      &-z(2,\{1\})z(2,\{3\})z(1,\{2,4\})\\
      =& (6x^2+6xy) \alpha_1 \alpha_2 \alpha_3 \alpha_4 + 3xy \alpha_1 \alpha_2 \beta_3 \beta_4 + (3x^2+3xy) \alpha_1 \alpha_2 \alpha_3 \beta_4\\
      +& 6xy \alpha_1 \alpha_2 \beta_3 \alpha_4+ (3x^2+3xy) \alpha_1 \beta_2 \alpha_3 \alpha_4 + 6xy \alpha_1 \beta_2 \beta_3 \beta_4\\
      +& (6x^2+6xy) \alpha_1 \beta_2 \alpha_3 \beta_4 + 3xy \alpha_1 \beta_2 \beta_3 \alpha_4+ 6xy \beta_1 \alpha_2 \alpha_3 \alpha_4 - 3x^2 \beta_1 \alpha_2 \beta_3 \beta_4\\
      +& 3xy \beta_1 \alpha_2 \alpha_3 \beta_4 -6x^2 \beta_1 \alpha_2 \beta_3 \alpha_4+ 3xy \beta_1 \beta_2 \alpha_3 \alpha_4 - 6x^2 \beta_1 \beta_2 \beta_3 \beta_4\\
      +& 6xy \beta_1 \beta_2 \alpha_3 \beta_4 -3x^2 \beta_1 \beta_2 \beta_3 \alpha_4,\\
  \end{align*}
  \begin{align*}
      &z(2,\{1\})z(2,\{4\})z(1,\{2,3\})\\
      =& (6x^2+6xy) \alpha_1 \alpha_2 \alpha_3 \alpha_4 + 3xy \alpha_1 \alpha_2 \beta_3 \beta_4 + 6xy \alpha_1 \alpha_2 \alpha_3 \beta_4\\
      +& (3x^2+3xy) \alpha_1 \alpha_2 \beta_3 \alpha_4+ (3x^2+3xy) \alpha_1 \beta_2 \alpha_3 \alpha_4 + 6xy \alpha_1 \beta_2 \beta_3 \beta_4\\
      +& 3xy \alpha_1 \beta_2 \alpha_3 \beta_4 + (6x^2+6xy) \alpha_1 \beta_2 \beta_3 \alpha_4+ 6xy \beta_1 \alpha_2 \alpha_3 \alpha_4- 3x^2 \beta_1 \alpha_2 \beta_3 \beta_4\\
      -& 6x^2 \beta_1 \alpha_2 \alpha_3 \beta_4 + 3xy \beta_1 \alpha_2 \beta_3 \alpha_4+ 3xy \beta_1 \beta_2 \alpha_3 \alpha_4 - 6x^2 \beta_1 \beta_2 \beta_3 \beta_4\\
      -&3x^2 \beta_1 \beta_2 \alpha_3 \beta_4 +6xy \beta_1 \beta_2 \beta_3 \alpha_4,\\
  \end{align*}
  \begin{align*}
      &z(2,\{2\})z(2,\{3\})z(1,\{1,4\})\\
      =& (6x^2+6xy) \alpha_1 \alpha_2 \alpha_3 \alpha_4 + 3xy \alpha_1 \alpha_2 \beta_3 \beta_4 + (3x^2+3xy) \alpha_1 \alpha_2 \alpha_3 \beta_4\\
      +& 6xy \alpha_1 \alpha_2 \beta_3 \alpha_4+ 6xy \alpha_1 \beta_2 \alpha_3 \alpha_4 - 3x^2 \alpha_1 \beta_2 \beta_3 \beta_4 + 3xy \alpha_1 \beta_2 \alpha_3 \beta_4\\
      -&6x^2 \alpha_1 \beta_2 \beta_3 \alpha_4+ (3x^2+3xy) \beta_1 \alpha_2 \alpha_3 \alpha_4 + 6xy \beta_1 \alpha_2 \beta_3 \beta_4\\
      +& (6x^2+6xy) \beta_1 \alpha_2 \alpha_3 \beta_4+ 3xy \beta_1 \alpha_2 \beta_3 \alpha_4+ 3xy \beta_1 \beta_2 \alpha_3 \alpha_4 - 6x^2 \beta_1 \beta_2 \beta_3 \beta_4\\
      +&6xy \beta_1 \beta_2 \alpha_3 \beta_4 -3x^2 \beta_1 \beta_2 \beta_3 \alpha_4.\\
  \end{align*}

  We consider the equation 
  \begin{align*}
    &a_1z(2,\{1\})z(2,\{2\})z(1,\{3,4\})  -a_2 z(2,\{1\})z(2,\{3\})z(1,\{2,4\})\\
    +& a_3z(2,\{1\})z(2,\{4\})z(1,\{2,3\})+a_4 z(2,\{2\})z(2,\{3\})z(1,\{1,4\})\\
    -&3a_5 z(3,\{1,2,3,4\})=0.
  \end{align*}
  By considering the coefficient of 
  \[
    \alpha_1 \alpha_2 \beta_3 \beta_4,\alpha_1 \beta_2 \alpha_3 \beta_4,\beta_1 \beta_2 \beta_3 \alpha_4,\alpha_1 \beta_2 \beta_3 \alpha_4,
  \]
  we obtain the equation
  \begin{align*}
    (2x^2+2xy)a_1+xya_2+ xya_3 +xya_4 +xy a_5 &=0\\
    xya_1+(2x^2+2xy)a_2+ xya_3 +xya_4 +xy a_5 &=0\\
    -x^2a_1-x^2a_2+ 2xya_3 -x^2a_4 +xy a_5 &=0\\
    xya_1+xya_2+ (2x^2+2xy)a_3 -2x^2a_5 +xy a_5 &=0.\\
  \end{align*}
  Thus $a_1=a_2=a_3=a_4=a_5=0$, and we obtain that these are linearly independent.
  Since the dimension of $(\mathcal{K}(4)^{D_6}/F_{3}\mathcal{K}(4)^{D_6})^{(4,4)}$ is $5$, this pair is a basis of this part.
\end{proof}

\begin{proof}[Proof of the case $(m,i,j)=(3,2,5)$]

  There is an equation
  \begin{align*}
      &z(2,\{1\})z(1,\{2,3\})^2 \\
      =&((2x+y)\alpha_1 +  (x+2y) \beta_1)(2\alpha_2\alpha_3 +  2\beta_2 \beta_3 + \alpha_2\beta_3 +\beta_2 \alpha_3 )^2\\
      =&((2x+y)\alpha_1 +  (x+2y) \beta_1)(-6\alpha_2\alpha_3\beta_2\beta_3)\\
      =&-6(2x+y)\alpha_1 \alpha_2\alpha_3\beta_2\beta_3-6  (x+2y) \beta_1\alpha_2\alpha_3\beta_2\beta_3.
  \end{align*}
  Similarly,
  \begin{align*}
    &z(2,\{2\})z(1,\{1,3\})^2 \\
    =&-6(2x+y)\alpha_2 \alpha_1\alpha_3\beta_1\beta_3-6(x+2y) \beta_2\alpha_1\alpha_3\beta_1\beta_3,\\
    &z(2,\{3\})z(1,\{1,2\})^2 \\
    =&-6(2x+y)\alpha_3 \alpha_1\alpha_2\beta_1\beta_2-6(x+2y) \beta_3\alpha_1\alpha_2\beta_1\beta_2.
  \end{align*}
  By looking at the coefficient of $\alpha_1 \alpha_2\alpha_3\beta_2\beta_3,\alpha_2 \alpha_1\alpha_3\beta_1\beta_3,\alpha_3 \alpha_1\alpha_2\beta_1\beta_2$, we obtain that these are linearly independent.
  Since the dimension of $(\mathcal{K}(3)^{D_6}/F_{2}\mathcal{K}(3)^{D_6})^{(2,5)}$ is $3$, this pair is a basis of this part.
\end{proof}

\begin{proof}[Proof of the case $(m,i,j)=(4,2,5)$]

  There is an equation
\begin{align*}
  &z(2,\{1\})z(1,\{1,2\})z(1,\{3,4\}) \\
  =&((2x+y)\alpha_1 +  (x+2y) \beta_1)(2\alpha_1\alpha_2 + 2\beta_1 \beta_2 +\alpha_1 \beta_2 +\beta_1 \alpha_2)\\
  \times & (2\alpha_3\alpha_4 +2\beta_3 \beta_4 + \alpha_3 \beta_4 +\beta_3 \alpha_4)\\
  =&-3x\alpha_1\beta_1 \alpha_2 \alpha_3 \alpha_4 - 3x \alpha_1 \beta_1 \alpha_2 \alpha_3 \beta_4 - 3x \alpha_1 \beta_1 \alpha_2 \beta_3 \alpha_4 - 3x \alpha_1 \beta_1 \alpha_2 \beta_3 \beta_4 \\
  +&3y \alpha_1 \beta_1 \beta_2\beta_3 \beta_4 +3y \alpha_1 \beta_1 \beta_2\beta_3 \alpha_4 +3y \alpha_1 \beta_1 \beta_2\alpha_3 \beta_4 +3y \alpha_1 \beta_1 \beta_2\alpha_3 \alpha_4.
\end{align*}
Similarly,
\begin{align*}
  &-z(2,\{1\})z(1,\{1,3\})z(1,\{2,4\}) \\
  =&-3x\alpha_1\beta_1 \alpha_2 \alpha_3 \alpha_4 - 3x \alpha_1 \beta_1 \alpha_2 \alpha_3 \beta_4 + 3y \alpha_1 \beta_1 \alpha_2 \beta_3 \alpha_4 + 3y \alpha_1 \beta_1 \alpha_2 \beta_3 \beta_4 \\
  +&3y \alpha_1 \beta_1 \beta_2\beta_3 \beta_4 +3y \alpha_1 \beta_1 \beta_2\beta_3 \alpha_4 -3x \alpha_1 \beta_1 \beta_2\alpha_3 \beta_4 -3x \alpha_1 \beta_1 \beta_2\alpha_3 \alpha_4,
\end{align*}
\begin{align*}
  &z(2,\{1\})z(1,\{1,4\})z(1,\{2,3\}) \\
  =&-3x\alpha_1\beta_1 \alpha_2 \alpha_3 \alpha_4 + 3y \alpha_1 \beta_1 \alpha_2 \alpha_3 \beta_4 - 3x \alpha_1 \beta_1 \alpha_2 \beta_3 \alpha_4 + 3y \alpha_1 \beta_1 \alpha_2 \beta_3 \beta_4 \\
  +&3y \alpha_1 \beta_1 \beta_2\beta_3 \beta_4 -3x \alpha_1 \beta_1 \beta_2\beta_3 \alpha_4 +3y \alpha_1 \beta_1 \beta_2\alpha_3 \beta_4 -3x \alpha_1 \beta_1 \beta_2\alpha_3 \alpha_4.
\end{align*}
By looking at the coefficient of $\alpha_1 \beta_1 \alpha_2 \alpha_3 \beta_4,\alpha_1 \beta_1 \alpha_2 \beta_3 \alpha_4,\alpha_1 \beta_1 \alpha_2 \beta_3 \beta_4$, we obtain that these are linearly independent.
When $z(2,\{i\})z(1,\{1,2\})z(1,\{3,4\})$, $z(2,\{i\})z(1,\{1,3\})z(1,\{2,4\})$ and $z(2,\{i\})z(1,\{1,4\})z(1,\{2,3\})$ are represented by linear combination of basis in Lemma \ref{basis_K}, the basis with non-zero coefficient are multiple of $\alpha_i\beta_i$ and not multiple of $\alpha_j\beta_j$ for $j\ne i$.
Thus the following are linearly independent 
\begin{align*}
  z(2,\{1\})z(1,\{1,2\})z(1,\{3,4\}),z(2,\{1\})z(1,\{1,3\})z(1,\{2,4\}),\\
  z(2,\{1\})z(1,\{1,4\})z(1,\{2,3\}),z(2,\{2\})z(1,\{1,2\})z(1,\{3,4\}),\\
  z(2,\{2\})z(1,\{1,3\})z(1,\{2,4\}),z(2,\{2\})z(1,\{1,4\})z(1,\{2,3\}),\\
  z(2,\{3\})z(1,\{1,2\})z(1,\{3,4\}),z(2,\{3\})z(1,\{1,3\})z(1,\{2,4\}),\\
  z(2,\{3\})z(1,\{1,4\})z(1,\{2,3\}),z(2,\{4\})z(1,\{1,2\})z(1,\{3,4\}),\\
  z(2,\{4\})z(1,\{1,3\})z(1,\{2,4\}),z(2,\{4\})z(1,\{1,4\})z(1,\{2,3\}).
\end{align*}
Since dimension of $(\mathcal{K}(4)^{D_6}/F_{3}\mathcal{K}(4)^{D_6})^{(2,5)}$ is $12$, these are a basis of this part.
\end{proof}

\begin{proof}[Proof of the case $(m,i,j)=(5,2,5)$]

  There is an equation
  \begin{align*}
      &z(2,\{1\})z(1,\{2,3\})z(1,\{4,5\}) \\
      =&((2x+y)\alpha_1 +  (x+2y) \beta_1)(2\alpha_2\alpha_3 + 2 \beta_2 \beta_3 + \alpha_2\beta_3 +\beta_2 \alpha_3)\\
      \times &(2\alpha_4\alpha_5 +  2\beta_4 \beta_5 + \alpha_4 \beta_5 +\beta_4 \alpha_5)\\
      =&4(2x+y)\alpha_1 \alpha_2\alpha_3 \alpha_4 \alpha_5 + 2(2x+y)\alpha_1 \beta_2\alpha_3 \alpha_4 \alpha_5 +2(2x+y)\alpha_1 \alpha_2\beta_3 \alpha_4 \alpha_5\\
      +&4(2x+y)\alpha_1 \beta_2 \beta_3 \alpha_4 \alpha_5+2(2x+y)\alpha_1 \alpha_2\alpha_3 \beta_4 \alpha_5 + (2x+y)\alpha_1 \beta_2\alpha_3 \beta_4 \alpha_5\\
      +&(2x+y)\alpha_1 \alpha_2\beta_3 \beta_4 \alpha_5 +2(2x+y)\alpha_1 \beta_2 \beta_3 \beta_4 \alpha_5+2(2x+y)\alpha_1 \alpha_2\alpha_3 \alpha_4 \beta_5\\
      +& (2x+y)\alpha_1 \beta_2\alpha_3 \alpha_4 \beta_5 +(2x+y)\alpha_1 \alpha_2\beta_3 \alpha_4 \beta_5 +2(2x+y)\alpha_1 \beta_2 \beta_3 \alpha_4 \beta_5\\
      +&4(2x+y)\alpha_1 \alpha_2\alpha_3 \beta_4 \beta_5 + 2(2x+y)\alpha_1 \beta_2\alpha_3 \beta_4 \beta_5 +2(2x+y)\alpha_1 \alpha_2\beta_3 \beta_4 \beta_5\\
      +&4(2x+y)\alpha_1 \beta_2 \beta_3 \beta_4 \beta_5+4(x+2y)\beta_1 \alpha_2\alpha_3 \alpha_4 \alpha_5 + 2(x+2y)\beta_1 \beta_2\alpha_3 \alpha_4 \alpha_5\\
      +&2(x+2y)\beta_1 \alpha_2\beta_3 \alpha_4 \alpha_5 +4(x+2y)\beta_1 \beta_2 \beta_3 \alpha_4 \alpha_5+2(x+2y)\beta_1 \alpha_2\alpha_3 \beta_4 \alpha_5\\
      +&(x+2y)\beta_1 \beta_2\alpha_3 \beta_4 \alpha_5 +(x+2y)\beta_1 \alpha_2\beta_3 \beta_4 \alpha_5 +2(x+2y)\beta_1 \beta_2 \beta_3 \beta_4 \alpha_5\\ 
      +&2(x+2y)\beta_1 \alpha_2\alpha_3 \alpha_4 \beta_5 + (x+2y)\beta_1 \beta_2\alpha_3 \alpha_4 \beta_5 +(x+2y)\beta_1 \alpha_2\beta_3 \alpha_4 \beta_5\\
      +&2(x+2y)\beta_1 \beta_2 \beta_3 \alpha_4 \beta_5+4(x+2y)\beta_1 \alpha_2\alpha_3 \beta_4 \beta_5 + 2(x+2y)\beta_1 \beta_2\alpha_3 \beta_4 \beta_5\\
      +&2(x+2y)\beta_1 \alpha_2\beta_3 \beta_4 \beta_5 +4(x+2y)\beta_1 \beta_2 \beta_3 \beta_4 \beta_5,\\
  \end{align*}
  and similar elements can be computed similarly.
  We consider the equation
  \begin{align*}
      &a_1z(1,\{1\})z(0,\{2,3\})z(0,\{4,5\})-a_2z(1,\{1\})z(0,\{2,4\})z(0,\{3,5\})\\
      +&a_3z(1,\{1\})z(0,\{2,5\})z(0,\{3,4\})-a_4z(1,\{2\})z(0,\{1,3\})z(0,\{4,5\})\\
      +&a_5 z(1,\{2\})z(0,\{1,4\})z(0,\{3,5\})-a_6 z(1,\{2\})z(0,\{1,5\})z(0,\{3,4\})\\
      +&a_7 z(1,\{3\})z(0,\{1,4\})z(0,\{2,5\})+a_8 z(1,\{3\})z(0,\{1,5\})z(0,\{2,4\})\\
      +&a_9 z(1,\{3\})z(0,\{1,2\})z(0,\{4,5\})-a_{10}z(1,\{4\})z(0,\{1,2\})z(0,\{3,5\})\\
      -&a_{11}z(1,\{4\})z(0,\{1,5\})z(0,\{2,3\})+a_{12}z(1,\{4\})z(0,\{1,3\})z(0,\{2,5\})\\
      +&a_{13}z(1,\{5\})z(0,\{1,2\})z(0,\{3,4\})-a_{14}z(1,\{5\})z(0,\{1,3\})z(0,\{2,4\})\\
      +&a_{15}z(1,\{5\})z(0,\{1,4\})z(0,\{2,3\})+a_{16}z(1,\{1,2,3,4,5\})=0.
  \end{align*}
    By the coefficient of 
      \[
        \beta_1 \alpha_2\alpha_3 \alpha_4 \alpha_5, \alpha_1 \alpha_2\alpha_3 \alpha_4 \alpha_5
      \]
      we obtain the following equations
      \begin{align*}
        &2(2x+y)(a_4+a_5+a_6+a_7+a_8+a_9+a_{10}+a_{11}+a_{12}+a_{13}+a_{14}+a_{15})\\
        +&4(x+2y)(a_1+a_2+a_3)+(x+y)a_{16}=0\\
        &(2x+y)(a_1+a_2+a_3+a_4+a_5+a_6+a_7+a_8+a_9+a_{10}+a_{11}+a_{12}\\
        +&a_{13}+a_{14}+a_{15}+a_{16})=0,
      \end{align*}
      and we obtain
      \begin{align*}
        &4a_1+4a_2+4a_3+4a_4+4a_5+4a_6+4a_7+4a_8+4a_9+4a_{10}+4a_{11}\\
        +&4a_{12}+4a_{13}+4a_{14}+4a_{15}+a_{16}=0\\
        &a_1+a_2+a_3+a_4+a_5+a_6+a_7+a_8+a_9+a_{10}+a_{11}+a_{12}+a_{13}\\
        +&a_{14}+a_{15}+a_{16}=0.
      \end{align*}
      By these equations, we obtain that $a_{16}=0$, and $z(2,\{1,2,3,4,5\})$ is linearly independent to the others.

      And in the following equations
  \begin{align*}
      &a_1z(2,\{1\})z(1,\{2,3\})z(1,\{4,5\})-a_2z(2,\{1\})z(1,\{2,4\})z(1,\{3,5\})\\
      -&a_3z(2,\{2\})z(1,\{1,3\})z(1,\{4,5\})+a_4z(2,\{2\})z(1,\{1,4\})z(1,\{3,5\})\\
      +&a_5 z(2,\{3\})z(1,\{1,4\})z(1,\{2,5\})+a_6 z(2,\{3\})z(1,\{1,5\})z(1,\{2,4\})\\
      -&a_7z(2,\{4\})z(1,\{1,2\})z(1,\{3,5\})-a_8z(2,\{4\})z(1,\{1,5\})z(1,\{2,3\})\\
      +&a_9z(2,\{5\})z(1,\{1,2\})z(1,\{3,4\})-a_{10}z(2,\{5\})z(1,\{1,3\})z(1,\{2,4\})=0,
  \end{align*}
  the coefficient of 
  \begin{align*}
    &\beta_1 \alpha_2\alpha_3 \alpha_4 \alpha_5,\alpha_1 \beta_2\alpha_3 \alpha_4 \alpha_5,\alpha_1 \alpha_2\beta_3 \alpha_4 \alpha_5,\alpha_1 \alpha_2\alpha_3 \beta_4 \alpha_5,\alpha_1 \alpha_2\alpha_3 \alpha_4 \beta_5\\
    &\alpha_1 \beta_2\beta_3 \alpha_4 \alpha_5,\alpha_1 \beta_2\alpha_3 \beta_4 \alpha_5, \alpha_1 \beta_2\alpha_3 \alpha_4 \beta_5, \alpha_1 \alpha_2\beta_3 \beta_4 \alpha_5,\alpha_1 \alpha_2\beta_3 \alpha_4 \beta_5,\\
    &\alpha_1 \alpha_2\alpha_3 \beta_4 \beta_5,\beta_1 \beta_2\alpha_3 \alpha_4 \alpha_5,\beta_1 \alpha_2\beta_3 \alpha_4 \alpha_5,\beta_1 \alpha_2\alpha_3 \beta_4 \alpha_5,\beta_1 \alpha_2\alpha_3 \alpha_4 \beta_5,\\
  \end{align*}%
  satisfies the following equations
  \begin{align*}
      &4(x+2y)(a_1+a_2)+2(2x+y)(a_3+a_4+a_5+a_6+a_7+a_8+a_9+a_{10})=0\\
      &4(x+2y)(a_3+a_4)+2(2x+y)(a_1+a_2+a_5+a_6+a_7+a_8+a_9+a_{10})=0\\
      &4(x+2y)(a_5+a_6)+2(2x+y)(a_1+a_2+a_3+a_4+a_7+a_8+a_9+a_{10})=0\\
      &4(x+2y)(a_7+a_8)+2(2x+y)(a_1+a_2+a_3+a_4+a_5+a_6+a_9+a_{10})=0\\
      &4(x+2y)(a_9+a_{10})+2(2x+y)(a_1+a_2+a_3+a_4+a_5+a_6+a_7+a_8)=0\\
      &2(x+2y)(a_3+a_4+a_5+a_6)+(2x+y)(4a_1+a_2+a_7+4a_8+a_9+a_{10})=0\\
      &2(x+2y)(a_3+a_4+a_9+a_{10})+(2x+y)(a_1+a_2+4a_5+a_6+a_7+a_8)=0\\
      &2(x+2y)(a_1+a_2+a_5+a_6)+(2x+y)(4a_3+a_4+a_7+a_8+a_9+4a_{10})=0\\
      &2(x+2y)(a_1+a_2+a_7+a_8)+(2x+y)(a_3+4a_4+4a_5+a_6+a_9+a_{10})=0\\
      &2(x+2y)(a_1+a_2+a_9+a_{10})+(2x+y)(a_3+a_4+a_5+4a_6+a_7+4a_8)=0.\\
  \end{align*}
  Then by comparing the coefficient of $x$ and $y$ for each equation, we can obtain the following equations
  \begin{align*}
      &a_1+a_2=0\\
      &a_3+a_4=0\\
      &a_5+a_6=0\\
      &a_7+a_8=0\\
      &a_9+a_{10}=0\\
      &4a_1+a_2+a_7+4a_8+a_9+a_{10}=0\\
      &a_1+a_2+4a_5+a_6+a_7+a_8=0\\
      &4a_3+a_4+a_7+a_8+a_9+4a_{10}=0\\
      &a_3+4a_4+4a_5+a_6+a_9+a_{10}=0\\
      &a_3+a_4+a_5+4a_6+a_7+4a_8=0,\\
  \end{align*}
  and these equation induces $a_1=a_2=a_3=a_4=a_5=a_6=a_7=a_8=a_9=a_{10}=0$.
  The dimension of $(\mathcal{K}(5)^{D_6}/F_{4}\mathcal{K}(5)^{D_6})^{(2,5)}$ is $11$.
  Thus the pair of these $10$ elements and $z(2,\{1,2,3,4,5\})$ is a basis of this part.
\end{proof}


\begin{proof}[Proof of the case $(m,i,j)=(4,0,6)$]

  There is an equation
  \begin{align*}
      &z(1,\{1,2\})^2z(1,\{3,4\})\\
      =&(2\alpha_1\alpha_2 + 2\beta_1 \beta_2 + \alpha_1 \beta_2 +\beta_1 \alpha_2 )^2(2\alpha_3\alpha_4 +  2\beta_3 \beta_4 +\alpha_3 \beta_4 +\beta_3 \alpha_4 )\\
      =&6\alpha_1\beta_1 \alpha_2\beta_2(2\alpha_3\alpha_4 +  2\beta_3 \beta_4 + \alpha_3 \beta_4 +\beta_3 \alpha_4 )\ne 0.
  \end{align*}
  When $z(1,\{i_1,i_2\})^2z(1,\{i_3,i_4\})$ for $\{i_1,i_2,i_3,i_4\}=\{1,2,3,4\}$ is represented by linear combination of basis in Lemma \ref{basis_K}, the basis with non-zero coefficient are multiple of $\alpha_{i_1}\beta_{i_1}\alpha_{i_2}\beta_{i_2}$.
  Thus
  \begin{align*}
    &z(0,\{1,2\})^2z(0,\{3,4\}),z(0,\{1,3\})^2z(0,\{2,4\}),z(0,\{1,4\})^2z(0,\{2,3\}),\\
    &z(0,\{2,3\})^2z(0,\{1,4\}),z(0,\{2,4\})^2z(0,\{1,3\}),z(0,\{3,4\})^2z(0,\{1,2\})
  \end{align*}
  are linearly independent.
  Since the dimension of $(\mathcal{K}(4)^{D_6}/F_{3}\mathcal{K}(4)^{D_6})^{(0,6)}$ is $6$, this pair is a basis of this part.
\end{proof}

\begin{proof}[Proof of the case $(m,i,j)=(5,0,6)$]

There are equations
  \begin{align*}
      &z(1,\{1,2\})z(1,\{1,3\})z(1,\{4,5\})\\
      =&3 \alpha_1\beta_1 (2\beta_2 \alpha_3 \alpha_4 \alpha_5  - 2 \alpha_2 \beta_3 \alpha_4 \alpha_5  +  \beta_2 \alpha_3 \beta_4 \alpha_5 +  \beta_2 \alpha_3 \alpha_4 \beta_5 -  \alpha_2 \beta_3 \beta_4 \alpha_5\\
      -& \alpha_2 \beta_3 \alpha_4 \beta_5+ 2 \beta_2 \alpha_3 \beta_4 \beta_5 -2  \alpha_2 \beta_3 \beta_4 \beta_5),\\
      &z(1,\{1,2\})z(1,\{1,4\})z(1,\{3,5\})\\
      =& 3\alpha_1\beta_1 (-2\beta_2 \alpha_3 \alpha_4 \alpha_5  + 2 \alpha_2 \alpha_3\beta_4 \alpha_5  - \beta_2 \beta_3 \alpha_4 \alpha_5 -  \beta_2 \alpha_3 \alpha_4 \beta_5 +  \alpha_2 \beta_3 \beta_4 \alpha_5\\
      +&\alpha_2 \alpha_3 \beta_4 \beta_5- 2 \beta_2 \beta_3 \alpha_4 \beta_5 + 2 \alpha_2 \beta_3 \beta_4 \beta_5),\\
      &z(1,\{1,2\})z(1,\{1,5\})z(1,\{3,4\})\\
      =&3 \alpha_1\beta_1 (2\beta_2 \alpha_3 \alpha_4 \alpha_5  - 2 \alpha_2 \alpha_3 \alpha_4 \beta_5 + \beta_2 \beta_3 \alpha_4\alpha_5 +\beta_2 \alpha_3 \beta_4\alpha_5-\alpha_2 \beta_3 \alpha_4 \beta_5\\
      -& \alpha_2 \alpha_3 \beta_4 \beta_5+2 \beta_2 \beta_3 \beta_4\alpha_5 -2\alpha_2 \beta_3 \beta_4\beta_5).
  \end{align*}
  By considering the coefficient of 
  \[
    \alpha_1\beta_1 \alpha_2 \beta_3 \alpha_4 \alpha_5, \alpha_1\beta_1 \alpha_2 \alpha_3 \beta_4 \alpha_5,\alpha_1\beta_1 \alpha_2 \alpha_3 \alpha_4 \beta_5,
  \]
  we obtain that 
  \begin{align*}
    &z(0,\{1,2\})z(0,\{1,3\})z(0,\{4,5\}),z(0,\{1,2\})z(0,\{1,4\})z(0,\{3,5\}),\\
    &z(0,\{1,2\})z(0,\{1,5\})z(0,\{3,4\})
  \end{align*}
  are linearly independent.
  When $z(1,\{i_1,i_2\})z(1,\{i_1,i_3\})z(1,\{i_4,i_5\})$ for $\{i_1,i_2,i_3,i_4,i_5\}=\{1,2,3,4,5\}$ is represented by linear combination of basis in Lemma \ref{basis_K}, the basis with non-zero coefficient are multiple of $\alpha_{i_1}\beta_{i_1}$, and not multiple of $\alpha_{i_j}\beta_{i_j}$ for $j\ne 1$.
  Thus the following are linearly independent
  \begin{align*}
      &z(0,\{1,2\})z(0,\{1,3\})z(0,\{4,5\}),z(0,\{1,2\})z(0,\{1,4\})z(0,\{3,5\}),\\
      &z(0,\{1,2\})z(0,\{1,5\})z(0,\{3,4\}),z(0,\{1,2\})z(0,\{2,3\})z(0,\{4,5\}),\\
      &z(0,\{1,2\})z(0,\{2,4\})z(0,\{3,5\}),z(0,\{1,2\})z(0,\{2,5\})z(0,\{3,4\}),\\
      &z(0,\{1,3\})z(0,\{2,3\})z(0,\{4,5\}),z(0,\{1,3\})z(0,\{3,4\})z(0,\{2,5\}),\\
      &z(0,\{1,3\})z(0,\{3,5\})z(0,\{2,4\}),z(0,\{1,4\})z(0,\{2,4\})z(0,\{3,5\}),\\
      &z(0,\{1,4\})z(0,\{3,4\})z(0,\{2,5\}),z(0,\{1,4\})z(0,\{4,5\})z(0,\{2,3\}),\\
      &z(0,\{1,5\})z(0,\{2,5\})z(0,\{3,4\}),z(0,\{1,5\})z(0,\{3,5\})z(0,\{2,4\}),\\
      &z(0,\{1,5\})z(0,\{4,5\})z(0,\{2,3\}).\\
  \end{align*}
  Since the dimension of $(\mathcal{K}(5)^{D_6}/F_{4}\mathcal{K}(5)^{D_6})^{(0,6)}$ is $15$, this pair is a basis of this part.
\end{proof}

\begin{proof}[Proof of the case $(m,i,j)=(6,0,6)$]

  There is an equation
  \begin{align*}
      &z(0,\{1,2\})z(0,\{3,4\})z(0,\{5,6\})\\
      =& 8\alpha_1\alpha_2 \alpha_3 \alpha_4 \alpha_5 \alpha_6 + 4\alpha_1\alpha_2 \alpha_3 \alpha_4 \alpha_5 \beta_6 + 4\alpha_1\alpha_2 \alpha_3 \alpha_4 \beta_5 \alpha_6 + 8 \alpha_1\alpha_2 \alpha_3 \alpha_4 \beta_5 \beta_6 \\
      +& 4\alpha_1\alpha_2 \alpha_3 \beta_4 \alpha_5 \alpha_6 +2\alpha_1\alpha_2 \alpha_3 \beta_4 \alpha_5 \beta_6 +2\alpha_1\alpha_2 \alpha_3 \beta_4 \beta_5 \alpha_6 + 4 \alpha_1\alpha_2 \alpha_3 \beta_4 \beta_5 \beta_6 \\
      +& 4\alpha_1\alpha_2 \beta_3 \alpha_4 \alpha_5 \alpha_6 +2\alpha_1\alpha_2 \beta_3 \alpha_4 \alpha_5 \beta_6 +2\alpha_1\alpha_2 \beta_3 \alpha_4 \beta_5 \alpha_6 + 4 \alpha_1\alpha_2 \beta_3 \alpha_4 \beta_5 \beta_6 \\
      +& 8\alpha_1\alpha_2 \beta_3 \beta_4 \alpha_5 \alpha_6 + 4\alpha_1\alpha_2 \beta_3 \beta_4 \alpha_5 \beta_6 + 4\alpha_1\alpha_2 \beta_3 \beta_4 \beta_5 \alpha_6 + 8 \alpha_1\alpha_2 \beta_3 \beta_4 \beta_5 \beta_6 \\
      +& 4\alpha_1\beta_2 \alpha_3 \alpha_4 \alpha_5 \alpha_6 +2\alpha_1\beta_2 \alpha_3 \alpha_4 \alpha_5 \beta_6 +2\alpha_1\beta_2 \alpha_3 \alpha_4 \beta_5 \alpha_6 + 4 \alpha_1\beta_2 \alpha_3 \alpha_4 \beta_5 \beta_6 \\
      +&2\alpha_1\beta_2 \alpha_3 \beta_4 \alpha_5 \alpha_6 + \alpha_1\beta_2 \alpha_3 \beta_4 \alpha_5 \beta_6 + \alpha_1\beta_2 \alpha_3 \beta_4 \beta_5 \alpha_6 +2 \alpha_1\beta_2 \alpha_3 \beta_4 \beta_5 \beta_6 \\
      +&2\alpha_1\beta_2 \beta_3 \alpha_4 \alpha_5 \alpha_6 + \alpha_1\beta_2 \beta_3 \alpha_4 \alpha_5 \beta_6 + \alpha_1\beta_2 \beta_3 \alpha_4 \beta_5 \alpha_6 +2 \alpha_1\beta_2 \beta_3 \alpha_4 \beta_5 \beta_6 \\
      +& 4\alpha_1\beta_2 \beta_3 \beta_4 \alpha_5 \alpha_6 +2\alpha_1\beta_2 \beta_3 \beta_4 \alpha_5 \beta_6 +2\alpha_1\beta_2 \beta_3 \beta_4 \beta_5 \alpha_6 + 4 \alpha_1\beta_2 \beta_3 \beta_4 \beta_5 \beta_6 \\
      +& 4\beta_1\alpha_2 \alpha_3 \alpha_4 \alpha_5 \alpha_6 +2\beta_1\alpha_2 \alpha_3 \alpha_4 \alpha_5 \beta_6 +2\beta_1\alpha_2 \alpha_3 \alpha_4 \beta_5 \alpha_6 + 4 \beta_1\alpha_2 \alpha_3 \alpha_4 \beta_5 \beta_6 \\
      +&2\beta_1\alpha_2 \alpha_3 \beta_4 \alpha_5 \alpha_6 + \beta_1\alpha_2 \alpha_3 \beta_4 \alpha_5 \beta_6 + \beta_1\alpha_2 \alpha_3 \beta_4 \beta_5 \alpha_6 +2 \beta_1\alpha_2 \alpha_3 \beta_4 \beta_5 \beta_6 \\
      +&2\beta_1\alpha_2 \beta_3 \alpha_4 \alpha_5 \alpha_6 + \beta_1\alpha_2 \beta_3 \alpha_4 \alpha_5 \beta_6 + \beta_1\alpha_2 \beta_3 \alpha_4 \beta_5 \alpha_6 +2 \beta_1\alpha_2 \beta_3 \alpha_4 \beta_5 \beta_6 \\
      +& 4\beta_1\alpha_2 \beta_3 \beta_4 \alpha_5 \alpha_6 +2\beta_1\alpha_2 \beta_3 \beta_4 \alpha_5 \beta_6 +2\beta_1\alpha_2 \beta_3 \beta_4 \beta_5 \alpha_6 + 4 \beta_1\alpha_2 \beta_3 \beta_4 \beta_5 \beta_6 \\
      +& 8\beta_1\beta_2 \alpha_3 \alpha_4 \alpha_5 \alpha_6 + 4\beta_1\beta_2 \alpha_3 \alpha_4 \alpha_5 \beta_6 + 4\beta_1\beta_2 \alpha_3 \alpha_4 \beta_5 \alpha_6 + 8 \beta_1\beta_2 \alpha_3 \alpha_4 \beta_5 \beta_6 \\
      +& 4\beta_1\beta_2 \alpha_3 \beta_4 \alpha_5 \alpha_6 +2\beta_1\beta_2 \alpha_3 \beta_4 \alpha_5 \beta_6 +2\beta_1\beta_2 \alpha_3 \beta_4 \beta_5 \alpha_6 + 4 \beta_1\beta_2 \alpha_3 \beta_4 \beta_5 \beta_6 \\
      +& 4\beta_1\beta_2 \beta_3 \alpha_4 \alpha_5 \alpha_6 +2\beta_1\beta_2 \beta_3 \alpha_4 \alpha_5 \beta_6 +2\beta_1\beta_2 \beta_3 \alpha_4 \beta_5 \alpha_6 + 4 \beta_1\beta_2 \beta_3 \alpha_4 \beta_5 \beta_6 \\
      +& 8\beta_1\beta_2 \beta_3 \beta_4 \alpha_5 \alpha_6 + 4\beta_1\beta_2 \beta_3 \beta_4 \alpha_5 \beta_6 + 4\beta_1\beta_2 \beta_3 \beta_4 \beta_5 \alpha_6 + 8 \beta_1\beta_2 \beta_3 \beta_4 \beta_5 \beta_6, \\
  \end{align*}
  and similar elements can be computed similarly.

The following is the matrix consisted by the coefficient of 
\begin{align*}
  &\alpha_1\alpha_2 \alpha_3 \alpha_4 \alpha_5 \alpha_6,\alpha_1\alpha_2 \alpha_3 \alpha_4 \beta_5 \beta_6,\alpha_1\alpha_2 \alpha_3 \beta_4 \alpha_5 \beta_6,\alpha_1\alpha_2 \beta_3 \alpha_4 \alpha_5 \beta_6\\
  &\alpha_1\beta_2 \alpha_3 \alpha_4 \alpha_5 \beta_6,\beta_1\alpha_2 \alpha_3 \alpha_4 \alpha_5 \beta_6,\alpha_1\alpha_2 \alpha_3 \beta_4 \beta_5 \alpha_6,\alpha_1\alpha_2 \beta_3 \alpha_4 \beta_5 \alpha_6\\
  &\alpha_1\beta_2 \alpha_3 \alpha_4 \beta_5 \alpha_6,\alpha_1\alpha_2 \beta_3 \beta_4 \alpha_5 \alpha_6,\alpha_1\beta_2 \alpha_3 \beta_4 \alpha_5 \alpha_6
\end{align*}
in the $11$ elements
\begin{align*}
  &z(1,\{1,2,3,4,5,6\}),z(1,\{1,2\})z(1,\{3,4\})z(1,\{5,6\}),\\
  &z(1,\{1,2\})z(1,\{3,5\})z(1,\{4,6\}),z(1,\{1,2\})z(1,\{3,6\})z(1,\{4,5\}),\\
  &z(1,\{1,3\})z(1,\{2,4\})z(1,\{5,6\}),z(1,\{1,3\})z(1,\{2,5\})z(1,\{4,6\}),\\
  &z(1,\{1,3\})z(1,\{2,6\})z(1,\{4,5\}),z(1,\{1,4\})z(1,\{2,3\})z(1,\{5,6\}),\\
  &z(1,\{1,4\})z(1,\{2,5\})z(1,\{3,6\}),z(1,\{1,5\})z(1,\{2,3\})z(1,\{4,6\}),\\
  &z(1,\{1,6\})z(1,\{2,3\})z(1,\{4,5\}),
\end{align*}
\[
  \left(\begin{array}{cccccccccccccccccc}
    2 & 1 & 1 & 1 & 1 & 1 & 1 & 1 & 1 & 1 & 1 \\ 
    8 & 8 & 2 & 2 & 2 & 2 & 2 & 2 & 2 & 8 & 2 \\ 
    8 & 2 & 8 & 2 & 2 & 2 & 2 & 8 & 2 & 2 & 2 \\ 
    8 & 2 & 2 & 8 & 2 & 2 & 8 & 2 & 2 & 2 & 2 \\ 
    8 & 8 & 2 & 2 & 2 & 2 & 2 & 2 & 2 & 2 & 8 \\ 
    8 & 2 & 8 & 2 & 2 & 2 & 2 & 2 & 8 & 2 & 2 \\ 
    8 & 2 & 2 & 2 & 8 & 2 & 8 & 2 & 2 & 2 & 2 \\ 
    8 & 8 & 2 & 2 & 2 & 2 & 2 & 2 & 2 & 2 & 2 \\ 
    8 & 2 & 2 & 8 & 2 & 2 & 2 & 2 & 8 & 2 & 2 \\ 
    8 & 2 & 8 & 2 & 2 & 2 & 2 & 2 & 2 & 2 & 2 \\ 
    8 & 2 & 2 & 2 & 2 & 8 & 8 & 2 & 2 & 2 & 2 \\ 
  \end{array}\right).
\]
The rank of this matrix is $11$, and this means that these elements are linearly independent.
Since the dimension of $(\mathcal{K}(6)^{D_6}/F_{5}\mathcal{K}(6)^{D_6})^{(0,6)}$ is $11$, this pair is a basis of this part.

\end{proof}

\begin{proof}[Proof of the case $(m,i,j)=(3,8,4)$]

  There is an equation
  \begin{align*}
      &z(5,\{1,2\})z(1,\{1,3\})\\
      =&(-2x^4-x^3y)\alpha_1\beta_1\alpha_2\alpha_3+(-x^4+x^3y)\alpha_1\beta_1\alpha_2\beta_3\\
      +&(x^4+2x^3y)\alpha_1\beta_1\beta_2\alpha_3+(x^4+2x^3y)\alpha_1\beta_1\beta_2\beta_3
  \end{align*}
  When for $\{i_1,i_2,i_3\}=\{1,2,3\}$ $z(5,\{i_1,i_2\})z(1,\{i_1,i_3\})$ is represented by linear combination of basis in Lemma \ref{basis_K}, the basis with non-zero coefficient are multiple of $\alpha_{i_1}\beta_{i_1}$ and not multiple of $\alpha_{i_2}\beta_{i_2}$ or $\alpha_{i_3}\beta_{i_3}$.
  Thus $z(5,\{1,2\})z(1,\{1,3\}), z(5,\{1,2\})z(1,\{2,3\}),z(5,\{1,3\})z(1,\{2,3\})$ are linearly independent.
  Since the dimension of $(\mathcal{K}(3)^{D_6}/F_{2}\mathcal{K}(3)^{D_6})^{(8,4)}$ is $3$, these are a basis of this part.
  
\end{proof}

\begin{proof}[Proof of the case $(m,i,j)=(4,8,4)$]

  There is an equation
  \begin{align*}
      &z(5,\{1,2\})z(1,\{3,4\})\\
      =& 2x^3y \alpha_1\alpha_2 \alpha_3\alpha_4 + x^3y \alpha_1\alpha_2 \alpha_3\beta_4 + x^3y\alpha_1\alpha_2 \beta_3 \alpha_4 + 2x^3y \alpha_1\alpha_2 \beta_3\beta_4\\
      -&2x^4 \beta_1\beta_2 \alpha_3\alpha_4 - x^4 \beta_1\beta_2 \alpha_3\beta_4 -x^4\beta_1\beta_2 \beta_3 \alpha_4 -2x^4 \beta_1\beta_2 \beta_3\beta_4\\
      -& 2(x^4+ x^3y) \alpha_1\beta_2 \alpha_3\alpha_4 - (x^4+ x^3y) \alpha_1\beta_2 \alpha_3\beta_4 - (x^4+ x^3y)\alpha_1\beta_2 \beta_3 \alpha_4\\
      -& 2(x^4+ x^3y) \alpha_1\beta_2 \beta_3\beta_4- 2(x^4+ x^3y) \beta_1\alpha_2 \alpha_3\alpha_4 - (x^4+ x^3y) \beta_1 \alpha_2 \alpha_3\beta_4\\
      -& (x^4+ x^3y) \beta_1\alpha_2 \beta_3 \alpha_4 - 2(x^4+ x^3y) \beta_1\alpha_2 \beta_3\beta_4.
  \end{align*}
  $z(4,\{1,3\})z(0,\{2,4\})$, $z(4,\{1,4\})z(0,\{2,3\})$, $z(4,\{2,3\})z(0,\{1,4\})$, $z(4,\{2,4\})z(0,\{1,3\})$ can be computed similarly.
  We consider the following equations
  \begin{align*}
    &a_1 z(4,\{1,2\})z(0,\{3,4\})- a_1 z(4,\{1,3\})z(0,\{2,4\})+a_3 z(4,\{1,4\})z(0,\{2,3\})\\
    +& a_4 z(4,\{2,3\})z(0,\{1,4\})-a_5 z(4,\{2,4\})z(0,\{1,3\})=0.
  \end{align*}
  By the coefficient of the $\alpha_1\alpha_2 \alpha_3\beta_4,\alpha_1\alpha_2 \beta_3\alpha_4,\alpha_1\beta_2 \alpha_3\alpha_4,\alpha_1\beta_2 \alpha_3\beta_4$, we obtain the following equations
  \begin{align*}
    &a_1x^3y+ a_2x^3y - 2a_3(x^4+ x^3y) + a_4x^3y -2a_5(x^4+ x^3y)=0\\
    &a_1x^3y-2 a_2(x^4+ x^3y) + a_3x^3y - 2a_4(x^4+ x^3y) + a_5x^3y =0\\
    -&2a_1(x^4+ x^3y)+a_2x^3y +a_3x^3y - 2a_4(x^4+ x^3y) - 2a_5(x^4+ x^3y) =0\\
    -&a_1(x^4+ x^3y)+ 2a_2x^3y -a_3(x^4+ x^3y) -a_4(x^4+ x^3y) - 2a_5x^4 =0
  \end{align*}
  Then 
  \begin{align*}
    &-2a_3-2a_5=0\\
    &-2a_2-2a_4=0\\
    &-2a_1-2a_3-2a_5=0\\
    &-a_1+a_2-2a_3-a_4=0\\
    &-a_1-a_3-a_4-2a_5=0,\\
  \end{align*}
  and we obtain $a_1=a_2=a_3=a_4=a_5=0$.
  Since the dimension of $(\mathcal{K}(4)^{D_6}/F_{3}\mathcal{K}(4)^{D_6})^{(8,4)}$ is $5$, this pair is a basis of this part.

\end{proof}

By combining these proof we have proved Lemma \ref{other_cases_generator}.
\begin{proof}
  [Proof of Theorem \ref{G_2_generator}]
  By Lemma \ref{basis_to_check} and \ref{other_cases_generator}, we can obtain this theorem.
\end{proof}

\section{Computation of Hilbert-Poincar\'e series of $\Hom(\Z^2,G)_0$}\label{sec:computation_of_Poincare}

We calculate the Hilbert-Poincar\'e series of $\Hom(\Z^3,G)_1$ for $G=F_4,$ by Theorem \ref{bigraded_Poincare_formula} with an assistance of computer. In this section, we only record the result.

\begin{align*}
  &P(\Hom(\Z^3,F_4)_0;s,t)={{s}^{48}} {{t}^{12}}+3 {{s}^{46}} {{t}^{11}}+3 {{s}^{38}} {{t}^{11}}+3 {{s}^{34}} {{t}^{11}}+3 {{s}^{26}} {{t}^{11}}\\
  +&3 {{s}^{48}} {{t}^{10}}+3 {{s}^{44}} {{t}^{10}}+3 {{s}^{40}} {{t}^{10}}+12 {{s}^{36}} {{t}^{10}}+9 {{s}^{32}} {{t}^{10}}+6 {{s}^{28}} {{t}^{10}}+15 {{s}^{24}} {{t}^{10}}\\
  +&3 {{s}^{20}} {{t}^{10}}+6 {{s}^{16}} {{t}^{10}}+6 {{s}^{12}} {{t}^{10}}+9 {{s}^{46}} {{t}^{9}}+2 {{s}^{42}} {{t}^{9}}+19 {{s}^{38}} {{t}^{9}}+27 {{s}^{34}} {{t}^{9}}\\
  +&20{{s}^{30}} {{t}^{9}}+35 {{s}^{26}} {{t}^{9}}+36 {{s}^{22}} {{t}^{9}}+18 {{s}^{18}} {{t}^{9}}+26 {{s}^{14}} {{t}^{9}}+18 {{s}^{10}} {{t}^{9}}+10 {{s}^{2}} {{t}^{9}}\\
  +&6 {{s}^{48}} {{t}^{8}}+9 {{s}^{44}} {{t}^{8}}+12 {{s}^{40}} {{t}^{8}}+48 {{s}^{36}} {{t}^{8}}+51 {{s}^{32}} {{t}^{8}}+51 {{s}^{28}} {{t}^{8}}+96 {{s}^{24}} {{t}^{8}}\\
  +&57 {{s}^{20}} {{t}^{8}}+57 {{s}^{16}} {{t}^{8}}+60 {{s}^{12}} {{t}^{8}}+18 {{s}^{8}} {{t}^{8}}+15 {{s}^{4}} {{t}^{8}}+15 {{t}^{8}}+18 {{s}^{46}} {{t}^{7}}\\
  +&6 {{s}^{42}} {{t}^{7}}+51 {{s}^{38}} {{t}^{7}}+84 {{s}^{34}} {{t}^{7}}+78 {{s}^{30}} {{t}^{7}}+123 {{s}^{26}} {{t}^{7}}+135 {{s}^{22}} {{t}^{7}}+84 {{s}^{18}} {{t}^{7}}\\
  +&102 {{s}^{14}} {{t}^{7}}+69 {{s}^{10}} {{t}^{7}}+12 {{s}^{6}} {{t}^{7}}+30 {{s}^{2}} {{t}^{7}}+10 {{s}^{48}} {{t}^{6}}+18 {{s}^{44}} {{t}^{6}}+27 {{s}^{40}} {{t}^{6}}\\
  +&100 {{s}^{36}} {{t}^{6}}+108 {{s}^{32}} {{t}^{6}}+108 {{s}^{28}} {{t}^{6}}+182 {{s}^{24}} {{t}^{6}}+108 {{s}^{20}} {{t}^{6}}+108{{s}^{16}} {{t}^{6}}\\
  +&100 {{s}^{12}} {{t}^{6}}+27 {{s}^{8}} {{t}^{6}}+18 {{s}^{4}} {{t}^{6}}+10 {{t}^{6}}+30 {{s}^{46}} {{t}^{5}}+12 {{s}^{42}} {{t}^{5}}+69 {{s}^{38}} {{t}^{5}}\\
  +&102 {{s}^{34}} {{t}^{5}}+84 {{s}^{30}} {{t}^{5}}+135 {{s}^{26}} {{t}^{5}}+123 {{s}^{22}} {{t}^{5}}+78 {{s}^{18}} {{t}^{5}}+84 {{s}^{14}} {{t}^{5}}\\
  +&51 {{s}^{10}} {{t}^{5}}+6 {{s}^{6}} {{t}^{5}}+18 {{s}^{2}} {{t}^{5}}+15 {{s}^{48}} {{t}^{4}}+15 {{s}^{44}} {{t}^{4}}+18 {{s}^{40}} {{t}^{4}}+60 {{s}^{36}} {{t}^{4}}\\
  +&57 {{s}^{32}} {{t}^{4}}+57 {{s}^{28}} {{t}^{4}}+96 {{s}^{24}} {{t}^{4}}+51 {{s}^{20}} {{t}^{4}}+51 {{s}^{16}} {{t}^{4}}+48 {{s}^{12}} {{t}^{4}}+12 {{s}^{8}} {{t}^{4}}\\
  +&9 {{s}^{4}} {{t}^{4}}+6 {{t}^{4}}+10 {{s}^{46}} {{t}^{3}}+18 {{s}^{38}} {{t}^{3}}+26 {{s}^{34}} {{t}^{3}}+18 {{s}^{30}} {{t}^{3}}+36 {{s}^{26}} {{t}^{3}}+35 {{s}^{22}} {{t}^{3}}\\
  +&\underline{20 {{s}^{18}} {{t}^{3}}}+27 {{s}^{14}} {{t}^{3}}+19 {{s}^{10}} {{t}^{3}}+2 {{s}^{6}} {{t}^{3}}+9 {{s}^{2}} {{t}^{3}}+6 {{s}^{36}}{{t}^{2}}+6 {{s}^{32}} {{t}^{2}}+3 {{s}^{28}} {{t}^{2}}\\
  +&15 {{s}^{24}} {{t}^{2}}+6 {{s}^{20}} {{t}^{2}}+9 {{s}^{16}} {{t}^{2}}+12 {{s}^{12}} {{t}^{2}}+3 {{s}^{8}} {{t}^{2}}+3 {{s}^{4}} {{t}^{2}}+3 {{t}^{2}}+3 {{s}^{22}} t\\
  +&3 {{s}^{14}} t+3 {{s}^{10}} t+3 {{s}^{2}} t+1
\end{align*}

\begin{align*}
  &P(\Hom(\Z^3,E_6)_0;s,t)={{s}^{72}} {{t}^{18}}+3 {{s}^{70}} {{t}^{17}}+3 {{s}^{64}} {{t}^{17}}+3 {{s}^{62}} {{t}^{17}}+3 {{s}^{58}} {{t}^{17}}+3 {{s}^{56}} {{t}^{17}}\\
  +&3 {{s}^{50}} {{t}^{17}}+3 {{s}^{72}} {{t}^{16}}+3 {{s}^{68}} {{t}^{16}}+3 {{s}^{66}} {{t}^{16}}+3 {{s}^{64}} {{t}^{16}}+9 {{s}^{62}} {{t}^{16}}+12 {{s}^{60}} {{t}^{16}}+3 {{s}^{58}} {{t}^{16}}\\
  +&12 {{s}^{56}} {{t}^{16}}+18 {{s}^{54}} {{t}^{16}}+6 {{s}^{52}} {{t}^{16}}+9 {{s}^{50}} {{t}^{16}}+24 {{s}^{48}} {{t}^{16}}+9 {{s}^{46}} {{t}^{16}}+3 {{s}^{44}} {{t}^{16}}+15 {{s}^{42}} {{t}^{16}}\\
  +&9 {{s}^{40}}{{t}^{16}}+6 {{s}^{36}} {{t}^{16}}+6 {{s}^{34}} {{t}^{16}}+9 {{s}^{70}} {{t}^{15}}+{{s}^{68}} {{t}^{15}}+2 {{s}^{66}} {{t}^{15}}+18 {{s}^{64}} {{t}^{15}}+19 {{s}^{62}} {{t}^{15}}\\
  +&11 {{s}^{60}} {{t}^{15}}+36 {{s}^{58}} {{t}^{15}}+36 {{s}^{56}} {{t}^{15}}+29 {{s}^{54}} {{t}^{15}}+54 {{s}^{52}} {{t}^{15}}+53 {{s}^{50}} {{t}^{15}}+46 {{s}^{48}} {{t}^{15}}\\
  +&72 {{s}^{46}} {{t}^{15}}+61 {{s}^{44}} {{t}^{15}}+36 {{s}^{42}} {{t}^{15}}+72 {{s}^{40}} {{t}^{15}}+62 {{s}^{38}} {{t}^{15}}+25 {{s}^{36}} {{t}^{15}}+45 {{s}^{34}} {{t}^{15}}\\
  +&54 {{s}^{32}}{{t}^{15}}+9 {{s}^{30}} {{t}^{15}}+18 {{s}^{28}} {{t}^{15}}+28 {{s}^{26}} {{t}^{15}}+10 {{s}^{24}} {{t}^{15}}+10 {{s}^{20}} {{t}^{15}}+6 {{s}^{72}} {{t}^{14}}\\
  +&9 {{s}^{68}} {{t}^{14}}+12 {{s}^{66}} {{t}^{14}}+12 {{s}^{64}} {{t}^{14}}+39 {{s}^{62}} {{t}^{14}}+57 {{s}^{60}} {{t}^{14}}+33 {{s}^{58}} {{t}^{14}}+87 {{s}^{56}} {{t}^{14}}\\
  +&132 {{s}^{54}} {{t}^{14}}+84 {{s}^{52}} {{t}^{14}}+138 {{s}^{50}} {{t}^{14}}+213 {{s}^{48}} {{t}^{14}}+159 {{s}^{46}} {{t}^{14}}+165 {{s}^{44}} {{t}^{14}}+252 {{s}^{42}} {{t}^{14}}\\
  +&192 {{s}^{40}}{{t}^{14}}+189 {{s}^{38}} {{t}^{14}}+234 {{s}^{36}} {{t}^{14}}+180 {{s}^{34}} {{t}^{14}}+147 {{s}^{32}} {{t}^{14}}+192 {{s}^{30}} {{t}^{14}}+108 {{s}^{28}} {{t}^{14}}\\
  +&90 {{s}^{26}} {{t}^{14}}+123 {{s}^{24}} {{t}^{14}}+69 {{s}^{22}} {{t}^{14}}+18 {{s}^{20}} {{t}^{14}}+60 {{s}^{18}} {{t}^{14}}+30 {{s}^{16}} {{t}^{14}}+15 {{s}^{12}} {{t}^{14}}\\
  +&15 {{s}^{10}} {{t}^{14}}+18 {{s}^{70}} {{t}^{13}}+3 {{s}^{68}} {{t}^{13}}+6 {{s}^{66}} {{t}^{13}}+51 {{s}^{64}} {{t}^{13}}+54 {{s}^{62}} {{t}^{13}}+51 {{s}^{60}} {{t}^{13}}\\
  +&141{{s}^{58}} {{t}^{13}}+156 {{s}^{56}} {{t}^{13}}+168 {{s}^{54}} {{t}^{13}}+303 {{s}^{52}} {{t}^{13}}+309 {{s}^{50}} {{t}^{13}}+330 {{s}^{48}} {{t}^{13}}+507 {{s}^{46}} {{t}^{13}}\\
  +&492 {{s}^{44}} {{t}^{13}}+423 {{s}^{42}} {{t}^{13}}+645 {{s}^{40}} {{t}^{13}}+618 {{s}^{38}} {{t}^{13}}+459 {{s}^{36}} {{t}^{13}}+603 {{s}^{34}} {{t}^{13}}+621 {{s}^{32}} {{t}^{13}}\\
  +&390 {{s}^{30}} {{t}^{13}}+447 {{s}^{28}} {{t}^{13}}+459 {{s}^{26}} {{t}^{13}}+282 {{s}^{24}} {{t}^{13}}+264 {{s}^{22}} {{t}^{13}}+255 {{s}^{20}} {{t}^{13}}+129 {{s}^{18}} {{t}^{13}}\\
  +&126 {{s}^{16}} {{t}^{13}}+117 {{s}^{14}} {{t}^{13}}+30 {{s}^{12}} {{t}^{13}}+45 {{s}^{10}} {{t}^{13}}+45 {{s}^{8}} {{t}^{13}}+21 {{s}^{2}} {{t}^{13}}+10 {{s}^{72}} {{t}^{12}}\\
  +&18 {{s}^{68}} {{t}^{12}}+28 {{s}^{66}} {{t}^{12}}+27 {{s}^{64}} {{t}^{12}}+101 {{s}^{62}} {{t}^{12}}+155 {{s}^{60}} {{t}^{12}}+118 {{s}^{58}} {{t}^{12}}+290 {{s}^{56}} {{t}^{12}}\\
  +&446 {{s}^{54}} {{t}^{12}}+369 {{s}^{52}} {{t}^{12}}+585 {{s}^{50}} {{t}^{12}}+867 {{s}^{48}} {{t}^{12}}+764 {{s}^{46}} {{t}^{12}}+916{{s}^{44}} {{t}^{12}}+1225 {{s}^{42}} {{t}^{12}}\\
  +&1088 {{s}^{40}} {{t}^{12}}+1166 {{s}^{38}} {{t}^{12}}+1395 {{s}^{36}} {{t}^{12}}+1160 {{s}^{34}} {{t}^{12}}+1141 {{s}^{32}} {{t}^{12}}+1295 {{s}^{30}} {{t}^{12}}\\
  +&964 {{s}^{28}} {{t}^{12}}+828 {{s}^{26}} {{t}^{12}}+981 {{s}^{24}} {{t}^{12}}+649 {{s}^{22}} {{t}^{12}}+434 {{s}^{20}} {{t}^{12}}+533 {{s}^{18}} {{t}^{12}}+361 {{s}^{16}} {{t}^{12}}\\
  +&146 {{s}^{14}} {{t}^{12}}+216 {{s}^{12}} {{t}^{12}}+135 {{s}^{10}} {{t}^{12}}+45 {{s}^{8}} {{t}^{12}}+45 {{s}^{6}} {{t}^{12}}+35 {{s}^{4}} {{t}^{12}}+28 {{t}^{12}}+30 {{s}^{70}} {{t}^{11}}\\
  +&6 {{s}^{68}} {{t}^{11}}+12 {{s}^{66}} {{t}^{11}}+105 {{s}^{64}} {{t}^{11}}+117 {{s}^{62}} {{t}^{11}}+132 {{s}^{60}} {{t}^{11}}+354 {{s}^{58}} {{t}^{11}}+408 {{s}^{56}} {{t}^{11}}\\
  +&501 {{s}^{54}} {{t}^{11}}+870 {{s}^{52}} {{t}^{11}}+951 {{s}^{50}} {{t}^{11}}+1068 {{s}^{48}} {{t}^{11}}+1620 {{s}^{46}} {{t}^{11}}+1644 {{s}^{44}} {{t}^{11}}\\
  +&1590 {{s}^{42}} {{t}^{11}}+2211 {{s}^{40}} {{t}^{11}}+2223 {{s}^{38}} {{t}^{11}}+1857{{s}^{36}} {{t}^{11}}+2304 {{s}^{34}} {{t}^{11}}+2313 {{s}^{32}} {{t}^{11}}\\
  +&1737 {{s}^{30}} {{t}^{11}}+1863 {{s}^{28}} {{t}^{11}}+1848 {{s}^{26}} {{t}^{11}}+1272 {{s}^{24}} {{t}^{11}}+1212 {{s}^{22}} {{t}^{11}}+1098 {{s}^{20}} {{t}^{11}}\\
  +&666 {{s}^{18}} {{t}^{11}}+588 {{s}^{16}} {{t}^{11}}+522 {{s}^{14}} {{t}^{11}}+207 {{s}^{12}} {{t}^{11}}+204 {{s}^{10}} {{t}^{11}}+183 {{s}^{8}} {{t}^{11}}+30 {{s}^{6}} {{t}^{11}}\\
  +&15 {{s}^{4}} {{t}^{11}}+63 {{s}^{2}} {{t}^{11}}+15 {{s}^{72}} {{t}^{10}}+30 {{s}^{68}} {{t}^{10}}+51 {{s}^{66}} {{t}^{10}}+51 {{s}^{64}} {{t}^{10}}+198 {{s}^{62}} {{t}^{10}}\\
  +&315 {{s}^{60}} {{t}^{10}}+270 {{s}^{58}} {{t}^{10}}+633 {{s}^{56}} {{t}^{10}}+969 {{s}^{54}} {{t}^{10}}+882 {{s}^{52}} {{t}^{10}}+1365 {{s}^{50}} {{t}^{10}}+1959 {{s}^{48}} {{t}^{10}}\\
  +&1836 {{s}^{46}} {{t}^{10}}+2211 {{s}^{44}} {{t}^{10}}+2877 {{s}^{42}} {{t}^{10}}+2643 {{s}^{40}} {{t}^{10}}+2847 {{s}^{38}} {{t}^{10}}+3309 {{s}^{36}} {{t}^{10}}\\
  +&2868 {{s}^{34}} {{t}^{10}}+2772 {{s}^{32}} {{t}^{10}}+3072 {{s}^{30}}{{t}^{10}}+2373 {{s}^{28}} {{t}^{10}}+2052 {{s}^{26}} {{t}^{10}}+2235 {{s}^{24}} {{t}^{10}}\\
  +&1581 {{s}^{22}} {{t}^{10}}+1059 {{s}^{20}} {{t}^{10}}+1194 {{s}^{18}} {{t}^{10}}+789 {{s}^{16}} {{t}^{10}}+372 {{s}^{14}} {{t}^{10}}+426 {{s}^{12}} {{t}^{10}}\\
  +&279 {{s}^{10}} {{t}^{10}}+78 {{s}^{8}} {{t}^{10}}+81 {{s}^{6}} {{t}^{10}}+45 {{s}^{4}} {{t}^{10}}+21 {{t}^{10}}+45 {{s}^{70}} {{t}^{9}}+10 {{s}^{68}} {{t}^{9}}+20 {{s}^{66}} {{t}^{9}}\\
  +&180 {{s}^{64}} {{t}^{9}}+199 {{s}^{62}} {{t}^{9}}+246 {{s}^{60}} {{t}^{9}}+603 {{s}^{58}} {{t}^{9}}+720 {{s}^{56}} {{t}^{9}}+866 {{s}^{54}} {{t}^{9}}+1486 {{s}^{52}} {{t}^{9}}\\
  +&1619 {{s}^{50}} {{t}^{9}}+1831 {{s}^{48}} {{t}^{9}}+2638 {{s}^{46}} {{t}^{9}}+2734 {{s}^{44}} {{t}^{9}}+2610 {{s}^{42}} {{t}^{9}}+3519 {{s}^{40}} {{t}^{9}}+3500 {{s}^{38}} {{t}^{9}}\\
  +&2968 {{s}^{36}} {{t}^{9}}+3500 {{s}^{34}} {{t}^{9}}+3519 {{s}^{32}} {{t}^{9}}+2610 {{s}^{30}} {{t}^{9}}+2734 {{s}^{28}} {{t}^{9}}+2638 {{s}^{26}} {{t}^{9}}+1831 {{s}^{24}} {{t}^{9}}\\
  +&1619 {{s}^{22}} {{t}^{9}}+1486{{s}^{20}} {{t}^{9}}+866 {{s}^{18}} {{t}^{9}}+720 {{s}^{16}} {{t}^{9}}+603 {{s}^{14}} {{t}^{9}}+246 {{s}^{12}} {{t}^{9}}+199 {{s}^{10}} {{t}^{9}}\\
  +&180 {{s}^{8}} {{t}^{9}}+20 {{s}^{6}} {{t}^{9}}+10 {{s}^{4}} {{t}^{9}}+45 {{s}^{2}} {{t}^{9}}+21 {{s}^{72}} {{t}^{8}}+45 {{s}^{68}} {{t}^{8}}+81 {{s}^{66}} {{t}^{8}}+78 {{s}^{64}} {{t}^{8}}\\
  +&279 {{s}^{62}} {{t}^{8}}+426 {{s}^{60}} {{t}^{8}}+372 {{s}^{58}} {{t}^{8}}+789 {{s}^{56}} {{t}^{8}}+1194 {{s}^{54}} {{t}^{8}}+1059 {{s}^{52}} {{t}^{8}}+1581 {{s}^{50}} {{t}^{8}}\\
  +&2235 {{s}^{48}} {{t}^{8}}+2052 {{s}^{46}} {{t}^{8}}+2373 {{s}^{44}} {{t}^{8}}+3072 {{s}^{42}} {{t}^{8}}+2772 {{s}^{40}} {{t}^{8}}+2868 {{s}^{38}} {{t}^{8}}+3309 {{s}^{36}} {{t}^{8}}\\
  +&2847 {{s}^{34}} {{t}^{8}}+2643 {{s}^{32}} {{t}^{8}}+2877 {{s}^{30}} {{t}^{8}}+2211 {{s}^{28}} {{t}^{8}}+1836 {{s}^{26}} {{t}^{8}}+1959 {{s}^{24}} {{t}^{8}}+1365 {{s}^{22}} {{t}^{8}}\\
  +&882 {{s}^{20}} {{t}^{8}}+969 {{s}^{18}} {{t}^{8}}+633 {{s}^{16}} {{t}^{8}}+270 {{s}^{14}} {{t}^{8}}+315 {{s}^{12}} {{t}^{8}}+198 {{s}^{10}} {{t}^{8}}+51 {{s}^{8}} {{t}^{8}}+51 {{s}^{6}} {{t}^{8}}\\
  +&30 {{s}^{4}} {{t}^{8}}+15 {{t}^{8}}+63 {{s}^{70}} {{t}^{7}}+15 {{s}^{68}} {{t}^{7}}+30 {{s}^{66}} {{t}^{7}}+183 {{s}^{64}} {{t}^{7}}+204 {{s}^{62}} {{t}^{7}}+207 {{s}^{60}} {{t}^{7}}\\
  +&522 {{s}^{58}} {{t}^{7}}+588 {{s}^{56}} {{t}^{7}}+666 {{s}^{54}} {{t}^{7}}+1098 {{s}^{52}} {{t}^{7}}+1212 {{s}^{50}} {{t}^{7}}+1272 {{s}^{48}} {{t}^{7}}+1848 {{s}^{46}} {{t}^{7}}\\
  +&1863 {{s}^{44}} {{t}^{7}}+1737 {{s}^{42}} {{t}^{7}}+2313 {{s}^{40}} {{t}^{7}}+2304 {{s}^{38}} {{t}^{7}}+1857 {{s}^{36}} {{t}^{7}}+2223 {{s}^{34}} {{t}^{7}}+2211 {{s}^{32}} {{t}^{7}}\\
  +&1590 {{s}^{30}} {{t}^{7}}+1644 {{s}^{28}} {{t}^{7}}+1620 {{s}^{26}} {{t}^{7}}+1068 {{s}^{24}} {{t}^{7}}+951 {{s}^{22}} {{t}^{7}}+870 {{s}^{20}} {{t}^{7}}+501 {{s}^{18}} {{t}^{7}}\\
  +&408 {{s}^{16}} {{t}^{7}}+354 {{s}^{14}} {{t}^{7}}+132 {{s}^{12}} {{t}^{7}}+117 {{s}^{10}} {{t}^{7}}+105 {{s}^{8}} {{t}^{7}}+12 {{s}^{6}} {{t}^{7}}+6 {{s}^{4}} {{t}^{7}}+30 {{s}^{2}} {{t}^{7}}\\
  +&28 {{s}^{72}} {{t}^{6}}+35 {{s}^{68}} {{t}^{6}}+45 {{s}^{66}} {{t}^{6}}+45 {{s}^{64}} {{t}^{6}}+135 {{s}^{62}} {{t}^{6}}+216 {{s}^{60}} {{t}^{6}}+146 {{s}^{58}} {{t}^{6}}+361 {{s}^{56}} {{t}^{6}}\\
  +&533 {{s}^{54}} {{t}^{6}}+434 {{s}^{52}} {{t}^{6}}+649 {{s}^{50}} {{t}^{6}}+981 {{s}^{48}} {{t}^{6}}+828 {{s}^{46}} {{t}^{6}}+964 {{s}^{44}} {{t}^{6}}+1295 {{s}^{42}} {{t}^{6}}\\
  +&1141 {{s}^{40}} {{t}^{6}}+1160 {{s}^{38}} {{t}^{6}}+1395 {{s}^{36}} {{t}^{6}}+1166 {{s}^{34}} {{t}^{6}}+1088 {{s}^{32}} {{t}^{6}}+1225 {{s}^{30}} {{t}^{6}}+916 {{s}^{28}} {{t}^{6}}\\
  +&764 {{s}^{26}} {{t}^{6}}+867 {{s}^{24}} {{t}^{6}}+585 {{s}^{22}} {{t}^{6}}+369 {{s}^{20}} {{t}^{6}}+446 {{s}^{18}} {{t}^{6}}+290 {{s}^{16}} {{t}^{6}}+118 {{s}^{14}} {{t}^{6}}\\
  +&155 {{s}^{12}} {{t}^{6}}+101 {{s}^{10}} {{t}^{6}}+27 {{s}^{8}} {{t}^{6}}+28 {{s}^{6}} {{t}^{6}}+18 {{s}^{4}} {{t}^{6}}+10 {{t}^{6}}+21 {{s}^{70}} {{t}^{5}}+45 {{s}^{64}} {{t}^{5}}+45 {{s}^{62}} {{t}^{5}}\\
  +&30 {{s}^{60}} {{t}^{5}}+117 {{s}^{58}} {{t}^{5}}+126 {{s}^{56}} {{t}^{5}}+129 {{s}^{54}} {{t}^{5}}+255 {{s}^{52}} {{t}^{5}}+264 {{s}^{50}} {{t}^{5}}+282 {{s}^{48}} {{t}^{5}}+459 {{s}^{46}} {{t}^{5}}\\
  +&447 {{s}^{44}} {{t}^{5}}+390{{s}^{42}} {{t}^{5}}+621 {{s}^{40}} {{t}^{5}}+603 {{s}^{38}} {{t}^{5}}+459 {{s}^{36}} {{t}^{5}}+618 {{s}^{34}} {{t}^{5}}+645 {{s}^{32}} {{t}^{5}}\\
  +&423 {{s}^{30}} {{t}^{5}}+492 {{s}^{28}} {{t}^{5}}+507 {{s}^{26}} {{t}^{5}}+330 {{s}^{24}} {{t}^{5}}+309 {{s}^{22}} {{t}^{5}}+303 {{s}^{20}} {{t}^{5}}+168 {{s}^{18}} {{t}^{5}}\\
  +&156 {{s}^{16}} {{t}^{5}}+141 {{s}^{14}} {{t}^{5}}+51 {{s}^{12}} {{t}^{5}}+54 {{s}^{10}} {{t}^{5}}+51 {{s}^{8}} {{t}^{5}}+6 {{s}^{6}} {{t}^{5}}+3 {{s}^{4}} {{t}^{5}}+18 {{s}^{2}} {{t}^{5}}\\
  +&15 {{s}^{62}} {{t}^{4}}+15 {{s}^{60}}{{t}^{4}}+30 {{s}^{56}} {{t}^{4}}+60 {{s}^{54}} {{t}^{4}}+18 {{s}^{52}} {{t}^{4}}+69 {{s}^{50}} {{t}^{4}}+123 {{s}^{48}} {{t}^{4}}+90 {{s}^{46}} {{t}^{4}}\\
  +&108 {{s}^{44}} {{t}^{4}}+192 {{s}^{42}} {{t}^{4}}+147 {{s}^{40}} {{t}^{4}}+180 {{s}^{38}} {{t}^{4}}+234 {{s}^{36}} {{t}^{4}}+189 {{s}^{34}} {{t}^{4}}+192 {{s}^{32}} {{t}^{4}}\\
  +&252 {{s}^{30}} {{t}^{4}}+165 {{s}^{28}} {{t}^{4}}+159 {{s}^{26}} {{t}^{4}}+213 {{s}^{24}} {{t}^{4}}+138 {{s}^{22}} {{t}^{4}}+84 {{s}^{20}} {{t}^{4}}+132 {{s}^{18}} {{t}^{4}}\\
  +&87 {{s}^{16}} {{t}^{4}}+33{{s}^{14}} {{t}^{4}}+57 {{s}^{12}} {{t}^{4}}+\underline{39 {{s}^{10}} {{t}^{4}}}+12 {{s}^{8}} {{t}^{4}}+12 {{s}^{6}} {{t}^{4}}+9 {{s}^{4}} {{t}^{4}}+6 {{t}^{4}}+10 {{s}^{52}} {{t}^{3}}\\
  +&10 {{s}^{48}} {{t}^{3}}+28 {{s}^{46}} {{t}^{3}}+18 {{s}^{44}} {{t}^{3}}+9 {{s}^{42}} {{t}^{3}}+54 {{s}^{40}} {{t}^{3}}+45 {{s}^{38}} {{t}^{3}}+25 {{s}^{36}} {{t}^{3}}+62 {{s}^{34}} {{t}^{3}}\\
  +&72 {{s}^{32}} {{t}^{3}}+36 {{s}^{30}} {{t}^{3}}+61 {{s}^{28}} {{t}^{3}}+72 {{s}^{26}} {{t}^{3}}+46 {{s}^{24}} {{t}^{3}}+53 {{s}^{22}} {{t}^{3}}+54 {{s}^{20}} {{t}^{3}}+29 {{s}^{18}} {{t}^{3}}\\
  +&36{{s}^{16}} {{t}^{3}}+36 {{s}^{14}} {{t}^{3}}+11 {{s}^{12}} {{t}^{3}}+19 {{s}^{10}} {{t}^{3}}+18 {{s}^{8}} {{t}^{3}}+2 {{s}^{6}} {{t}^{3}}+{{s}^{4}} {{t}^{3}}+9 {{s}^{2}} {{t}^{3}}+6 {{s}^{38}} {{t}^{2}}\\
  +&6 {{s}^{36}} {{t}^{2}}+9 {{s}^{32}} {{t}^{2}}+15 {{s}^{30}} {{t}^{2}}+3 {{s}^{28}} {{t}^{2}}+9 {{s}^{26}} {{t}^{2}}+24 {{s}^{24}} {{t}^{2}}+9 {{s}^{22}} {{t}^{2}}+6 {{s}^{20}} {{t}^{2}}+18 {{s}^{18}} {{t}^{2}}\\
  +&12 {{s}^{16}} {{t}^{2}}+3 {{s}^{14}} {{t}^{2}}+12 {{s}^{12}} {{t}^{2}}+9 {{s}^{10}} {{t}^{2}}+3 {{s}^{8}} {{t}^{2}}+3 {{s}^{6}} {{t}^{2}}+3 {{s}^{4}} {{t}^{2}}+3 {{t}^{2}}+3{{s}^{22}} t+3 {{s}^{16}} t\\
  +&3 {{s}^{14}} t+3 {{s}^{10}} t+3 {{s}^{8}} t+3 {{s}^{2}} t+1
\end{align*}

\begin{align*}
  &P(\Hom(\Z^3,E_7)_0;s,t)={{s}^{126}} {{t}^{21}}+3 {{s}^{124}} {{t}^{20}}+3 {{s}^{116}} {{t}^{20}}+3 {{s}^{112}} {{t}^{20}}+3 {{s}^{108}} {{t}^{20}}\\
  +&3 {{s}^{104}} {{t}^{20}}+3 {{s}^{100}} {{t}^{20}}+3 {{s}^{92}} {{t}^{20}}+3 {{s}^{126}} {{t}^{19}}+3 {{s}^{122}} {{t}^{19}}+3 {{s}^{118}} {{t}^{19}}+12 {{s}^{114}} {{t}^{19}}\\
  +&12 {{s}^{110}} {{t}^{19}}+15 {{s}^{106}} {{t}^{19}}+21 {{s}^{102}} {{t}^{19}}+21 {{s}^{98}} {{t}^{19}}+21 {{s}^{94}} {{t}^{19}}+27 {{s}^{90}} {{t}^{19}}+18 {{s}^{86}} {{t}^{19}}\\
  +&18 {{s}^{82}} {{t}^{19}}+15 {{s}^{78}} {{t}^{19}}+9 {{s}^{74}} {{t}^{19}}+6 {{s}^{70}} {{t}^{19}}+6 {{s}^{66}} {{t}^{19}}+9 {{s}^{124}} {{t}^{18}}+2 {{s}^{120}} {{t}^{18}}\\
  +&19 {{s}^{116}} {{t}^{18}}+28 {{s}^{112}} {{t}^{18}}+38 {{s}^{108}} {{t}^{18}}+55 {{s}^{104}} {{t}^{18}}+82 {{s}^{100}} {{t}^{18}}+83 {{s}^{96}} {{t}^{18}}+117 {{s}^{92}} {{t}^{18}}\\
  +&116 {{s}^{88}} {{t}^{18}}+126 {{s}^{84}} {{t}^{18}}+124 {{s}^{80}} {{t}^{18}}+125 {{s}^{76}} {{t}^{18}}+98 {{s}^{72}} {{t}^{18}}+98 {{s}^{68}} {{t}^{18}}+71 {{s}^{64}} {{t}^{18}}\\
  +&55 {{s}^{60}} {{t}^{18}}+36 {{s}^{56}}{{t}^{18}}+28 {{s}^{52}} {{t}^{18}}+10 {{s}^{48}} {{t}^{18}}+10 {{s}^{44}} {{t}^{18}}+6 {{s}^{126}} {{t}^{17}}+9 {{s}^{122}} {{t}^{17}}\\
  +&12 {{s}^{118}} {{t}^{17}}+48 {{s}^{114}} {{t}^{17}}+63 {{s}^{110}} {{t}^{17}}+105 {{s}^{106}} {{t}^{17}}+168 {{s}^{102}} {{t}^{17}}+228 {{s}^{98}} {{t}^{17}}\\
  +&285 {{s}^{94}} {{t}^{17}}+384 {{s}^{90}} {{t}^{17}}+420 {{s}^{86}} {{t}^{17}}+492 {{s}^{82}} {{t}^{17}}+522 {{s}^{78}} {{t}^{17}}+534 {{s}^{74}} {{t}^{17}}+507 {{s}^{70}} {{t}^{17}}\\
  +&501 {{s}^{66}} {{t}^{17}}+414 {{s}^{62}} {{t}^{17}}+372 {{s}^{58}} {{t}^{17}}+285 {{s}^{54}} {{t}^{17}}+228 {{s}^{50}} {{t}^{17}}+153 {{s}^{46}} {{t}^{17}}+114 {{s}^{42}} {{t}^{17}}\\
  +&60 {{s}^{38}} {{t}^{17}}+45 {{s}^{34}} {{t}^{17}}+15 {{s}^{30}} {{t}^{17}}+15 {{s}^{26}} {{t}^{17}}+18 {{s}^{124}} {{t}^{16}}+6 {{s}^{120}} {{t}^{16}}+51 {{s}^{116}} {{t}^{16}}+87 {{s}^{112}} {{t}^{16}}\\
  +&153 {{s}^{108}} {{t}^{16}}+252 {{s}^{104}} {{t}^{16}}+426 {{s}^{100}} {{t}^{16}}+552 {{s}^{96}} {{t}^{16}}+819 {{s}^{92}} {{t}^{16}}+1011 {{s}^{88}} {{t}^{16}}+1257{{s}^{84}} {{t}^{16}}\\
  +&1446 {{s}^{80}} {{t}^{16}}+1644 {{s}^{76}} {{t}^{16}}+1689 {{s}^{72}} {{t}^{16}}+1779 {{s}^{68}} {{t}^{16}}+1683 {{s}^{64}} {{t}^{16}}+1608 {{s}^{60}} {{t}^{16}}\\
  +&1404 {{s}^{56}} {{t}^{16}}+1218 {{s}^{52}} {{t}^{16}}+966 {{s}^{48}} {{t}^{16}}+771 {{s}^{44}} {{t}^{16}}+537 {{s}^{40}} {{t}^{16}}+399 {{s}^{36}} {{t}^{16}}+237 {{s}^{32}} {{t}^{16}}\\
  +&159 {{s}^{28}} {{t}^{16}}+90 {{s}^{24}} {{t}^{16}}+45 {{s}^{20}} {{t}^{16}}+21 {{s}^{16}} {{t}^{16}}+21 {{s}^{12}} {{t}^{16}}+10 {{s}^{126}}{{t}^{15}}+18 {{s}^{122}} {{t}^{15}}+27 {{s}^{118}} {{t}^{15}}\\
  +&110 {{s}^{114}} {{t}^{15}}+173 {{s}^{110}} {{t}^{15}}+336 {{s}^{106}} {{t}^{15}}+582 {{s}^{102}} {{t}^{15}}+922 {{s}^{98}} {{t}^{15}}+1317 {{s}^{94}} {{t}^{15}}+1913 {{s}^{90}} {{t}^{15}}\\
  +&2421 {{s}^{86}} {{t}^{15}}+3085 {{s}^{82}} {{t}^{15}}+3618 {{s}^{78}} {{t}^{15}}+4135 {{s}^{74}} {{t}^{15}}+4423 {{s}^{70}} {{t}^{15}}+4660 {{s}^{66}} {{t}^{15}}\\
  +&4541 {{s}^{62}} {{t}^{15}}+4408 {{s}^{58}} {{t}^{15}}+3927 {{s}^{54}} {{t}^{15}}+3493 {{s}^{50}} {{t}^{15}}+2855 {{s}^{46}} {{t}^{15}}+2299 {{s}^{42}} {{t}^{15}}+1694 {{s}^{38}} {{t}^{15}}\\
  +&1271 {{s}^{34}} {{t}^{15}}+785 {{s}^{30}} {{t}^{15}}+567 {{s}^{26}} {{t}^{15}}+305 {{s}^{22}} {{t}^{15}}+180 {{s}^{18}} {{t}^{15}}+98 {{s}^{14}} {{t}^{15}}+63 {{s}^{10}} {{t}^{15}}+28 {{s}^{2}} {{t}^{15}}\\
  +&30 {{s}^{124}} {{t}^{14}}+12 {{s}^{120}} {{t}^{14}}+99 {{s}^{116}} {{t}^{14}}+183 {{s}^{112}} {{t}^{14}}+378 {{s}^{108}} {{t}^{14}}+678 {{s}^{104}} {{t}^{14}}+1248 {{s}^{100}} {{t}^{14}}\\
  +&1851 {{s}^{96}} {{t}^{14}}+2886 {{s}^{92}} {{t}^{14}}+3921 {{s}^{88}} {{t}^{14}}+5247 {{s}^{84}} {{t}^{14}}+6492 {{s}^{80}} {{t}^{14}}+7818 {{s}^{76}} {{t}^{14}}+8763 {{s}^{72}} {{t}^{14}}\\
  +&9642 {{s}^{68}} {{t}^{14}}+9867 {{s}^{64}} {{t}^{14}}+9957 {{s}^{60}} {{t}^{14}}+9357 {{s}^{56}} {{t}^{14}}+8607 {{s}^{52}} {{t}^{14}}+7452 {{s}^{48}} {{t}^{14}}+6252 {{s}^{44}} {{t}^{14}}\\
  +&4857 {{s}^{40}} {{t}^{14}}+3783 {{s}^{36}} {{t}^{14}}+2586 {{s}^{32}} {{t}^{14}}+1791 {{s}^{28}} {{t}^{14}}+1134 {{s}^{24}} {{t}^{14}}+663 {{s}^{20}} {{t}^{14}}+348 {{s}^{16}} {{t}^{14}}\\
  +&231 {{s}^{12}} {{t}^{14}}+63 {{s}^{8}} {{t}^{14}}+48 {{s}^{4}} {{t}^{14}}+36 {{t}^{14}}+15 {{s}^{126}} {{t}^{13}}+30 {{s}^{122}} {{t}^{13}}+48 {{s}^{118}} {{t}^{13}}+198 {{s}^{114}} {{t}^{13}}\\
  +&348 {{s}^{110}} {{t}^{13}}+750 {{s}^{106}} {{t}^{13}}+1374 {{s}^{102}} {{t}^{13}}+2370 {{s}^{98}} {{t}^{13}}+3639 {{s}^{94}} {{t}^{13}}+5499 {{s}^{90}} {{t}^{13}}+7437 {{s}^{86}} {{t}^{13}}\\
  +&9858 {{s}^{82}} {{t}^{13}}+12066{{s}^{78}} {{t}^{13}}+14337 {{s}^{74}} {{t}^{13}}+15978 {{s}^{70}} {{t}^{13}}+17280 {{s}^{66}} {{t}^{13}}+17574 {{s}^{62}} {{t}^{13}}\\
  +&17451 {{s}^{58}} {{t}^{13}}+16173 {{s}^{54}} {{t}^{13}}+14709 {{s}^{50}} {{t}^{13}}+12474 {{s}^{46}} {{t}^{13}}+10248 {{s}^{42}} {{t}^{13}}+7866 {{s}^{38}} {{t}^{13}}\\
  +&5907 {{s}^{34}} {{t}^{13}}+3900 {{s}^{30}} {{t}^{13}}+2691 {{s}^{26}} {{t}^{13}}+1545 {{s}^{22}} {{t}^{13}}+885 {{s}^{18}} {{t}^{13}}+459 {{s}^{14}} {{t}^{13}}+255 {{s}^{10}} {{t}^{13}}\\
  +&42 {{s}^{6}} {{t}^{13}}+84 {{s}^{2}} {{t}^{13}}+45 {{s}^{124}} {{t}^{12}}+20 {{s}^{120}} {{t}^{12}}+163 {{s}^{116}} {{t}^{12}}+316 {{s}^{112}} {{t}^{12}}+724 {{s}^{108}} {{t}^{12}}\\
  +&1370 {{s}^{104}} {{t}^{12}}+2613 {{s}^{100}} {{t}^{12}}+4120 {{s}^{96}} {{t}^{12}}+6555 {{s}^{92}} {{t}^{12}}+9217 {{s}^{88}} {{t}^{12}}+12605 {{s}^{84}} {{t}^{12}}\\
  +&15919 {{s}^{80}} {{t}^{12}}+19438 {{s}^{76}} {{t}^{12}}+22186 {{s}^{72}} {{t}^{12}}+24584 {{s}^{68}} {{t}^{12}}+25537 {{s}^{64}} {{t}^{12}}+25867 {{s}^{60}} {{t}^{12}}\\
  +&24557 {{s}^{56}} {{t}^{12}}+22625 {{s}^{52}} {{t}^{12}}+19660 {{s}^{48}} {{t}^{12}}+16444 {{s}^{44}} {{t}^{12}}+12817 {{s}^{40}} {{t}^{12}}+9770 {{s}^{36}} {{t}^{12}}\\
  +&6680 {{s}^{32}} {{t}^{12}}+4484 {{s}^{28}} {{t}^{12}}+2722 {{s}^{24}} {{t}^{12}}+1522 {{s}^{20}} {{t}^{12}}+747 {{s}^{16}} {{t}^{12}}+424 {{s}^{12}} {{t}^{12}}+108 {{s}^{8}} {{t}^{12}}\\
  +&63 {{s}^{4}} {{t}^{12}}+28 {{t}^{12}}+21 {{s}^{126}} {{t}^{11}}+45 {{s}^{122}} {{t}^{11}}+75 {{s}^{118}}{{t}^{11}}+312 {{s}^{114}} {{t}^{11}}+588 {{s}^{110}} {{t}^{11}}+1308 {{s}^{106}} {{t}^{11}}\\
  +&2445 {{s}^{102}} {{t}^{11}}+4269 {{s}^{98}} {{t}^{11}}+6639 {{s}^{94}} {{t}^{11}}+10017 {{s}^{90}} {{t}^{11}}+13617 {{s}^{86}} {{t}^{11}}+17970 {{s}^{82}} {{t}^{11}}\\
  +&21987 {{s}^{78}} {{t}^{11}}+25947 {{s}^{74}} {{t}^{11}}+28764 {{s}^{70}} {{t}^{11}}+30858 {{s}^{66}} {{t}^{11}}+31137 {{s}^{62}} {{t}^{11}}+30522 {{s}^{58}} {{t}^{11}}\\
  +&28035 {{s}^{54}} {{t}^{11}}+25053 {{s}^{50}}{{t}^{11}}+20913 {{s}^{46}} {{t}^{11}}+16848 {{s}^{42}} {{t}^{11}}+12603 {{s}^{38}} {{t}^{11}}+9129 {{s}^{34}} {{t}^{11}}\\
  +&5874 {{s}^{30}} {{t}^{11}}+3795 {{s}^{26}} {{t}^{11}}+2040 {{s}^{22}} {{t}^{11}}+1083 {{s}^{18}} {{t}^{11}}+486 {{s}^{14}} {{t}^{11}}+243 {{s}^{10}} {{t}^{11}}+30 {{s}^{6}} {{t}^{11}}\\
  +&63 {{s}^{2}} {{t}^{11}}+63 {{s}^{124}} {{t}^{10}}+30 {{s}^{120}} {{t}^{10}}+243 {{s}^{116}} {{t}^{10}}+486 {{s}^{112}} {{t}^{10}}+1083 {{s}^{108}} {{t}^{10}}+2040 {{s}^{104}} {{t}^{10}}\\
  +&3795 {{s}^{100}} {{t}^{10}}+5874 {{s}^{96}} {{t}^{10}}+9129 {{s}^{92}} {{t}^{10}}+12603 {{s}^{88}} {{t}^{10}}+16848 {{s}^{84}} {{t}^{10}}+20913 {{s}^{80}} {{t}^{10}}\\
  +&25053 {{s}^{76}} {{t}^{10}}+28035 {{s}^{72}} {{t}^{10}}+30522 {{s}^{68}} {{t}^{10}}+31137 {{s}^{64}} {{t}^{10}}+30858 {{s}^{60}} {{t}^{10}}+28764 {{s}^{56}} {{t}^{10}}\\
  +&25947 {{s}^{52}} {{t}^{10}}+21987 {{s}^{48}} {{t}^{10}}+17970 {{s}^{44}} {{t}^{10}}+13617 {{s}^{40}} {{t}^{10}}+10017 {{s}^{36}} {{t}^{10}}+6639 {{s}^{32}} {{t}^{10}}\\
  +&4269 {{s}^{28}} {{t}^{10}}+2445 {{s}^{24}} {{t}^{10}}+1308 {{s}^{20}} {{t}^{10}}+588 {{s}^{16}} {{t}^{10}}+312 {{s}^{12}} {{t}^{10}}+75 {{s}^{8}} {{t}^{10}}+45 {{s}^{4}} {{t}^{10}}\\
  +&21 {{t}^{10}}+28 {{s}^{126}} {{t}^{9}}+63 {{s}^{122}} {{t}^{9}}+108 {{s}^{118}} {{t}^{9}}+424 {{s}^{114}} {{t}^{9}}+747 {{s}^{110}} {{t}^{9}}+1522 {{s}^{106}} {{t}^{9}}\\
  +&2722 {{s}^{102}} {{t}^{9}}+4484 {{s}^{98}} {{t}^{9}}+6680 {{s}^{94}} {{t}^{9}}+9770 {{s}^{90}} {{t}^{9}}+12817 {{s}^{86}}{{t}^{9}}+16444 {{s}^{82}} {{t}^{9}}+19660 {{s}^{78}} {{t}^{9}}\\
  +&22625 {{s}^{74}} {{t}^{9}}+24557 {{s}^{70}} {{t}^{9}}+25867 {{s}^{66}} {{t}^{9}}+25537 {{s}^{62}} {{t}^{9}}+24584 {{s}^{58}} {{t}^{9}}+22186 {{s}^{54}} {{t}^{9}}\\
  +&19438 {{s}^{50}} {{t}^{9}}+15919 {{s}^{46}} {{t}^{9}}+12605 {{s}^{42}} {{t}^{9}}+9217 {{s}^{38}} {{t}^{9}}+6555 {{s}^{34}} {{t}^{9}}+4120 {{s}^{30}} {{t}^{9}}+2613 {{s}^{26}} {{t}^{9}}\\
  +&1370 {{s}^{22}} {{t}^{9}}+724 {{s}^{18}} {{t}^{9}}+316 {{s}^{14}} {{t}^{9}}+163{{s}^{10}} {{t}^{9}}+20 {{s}^{6}} {{t}^{9}}+45 {{s}^{2}} {{t}^{9}}+84 {{s}^{124}} {{t}^{8}}+42 {{s}^{120}} {{t}^{8}}\\
  +&255 {{s}^{116}} {{t}^{8}}+459 {{s}^{112}} {{t}^{8}}+885 {{s}^{108}} {{t}^{8}}+1545 {{s}^{104}} {{t}^{8}}+2691 {{s}^{100}} {{t}^{8}}+3900 {{s}^{96}} {{t}^{8}}+5907 {{s}^{92}} {{t}^{8}}\\
  +&7866 {{s}^{88}} {{t}^{8}}+10248 {{s}^{84}} {{t}^{8}}+12474 {{s}^{80}} {{t}^{8}}+14709 {{s}^{76}} {{t}^{8}}+16173 {{s}^{72}} {{t}^{8}}+17451 {{s}^{68}} {{t}^{8}}\\
  +&17574 {{s}^{64}} {{t}^{8}}+17280 {{s}^{60}}{{t}^{8}}+15978 {{s}^{56}} {{t}^{8}}+14337 {{s}^{52}} {{t}^{8}}+12066 {{s}^{48}} {{t}^{8}}+9858 {{s}^{44}} {{t}^{8}}\\
  +&7437 {{s}^{40}} {{t}^{8}}+5499 {{s}^{36}} {{t}^{8}}+3639 {{s}^{32}} {{t}^{8}}+2370 {{s}^{28}} {{t}^{8}}+1374 {{s}^{24}} {{t}^{8}}+750 {{s}^{20}} {{t}^{8}}+348 {{s}^{16}} {{t}^{8}}\\
  +&198 {{s}^{12}} {{t}^{8}}+48 {{s}^{8}} {{t}^{8}}+30 {{s}^{4}} {{t}^{8}}+15 {{t}^{8}}+36 {{s}^{126}} {{t}^{7}}+48 {{s}^{122}} {{t}^{7}}+63 {{s}^{118}} {{t}^{7}}+231 {{s}^{114}} {{t}^{7}}\\
  +&348 {{s}^{110}} {{t}^{7}}+663{{s}^{106}} {{t}^{7}}+1134 {{s}^{102}} {{t}^{7}}+1791 {{s}^{98}} {{t}^{7}}+2586 {{s}^{94}} {{t}^{7}}+3783 {{s}^{90}} {{t}^{7}}+4857 {{s}^{86}} {{t}^{7}}\\
  +&6252 {{s}^{82}} {{t}^{7}}+7452 {{s}^{78}} {{t}^{7}}+8607 {{s}^{74}} {{t}^{7}}+9357 {{s}^{70}} {{t}^{7}}+9957 {{s}^{66}} {{t}^{7}}+9867 {{s}^{62}} {{t}^{7}}+9642 {{s}^{58}} {{t}^{7}}\\
  +&8763 {{s}^{54}} {{t}^{7}}+7818 {{s}^{50}} {{t}^{7}}+6492 {{s}^{46}} {{t}^{7}}+5247 {{s}^{42}} {{t}^{7}}+3921 {{s}^{38}} {{t}^{7}}+2886 {{s}^{34}} {{t}^{7}}+1851{{s}^{30}} {{t}^{7}}\\
  +&1248 {{s}^{26}} {{t}^{7}}+678 {{s}^{22}} {{t}^{7}}+378 {{s}^{18}} {{t}^{7}}+183 {{s}^{14}} {{t}^{7}}+99 {{s}^{10}} {{t}^{7}}+12 {{s}^{6}} {{t}^{7}}+30 {{s}^{2}} {{t}^{7}}+28 {{s}^{124}} {{t}^{6}}\\
  +&63 {{s}^{116}} {{t}^{6}}+98 {{s}^{112}} {{t}^{6}}+180 {{s}^{108}} {{t}^{6}}+305 {{s}^{104}} {{t}^{6}}+567 {{s}^{100}} {{t}^{6}}+785 {{s}^{96}} {{t}^{6}}+1271 {{s}^{92}} {{t}^{6}}\\
  +&1694 {{s}^{88}} {{t}^{6}}+2299 {{s}^{84}} {{t}^{6}}+2855 {{s}^{80}} {{t}^{6}}+3493 {{s}^{76}} {{t}^{6}}+3927 {{s}^{72}} {{t}^{6}}+4408 {{s}^{68}} {{t}^{6}}+4541 {{s}^{64}} {{t}^{6}}\\
  +&4660 {{s}^{60}} {{t}^{6}}+4423 {{s}^{56}} {{t}^{6}}+4135 {{s}^{52}} {{t}^{6}}+3618 {{s}^{48}} {{t}^{6}}+3085 {{s}^{44}} {{t}^{6}}+2421 {{s}^{40}} {{t}^{6}}+1913 {{s}^{36}} {{t}^{6}}\\
  +&1317 {{s}^{32}} {{t}^{6}}+922 {{s}^{28}} {{t}^{6}}+582 {{s}^{24}} {{t}^{6}}+336 {{s}^{20}} {{t}^{6}}+173 {{s}^{16}} {{t}^{6}}+110 {{s}^{12}} {{t}^{6}}+27 {{s}^{8}} {{t}^{6}}+18 {{s}^{4}} {{t}^{6}}\\
  +&10 {{t}^{6}}+21 {{s}^{114}} {{t}^{5}}+21 {{s}^{110}} {{t}^{5}}+45 {{s}^{106}} {{t}^{5}}+90 {{s}^{102}} {{t}^{5}}+159 {{s}^{98}} {{t}^{5}}+237 {{s}^{94}} {{t}^{5}}+399 {{s}^{90}} {{t}^{5}}\\
  +&537 {{s}^{86}} {{t}^{5}}+771 {{s}^{82}} {{t}^{5}}+966 {{s}^{78}} {{t}^{5}}+1218 {{s}^{74}} {{t}^{5}}+1404 {{s}^{70}} {{t}^{5}}+1608 {{s}^{66}} {{t}^{5}}+1683 {{s}^{62}} {{t}^{5}}\\
  +&1779 {{s}^{58}} {{t}^{5}}+1689 {{s}^{54}} {{t}^{5}}+1644 {{s}^{50}} {{t}^{5}}+1446 {{s}^{46}} {{t}^{5}}+1257 {{s}^{42}} {{t}^{5}}+1011 {{s}^{38}} {{t}^{5}}+819 {{s}^{34}} {{t}^{5}}\\
  +&552 {{s}^{30}} {{t}^{5}}+426 {{s}^{26}} {{t}^{5}}+252 {{s}^{22}} {{t}^{5}}+153 {{s}^{18}} {{t}^{5}}+87 {{s}^{14}} {{t}^{5}}+51 {{s}^{10}} {{t}^{5}}+6 {{s}^{6}} {{t}^{5}}+18 {{s}^{2}} {{t}^{5}}\\
  +&15 {{s}^{100}} {{t}^{4}}+15 {{s}^{96}} {{t}^{4}}+45 {{s}^{92}} {{t}^{4}}+60 {{s}^{88}} {{t}^{4}}+114 {{s}^{84}} {{t}^{4}}+153 {{s}^{80}} {{t}^{4}}+228 {{s}^{76}} {{t}^{4}}+285 {{s}^{72}} {{t}^{4}}\\
  +&372 {{s}^{68}} {{t}^{4}}+414 {{s}^{64}} {{t}^{4}}+501 {{s}^{60}} {{t}^{4}}+507 {{s}^{56}} {{t}^{4}}+534 {{s}^{52}} {{t}^{4}}+522 {{s}^{48}} {{t}^{4}}+492 {{s}^{44}} {{t}^{4}}\\
  +&420 {{s}^{40}} {{t}^{4}}+384 {{s}^{36}} {{t}^{4}}+285 {{s}^{32}} {{t}^{4}}+228 {{s}^{28}} {{t}^{4}}+168 {{s}^{24}} {{t}^{4}}+105 {{s}^{20}} {{t}^{4}}+63 {{s}^{16}} {{t}^{4}}\\
  +&48 {{s}^{12}} {{t}^{4}}+12 {{s}^{8}} {{t}^{4}}+9 {{s}^{4}} {{t}^{4}}+6 {{t}^{4}}+10 {{s}^{82}} {{t}^{3}}+10 {{s}^{78}} {{t}^{3}}+28 {{s}^{74}} {{t}^{3}}+36 {{s}^{70}} {{t}^{3}}+55 {{s}^{66}} {{t}^{3}}\\
  +&71 {{s}^{62}} {{t}^{3}}+98 {{s}^{58}} {{t}^{3}}+98 {{s}^{54}} {{t}^{3}}+125 {{s}^{50}} {{t}^{3}}+124 {{s}^{46}} {{t}^{3}}+126 {{s}^{42}} {{t}^{3}}+116 {{s}^{38}} {{t}^{3}}+117 {{s}^{34}} {{t}^{3}}\\
  +&83 {{s}^{30}} {{t}^{3}}+82 {{s}^{26}} {{t}^{3}}+55 {{s}^{22}} {{t}^{3}}+38 {{s}^{18}} {{t}^{3}}+28 {{s}^{14}} {{t}^{3}}+19 {{s}^{10}} {{t}^{3}}+\underline{2 {{s}^{6}} {{t}^{3}}}+9 {{s}^{2}} {{t}^{3}}+6 {{s}^{60}} {{t}^{2}}\\
  +&6 {{s}^{56}} {{t}^{2}}+9 {{s}^{52}} {{t}^{2}}+15 {{s}^{48}} {{t}^{2}}+18 {{s}^{44}} {{t}^{2}}+18 {{s}^{40}} {{t}^{2}}+27 {{s}^{36}} {{t}^{2}}+21 {{s}^{32}} {{t}^{2}}+21 {{s}^{28}} {{t}^{2}}+21 {{s}^{24}} {{t}^{2}}\\
  +&15 {{s}^{20}} {{t}^{2}}+12 {{s}^{16}} {{t}^{2}}+12 {{s}^{12}} {{t}^{2}}+3 {{s}^{8}} {{t}^{2}}+3 {{s}^{4}} {{t}^{2}}+3 {{t}^{2}}+3 {{s}^{34}} t+3 {{s}^{26}} t+3 {{s}^{22}} t+3 {{s}^{18}} t\\
  +&3 {{s}^{14}} t+3 {{s}^{10}} t+3 {{s}^{2}} t+1
\end{align*}

\begin{align*}
  &P(\Hom(\Z^3,E_8)_0;s,t)={{s}^{240}} {{t}^{24}}+3 {{s}^{238}} {{t}^{23}}+3 {{s}^{226}} {{t}^{23}}+3 {{s}^{218}} {{t}^{23}}+3 {{s}^{214}} {{t}^{23}}\\
  +&3 {{s}^{206}} {{t}^{23}}+3 {{s}^{202}} {{t}^{23}}+3 {{s}^{194}} {{t}^{23}}+3 {{s}^{182}} {{t}^{23}}+3 {{s}^{240}} {{t}^{22}}+3 {{s}^{236}} {{t}^{22}}+3 {{s}^{228}} {{t}^{22}}\\
  +&9 {{s}^{224}} {{t}^{22}}+3 {{s}^{220}} {{t}^{22}}+12 {{s}^{216}} {{t}^{22}}+12 {{s}^{212}} {{t}^{22}}+3 {{s}^{208}} {{t}^{22}}+21{{s}^{204}} {{t}^{22}}+18 {{s}^{200}} {{t}^{22}}\\
  +&6 {{s}^{196}} {{t}^{22}}+27 {{s}^{192}} {{t}^{22}}+12 {{s}^{188}} {{t}^{22}}+12 {{s}^{184}} {{t}^{22}}+33 {{s}^{180}} {{t}^{22}}+9 {{s}^{176}} {{t}^{22}}+12 {{s}^{172}} {{t}^{22}}\\
  +&24 {{s}^{168}} {{t}^{22}}+3 {{s}^{164}} {{t}^{22}}+15 {{s}^{160}} {{t}^{22}}+15 {{s}^{156}} {{t}^{22}}+9 {{s}^{148}} {{t}^{22}}+6 {{s}^{144}} {{t}^{22}}+6 {{s}^{136}} {{t}^{22}}\\
  +&9 {{s}^{238}} {{t}^{21}}+{{s}^{234}}{{t}^{21}}+{{s}^{230}} {{t}^{21}}+18 {{s}^{226}} {{t}^{21}}+10 {{s}^{222}} {{t}^{21}}+19 {{s}^{218}} {{t}^{21}}+36 {{s}^{214}} {{t}^{21}}\\
  +&20 {{s}^{210}} {{t}^{21}}+37 {{s}^{206}} {{t}^{21}}+72 {{s}^{202}} {{t}^{21}}+48 {{s}^{198}} {{t}^{21}}+63 {{s}^{194}} {{t}^{21}}+99 {{s}^{190}} {{t}^{21}}+65 {{s}^{186}} {{t}^{21}}\\
  +&98 {{s}^{182}} {{t}^{21}}+135 {{s}^{178}} {{t}^{21}}+73 {{s}^{174}} {{t}^{21}}+115 {{s}^{170}} {{t}^{21}}+144{{s}^{166}} {{t}^{21}}+72 {{s}^{162}} {{t}^{21}}+133 {{s}^{158}} {{t}^{21}}\\
  +&126 {{s}^{154}} {{t}^{21}}+52 {{s}^{150}} {{t}^{21}}+125 {{s}^{146}} {{t}^{21}}+90 {{s}^{142}} {{t}^{21}}+44 {{s}^{138}} {{t}^{21}}+99 {{s}^{134}} {{t}^{21}}+45 {{s}^{130}} {{t}^{21}}\\
  +&27 {{s}^{126}} {{t}^{21}}+64 {{s}^{122}} {{t}^{21}}+18 {{s}^{118}} {{t}^{21}}+18 {{s}^{114}} {{t}^{21}}+28 {{s}^{110}} {{t}^{21}}+10 {{s}^{102}} {{t}^{21}}+10 {{s}^{98}}{{t}^{21}}\\
  +&6 {{s}^{240}} {{t}^{20}}+9 {{s}^{236}} {{t}^{20}}+12 {{s}^{228}} {{t}^{20}}+36 {{s}^{224}} {{t}^{20}}+15 {{s}^{220}} {{t}^{20}}+57 {{s}^{216}} {{t}^{20}}+78 {{s}^{212}} {{t}^{20}}\\
  +&42 {{s}^{208}} {{t}^{20}}+144 {{s}^{204}} {{t}^{20}}+168 {{s}^{200}} {{t}^{20}}+117 {{s}^{196}} {{t}^{20}}+267 {{s}^{192}} {{t}^{20}}+258 {{s}^{188}} {{t}^{20}}\\
  +&234 {{s}^{184}} {{t}^{20}}+426 {{s}^{180}} {{t}^{20}}+378 {{s}^{176}}{{t}^{20}}+351 {{s}^{172}} {{t}^{20}}+558 {{s}^{168}} {{t}^{20}}+471 {{s}^{164}} {{t}^{20}}\\
  +&465 {{s}^{160}} {{t}^{20}}+642 {{s}^{156}} {{t}^{20}}+507 {{s}^{152}} {{t}^{20}}+480 {{s}^{148}} {{t}^{20}}+660 {{s}^{144}} {{t}^{20}}+480 {{s}^{140}} {{t}^{20}}\\
  +&456 {{s}^{136}} {{t}^{20}}+582 {{s}^{132}} {{t}^{20}}+360 {{s}^{128}} {{t}^{20}}+369 {{s}^{124}} {{t}^{20}}+468 {{s}^{120}} {{t}^{20}}+231 {{s}^{116}} {{t}^{20}}\\
  +&267 {{s}^{112}} {{t}^{20}}+294 {{s}^{108}} {{t}^{20}}+111 {{s}^{104}} {{t}^{20}}+183 {{s}^{100}} {{t}^{20}}+165 {{s}^{96}} {{t}^{20}}+33 {{s}^{92}} {{t}^{20}}+96 {{s}^{88}} {{t}^{20}}\\
  +&60 {{s}^{84}} {{t}^{20}}+15 {{s}^{80}} {{t}^{20}}+45 {{s}^{76}} {{t}^{20}}+15 {{s}^{72}} {{t}^{20}}+15 {{s}^{64}} {{t}^{20}}+18 {{s}^{238}} {{t}^{19}}+3 {{s}^{234}} {{t}^{19}}\\
  +&3 {{s}^{230}} {{t}^{19}}+48 {{s}^{226}} {{t}^{19}}+33 {{s}^{222}} {{t}^{19}}+54 {{s}^{218}} {{t}^{19}}+135 {{s}^{214}} {{t}^{19}}+105 {{s}^{210}} {{t}^{19}}+162 {{s}^{206}} {{t}^{19}}\\
  +&342 {{s}^{202}} {{t}^{19}}+297 {{s}^{198}} {{t}^{19}}+390 {{s}^{194}} {{t}^{19}}+672 {{s}^{190}} {{t}^{19}}+585 {{s}^{186}} {{t}^{19}}+783 {{s}^{182}} {{t}^{19}}\\
  +&1170 {{s}^{178}} {{t}^{19}}+960 {{s}^{174}} {{t}^{19}}+1272 {{s}^{170}} {{t}^{19}}+1683 {{s}^{166}} {{t}^{19}}+1365 {{s}^{162}} {{t}^{19}}+1812 {{s}^{158}} {{t}^{19}}\\
  +&2094 {{s}^{154}} {{t}^{19}}+1653 {{s}^{150}} {{t}^{19}}+2181 {{s}^{146}} {{t}^{19}}+2289 {{s}^{142}} {{t}^{19}}+1815 {{s}^{138}} {{t}^{19}}+2286 {{s}^{134}} {{t}^{19}}\\
  +&2163 {{s}^{130}} {{t}^{19}}+1695 {{s}^{126}} {{t}^{19}}+2100 {{s}^{122}} {{t}^{19}}+1845 {{s}^{118}} {{t}^{19}}+1386 {{s}^{114}} {{t}^{19}}+1659 {{s}^{110}} {{t}^{19}}\\
  +&1347 {{s}^{106}} {{t}^{19}}+975{{s}^{102}} {{t}^{19}}+1191 {{s}^{98}} {{t}^{19}}+849 {{s}^{94}} {{t}^{19}}+546 {{s}^{90}} {{t}^{19}}+735 {{s}^{86}} {{t}^{19}}\\
  +&447 {{s}^{82}} {{t}^{19}}+282 {{s}^{78}} {{t}^{19}}+402 {{s}^{74}} {{t}^{19}}+165 {{s}^{70}} {{t}^{19}}+114 {{s}^{66}} {{t}^{19}}+195 {{s}^{62}} {{t}^{19}}+45 {{s}^{58}} {{t}^{19}}\\
  +&45 {{s}^{54}} {{t}^{19}}+66 {{s}^{50}} {{t}^{19}}+21 {{s}^{42}} {{t}^{19}}+21 {{s}^{38}} {{t}^{19}}+10 {{s}^{240}} {{t}^{18}}+18 {{s}^{236}} {{t}^{18}}+27 {{s}^{228}} {{t}^{18}}\\
  +&82 {{s}^{224}} {{t}^{18}}+36 {{s}^{220}} {{t}^{18}}+146 {{s}^{216}} {{t}^{18}}+236 {{s}^{212}} {{t}^{18}}+154 {{s}^{208}} {{t}^{18}}+463 {{s}^{204}} {{t}^{18}}+625 {{s}^{200}} {{t}^{18}}\\
  +&532 {{s}^{196}} {{t}^{18}}+1105 {{s}^{192}} {{t}^{18}}+1300 {{s}^{188}} {{t}^{18}}+1297 {{s}^{184}} {{t}^{18}}+2203 {{s}^{180}} {{t}^{18}}+2413 {{s}^{176}} {{t}^{18}}\\
  +&2429{{s}^{172}} {{t}^{18}}+3642 {{s}^{168}} {{t}^{18}}+3786 {{s}^{164}} {{t}^{18}}+3867 {{s}^{160}} {{t}^{18}}+5249 {{s}^{156}} {{t}^{18}}+5141 {{s}^{152}} {{t}^{18}}\\
  +&5118 {{s}^{148}} {{t}^{18}}+6623 {{s}^{144}} {{t}^{18}}+6094 {{s}^{140}} {{t}^{18}}+6010 {{s}^{136}} {{t}^{18}}+7288 {{s}^{132}} {{t}^{18}}+6240 {{s}^{128}} {{t}^{18}}\\
  +&6157 {{s}^{124}} {{t}^{18}}+7127 {{s}^{120}} {{t}^{18}}+5633{{s}^{116}} {{t}^{18}}+5565 {{s}^{112}} {{t}^{18}}+6020 {{s}^{108}} {{t}^{18}}+4458 {{s}^{104}} {{t}^{18}}\\
  +&4459 {{s}^{100}} {{t}^{18}}+4501 {{s}^{96}} {{t}^{18}}+3029 {{s}^{92}} {{t}^{18}}+3071 {{s}^{88}} {{t}^{18}}+2908 {{s}^{84}} {{t}^{18}}+1831 {{s}^{80}} {{t}^{18}}\\
  +&1836 {{s}^{76}} {{t}^{18}}+1629 {{s}^{72}} {{t}^{18}}+866 {{s}^{68}} {{t}^{18}}+952 {{s}^{64}} {{t}^{18}}+820 {{s}^{60}} {{t}^{18}}+342{{s}^{56}} {{t}^{18}}+412 {{s}^{52}} {{t}^{18}}\\
  +&333 {{s}^{48}} {{t}^{18}}+80 {{s}^{44}} {{t}^{18}}+188 {{s}^{40}} {{t}^{18}}+126 {{s}^{36}} {{t}^{18}}+63 {{s}^{28}} {{t}^{18}}+28 {{s}^{24}} {{t}^{18}}+28 {{s}^{16}} {{t}^{18}}\\
  +&30 {{s}^{238}} {{t}^{17}}+6 {{s}^{234}} {{t}^{17}}+6 {{s}^{230}} {{t}^{17}}+93 {{s}^{226}} {{t}^{17}}+69 {{s}^{222}} {{t}^{17}}+108 {{s}^{218}} {{t}^{17}}+312 {{s}^{214}} {{t}^{17}}\\
  +&276 {{s}^{210}} {{t}^{17}}+417 {{s}^{206}} {{t}^{17}}+948 {{s}^{202}} {{t}^{17}}+933 {{s}^{198}} {{t}^{17}}+1266 {{s}^{194}} {{t}^{17}}+2283 {{s}^{190}} {{t}^{17}}\\
  +&2271 {{s}^{186}} {{t}^{17}}+3069 {{s}^{182}} {{t}^{17}}+4722 {{s}^{178}} {{t}^{17}}+4548 {{s}^{174}} {{t}^{17}}+5952 {{s}^{170}} {{t}^{17}}+8097 {{s}^{166}} {{t}^{17}}\\
  +&7722 {{s}^{162}} {{t}^{17}}+9807 {{s}^{158}} {{t}^{17}}+11955 {{s}^{154}} {{t}^{17}}+11142{{s}^{150}} {{t}^{17}}+13680 {{s}^{146}} {{t}^{17}}\\
  +&15411 {{s}^{142}} {{t}^{17}}+14112 {{s}^{138}} {{t}^{17}}+16569 {{s}^{134}} {{t}^{17}}+17355 {{s}^{130}} {{t}^{17}}+15492 {{s}^{126}} {{t}^{17}}\\
  +&17616 {{s}^{122}} {{t}^{17}}+17364 {{s}^{118}} {{t}^{17}}+14901 {{s}^{114}} {{t}^{17}}+16407 {{s}^{110}} {{t}^{17}}+15207 {{s}^{106}} {{t}^{17}}\\
  +&12564 {{s}^{102}} {{t}^{17}}+13602 {{s}^{98}}{{t}^{17}}+11667 {{s}^{94}} {{t}^{17}}+9150 {{s}^{90}} {{t}^{17}}+9867 {{s}^{86}} {{t}^{17}}+7758 {{s}^{82}} {{t}^{17}}\\
  +&5859 {{s}^{78}} {{t}^{17}}+6282 {{s}^{74}} {{t}^{17}}+4338 {{s}^{70}} {{t}^{17}}+3189 {{s}^{66}} {{t}^{17}}+3480 {{s}^{62}} {{t}^{17}}+2085 {{s}^{58}} {{t}^{17}}\\
  +&1497 {{s}^{54}} {{t}^{17}}+1611 {{s}^{50}} {{t}^{17}}+804 {{s}^{46}} {{t}^{17}}+597 {{s}^{42}} {{t}^{17}}+684 {{s}^{38}} {{t}^{17}}+267 {{s}^{34}} {{t}^{17}}+168 {{s}^{30}} {{t}^{17}}\\
  +&237 {{s}^{26}} {{t}^{17}}+84 {{s}^{22}} {{t}^{17}}+48 {{s}^{18}} {{t}^{17}}+84 {{s}^{14}} {{t}^{17}}+36 {{s}^{2}} {{t}^{17}}+15 {{s}^{240}} {{t}^{16}}+30 {{s}^{236}} {{t}^{16}}\\
  +&48 {{s}^{228}} {{t}^{16}}+147 {{s}^{224}} {{t}^{16}}+66 {{s}^{220}} {{t}^{16}}+279 {{s}^{216}} {{t}^{16}}+492 {{s}^{212}} {{t}^{16}}+348 {{s}^{208}} {{t}^{16}}\\
  +&1038 {{s}^{204}} {{t}^{16}}+1530{{s}^{200}} {{t}^{16}}+1434 {{s}^{196}} {{t}^{16}}+2934 {{s}^{192}} {{t}^{16}}+3804 {{s}^{188}} {{t}^{16}}+4068 {{s}^{184}} {{t}^{16}}\\
  +&6810 {{s}^{180}} {{t}^{16}}+8169 {{s}^{176}} {{t}^{16}}+8784 {{s}^{172}} {{t}^{16}}+12972 {{s}^{168}} {{t}^{16}}+14664 {{s}^{164}} {{t}^{16}}+15714 {{s}^{160}} {{t}^{16}}\\
  +&21117 {{s}^{156}} {{t}^{16}}+22440 {{s}^{152}} {{t}^{16}}+23460 {{s}^{148}} {{t}^{16}}+29685 {{s}^{144}} {{t}^{16}}+29739 {{s}^{140}} {{t}^{16}}\\
  +&30468 {{s}^{136}} {{t}^{16}}+36186 {{s}^{132}} {{t}^{16}}+34206 {{s}^{128}} {{t}^{16}}+34440 {{s}^{124}} {{t}^{16}}+38856 {{s}^{120}} {{t}^{16}}\\
  +&34563 {{s}^{116}} {{t}^{16}}+34164 {{s}^{112}} {{t}^{16}}+36366 {{s}^{108}} {{t}^{16}}+30672 {{s}^{104}} {{t}^{16}}+29805 {{s}^{100}} {{t}^{16}}\\
  +&30036 {{s}^{96}} {{t}^{16}}+23724{{s}^{92}} {{t}^{16}}+22620 {{s}^{88}} {{t}^{16}}+21600 {{s}^{84}} {{t}^{16}}+16017 {{s}^{80}} {{t}^{16}}+14952 {{s}^{76}} {{t}^{16}}\\
  +&13536 {{s}^{72}} {{t}^{16}}+9072 {{s}^{68}} {{t}^{16}}+8511 {{s}^{64}} {{t}^{16}}+7347 {{s}^{60}} {{t}^{16}}+4320 {{s}^{56}} {{t}^{16}}+4134 {{s}^{52}} {{t}^{16}}\\
  +&3354 {{s}^{48}} {{t}^{16}}+1599 {{s}^{44}} {{t}^{16}}+1806 {{s}^{40}} {{t}^{16}}+1308 {{s}^{36}} {{t}^{16}}+426{{s}^{32}} {{t}^{16}}+654 {{s}^{28}} {{t}^{16}}\\
  +&417 {{s}^{24}} {{t}^{16}}+105 {{s}^{20}} {{t}^{16}}+228 {{s}^{16}} {{t}^{16}}+84 {{s}^{12}} {{t}^{16}}+63 {{s}^{4}} {{t}^{16}}+45 {{t}^{16}}+45 {{s}^{238}} {{t}^{15}}\\
  +&10 {{s}^{234}} {{t}^{15}}+10 {{s}^{230}} {{t}^{15}}+153 {{s}^{226}} {{t}^{15}}+118 {{s}^{222}} {{t}^{15}}+181 {{s}^{218}} {{t}^{15}}+567 {{s}^{214}} {{t}^{15}}\\
  +&534 {{s}^{210}} {{t}^{15}}+813 {{s}^{206}} {{t}^{15}}+1954 {{s}^{202}} {{t}^{15}}+2068 {{s}^{198}} {{t}^{15}}+2893 {{s}^{194}} {{t}^{15}}+5366 {{s}^{190}} {{t}^{15}}\\
  +&5769 {{s}^{186}} {{t}^{15}}+7942 {{s}^{182}} {{t}^{15}}+12395 {{s}^{178}} {{t}^{15}}+12960 {{s}^{174}} {{t}^{15}}+17065 {{s}^{170}} {{t}^{15}}+23490 {{s}^{166}} {{t}^{15}}\\
  +&24137 {{s}^{162}} {{t}^{15}}+30462 {{s}^{158}} {{t}^{15}}+37806 {{s}^{154}} {{t}^{15}}+37766 {{s}^{150}}{{t}^{15}}+45668 {{s}^{146}} {{t}^{15}}\\
  +&52502 {{s}^{142}} {{t}^{15}}+51040 {{s}^{138}} {{t}^{15}}+58941 {{s}^{134}} {{t}^{15}}+63320 {{s}^{130}} {{t}^{15}}+59625 {{s}^{126}} {{t}^{15}}\\
  +&66314 {{s}^{122}} {{t}^{15}}+67130 {{s}^{118}} {{t}^{15}}+60786 {{s}^{114}} {{t}^{15}}+65202 {{s}^{110}} {{t}^{15}}+62210 {{s}^{106}} {{t}^{15}}\\
  +&54164 {{s}^{102}} {{t}^{15}}+56459 {{s}^{98}} {{t}^{15}}+50368 {{s}^{94}} {{t}^{15}}+41870 {{s}^{90}} {{t}^{15}}+42637 {{s}^{86}} {{t}^{15}}+35301 {{s}^{82}} {{t}^{15}}\\
  +&28137 {{s}^{78}} {{t}^{15}}+27986 {{s}^{74}} {{t}^{15}}+20998 {{s}^{70}} {{t}^{15}}+16058 {{s}^{66}} {{t}^{15}}+15771 {{s}^{62}} {{t}^{15}}+10565 {{s}^{58}} {{t}^{15}}\\
  +&7724 {{s}^{54}} {{t}^{15}}+7435 {{s}^{50}} {{t}^{15}}+4282 {{s}^{46}} {{t}^{15}}+3061 {{s}^{42}} {{t}^{15}}+3033{{s}^{38}} {{t}^{15}}+1394 {{s}^{34}} {{t}^{15}}\\
  +&910 {{s}^{30}} {{t}^{15}}+1017 {{s}^{26}} {{t}^{15}}+378 {{s}^{22}} {{t}^{15}}+235 {{s}^{18}} {{t}^{15}}+315 {{s}^{14}} {{t}^{15}}+28 {{s}^{10}} {{t}^{15}}+28 {{s}^{6}} {{t}^{15}}\\
  +&108 {{s}^{2}} {{t}^{15}}+21 {{s}^{240}} {{t}^{14}}+45 {{s}^{236}} {{t}^{14}}+75 {{s}^{228}} {{t}^{14}}+231 {{s}^{224}} {{t}^{14}}+105 {{s}^{220}} {{t}^{14}}+456 {{s}^{216}} {{t}^{14}}\\
  +&846 {{s}^{212}}{{t}^{14}}+624 {{s}^{208}} {{t}^{14}}+1881 {{s}^{204}} {{t}^{14}}+2910 {{s}^{200}} {{t}^{14}}+2901 {{s}^{196}} {{t}^{14}}+5889 {{s}^{192}} {{t}^{14}}\\
  +&8007 {{s}^{188}} {{t}^{14}}+8958 {{s}^{184}} {{t}^{14}}+14859 {{s}^{180}} {{t}^{14}}+18549 {{s}^{176}} {{t}^{14}}+20703 {{s}^{172}} {{t}^{14}}+30255 {{s}^{168}} {{t}^{14}}\\
  +&35313 {{s}^{164}} {{t}^{14}}+38910 {{s}^{160}} {{t}^{14}}+51753{{s}^{156}} {{t}^{14}}+56637 {{s}^{152}} {{t}^{14}}+60597 {{s}^{148}} {{t}^{14}}+75504 {{s}^{144}} {{t}^{14}}\\
  +&77859 {{s}^{140}} {{t}^{14}}+81138 {{s}^{136}} {{t}^{14}}+94926 {{s}^{132}} {{t}^{14}}+92478 {{s}^{128}} {{t}^{14}}+93978 {{s}^{124}} {{t}^{14}}+104205 {{s}^{120}} {{t}^{14}}\\
  +&95856 {{s}^{116}} {{t}^{14}}+94959 {{s}^{112}} {{t}^{14}}+99522 {{s}^{108}} {{t}^{14}}+86796 {{s}^{104}} {{t}^{14}}+83820 {{s}^{100}} {{t}^{14}}+83178 {{s}^{96}} {{t}^{14}}\\
  +&68169 {{s}^{92}} {{t}^{14}}+64131 {{s}^{88}} {{t}^{14}}+60213 {{s}^{84}} {{t}^{14}}+46287 {{s}^{80}} {{t}^{14}}+42306 {{s}^{76}} {{t}^{14}}+37458 {{s}^{72}} {{t}^{14}}\\
  +&26280 {{s}^{68}} {{t}^{14}}+23745 {{s}^{64}} {{t}^{14}}+19803 {{s}^{60}} {{t}^{14}}+12315 {{s}^{56}} {{t}^{14}}+11103 {{s}^{52}} {{t}^{14}}+8568 {{s}^{48}}{{t}^{14}}\\
  +&4425 {{s}^{44}} {{t}^{14}}+4431 {{s}^{40}} {{t}^{14}}+3033 {{s}^{36}} {{t}^{14}}+1113 {{s}^{32}} {{t}^{14}}+1413 {{s}^{28}} {{t}^{14}}+822 {{s}^{24}} {{t}^{14}}+210 {{s}^{20}} {{t}^{14}}\\
  +&420 {{s}^{16}} {{t}^{14}}+147 {{s}^{12}} {{t}^{14}}+84 {{s}^{4}} {{t}^{14}}+36 {{t}^{14}}+63 {{s}^{238}} {{t}^{13}}+15 {{s}^{234}} {{t}^{13}}+15 {{s}^{230}} {{t}^{13}}+228 {{s}^{226}} {{t}^{13}}\\
  +&180 {{s}^{222}} {{t}^{13}}+273 {{s}^{218}} {{t}^{13}}+900 {{s}^{214}} {{t}^{13}}+879 {{s}^{210}} {{t}^{13}}+1350 {{s}^{206}} {{t}^{13}}+3291 {{s}^{202}} {{t}^{13}}+3627 {{s}^{198}} {{t}^{13}}\\
  +&5172 {{s}^{194}} {{t}^{13}}+9585 {{s}^{190}} {{t}^{13}}+10677 {{s}^{186}} {{t}^{13}}+14790 {{s}^{182}} {{t}^{13}}+22941 {{s}^{178}} {{t}^{13}}+24789 {{s}^{174}} {{t}^{13}}\\
  +&32637 {{s}^{170}} {{t}^{13}}+44601 {{s}^{166}} {{t}^{13}}+47016 {{s}^{162}} {{t}^{13}}+58998 {{s}^{158}} {{t}^{13}}+72891 {{s}^{154}} {{t}^{13}}+74361 {{s}^{150}} {{t}^{13}}\\
  +&89097 {{s}^{146}} {{t}^{13}}+102006 {{s}^{142}} {{t}^{13}}+100806 {{s}^{138}} {{t}^{13}}+115218 {{s}^{134}} {{t}^{13}}+123306 {{s}^{130}} {{t}^{13}}\\
  +&117681 {{s}^{126}} {{t}^{13}}+129222 {{s}^{122}} {{t}^{13}}+130233 {{s}^{118}} {{t}^{13}}+119397 {{s}^{114}}{{t}^{13}}+126081 {{s}^{110}} {{t}^{13}}\\
  +&119685 {{s}^{106}} {{t}^{13}}+105360 {{s}^{102}} {{t}^{13}}+107586 {{s}^{98}} {{t}^{13}}+95511 {{s}^{94}} {{t}^{13}}+80226 {{s}^{90}} {{t}^{13}}+79485 {{s}^{86}} {{t}^{13}}\\
  +&65478 {{s}^{82}} {{t}^{13}}+52560 {{s}^{78}} {{t}^{13}}+50451 {{s}^{74}} {{t}^{13}}+37692 {{s}^{70}} {{t}^{13}}+28869 {{s}^{66}} {{t}^{13}}+27063 {{s}^{62}} {{t}^{13}}\\
  +&17922{{s}^{58}} {{t}^{13}}+13047 {{s}^{54}} {{t}^{13}}+11871 {{s}^{50}} {{t}^{13}}+6645 {{s}^{46}} {{t}^{13}}+4692 {{s}^{42}} {{t}^{13}}+4302 {{s}^{38}} {{t}^{13}}\\
  +&1857 {{s}^{34}} {{t}^{13}}+1206 {{s}^{30}} {{t}^{13}}+1227 {{s}^{26}} {{t}^{13}}+384 {{s}^{22}} {{t}^{13}}+255 {{s}^{18}} {{t}^{13}}+318 {{s}^{14}} {{t}^{13}}+21 {{s}^{10}} {{t}^{13}}\\
  +&21 {{s}^{6}} {{t}^{13}}+84 {{s}^{2}} {{t}^{13}}+28 {{s}^{240}} {{t}^{12}}+63{{s}^{236}} {{t}^{12}}+108 {{s}^{228}} {{t}^{12}}+334 {{s}^{224}} {{t}^{12}}+153 {{s}^{220}} {{t}^{12}}\\
  +&677 {{s}^{216}} {{t}^{12}}+1270 {{s}^{212}} {{t}^{12}}+982 {{s}^{208}} {{t}^{12}}+2856 {{s}^{204}} {{t}^{12}}+4430 {{s}^{200}} {{t}^{12}}+4555 {{s}^{196}} {{t}^{12}}\\
  +&9033 {{s}^{192}} {{t}^{12}}+12278 {{s}^{188}} {{t}^{12}}+13898 {{s}^{184}} {{t}^{12}}+22624 {{s}^{180}} {{t}^{12}}+28162 {{s}^{176}}{{t}^{12}}+31691 {{s}^{172}} {{t}^{12}}\\
  +&45564 {{s}^{168}} {{t}^{12}}+52902 {{s}^{164}} {{t}^{12}}+58519 {{s}^{160}} {{t}^{12}}+76658 {{s}^{156}} {{t}^{12}}+83482 {{s}^{152}} {{t}^{12}}+89505 {{s}^{148}} {{t}^{12}}\\
  +&109812 {{s}^{144}} {{t}^{12}}+112731 {{s}^{140}} {{t}^{12}}+117448 {{s}^{136}} {{t}^{12}}+135298 {{s}^{132}} {{t}^{12}}+131270 {{s}^{128}} {{t}^{12}}\\
  +&133116{{s}^{124}} {{t}^{12}}+145262 {{s}^{120}} {{t}^{12}}+133116 {{s}^{116}} {{t}^{12}}+131270 {{s}^{112}} {{t}^{12}}+135298 {{s}^{108}} {{t}^{12}}\\
  +&117448 {{s}^{104}} {{t}^{12}}+112731 {{s}^{100}} {{t}^{12}}+109812 {{s}^{96}} {{t}^{12}}+89505 {{s}^{92}} {{t}^{12}}+83482 {{s}^{88}} {{t}^{12}}+76658 {{s}^{84}} {{t}^{12}}\\
  +&58519 {{s}^{80}} {{t}^{12}}+52902 {{s}^{76}} {{t}^{12}}+45564 {{s}^{72}} {{t}^{12}}+31691 {{s}^{68}} {{t}^{12}}+28162 {{s}^{64}} {{t}^{12}}+22624 {{s}^{60}} {{t}^{12}}\\
  +&13898 {{s}^{56}} {{t}^{12}}+12278 {{s}^{52}} {{t}^{12}}+9033 {{s}^{48}} {{t}^{12}}+4555 {{s}^{44}} {{t}^{12}}+4430 {{s}^{40}} {{t}^{12}}+2856 {{s}^{36}} {{t}^{12}}\\
  +&982 {{s}^{32}} {{t}^{12}}+1270 {{s}^{28}} {{t}^{12}}+677 {{s}^{24}} {{t}^{12}}+153 {{s}^{20}} {{t}^{12}}+334 {{s}^{16}} {{t}^{12}}+108 {{s}^{12}} {{t}^{12}}+63 {{s}^{4}} {{t}^{12}}\\
  +&28 {{t}^{12}}+84 {{s}^{238}} {{t}^{11}}+21 {{s}^{234}} {{t}^{11}}+21 {{s}^{230}} {{t}^{11}}+318 {{s}^{226}} {{t}^{11}}+255 {{s}^{222}} {{t}^{11}}+384 {{s}^{218}} {{t}^{11}}\\
  +&1227 {{s}^{214}} {{t}^{11}}+1206 {{s}^{210}} {{t}^{11}}+1857 {{s}^{206}} {{t}^{11}}+4302 {{s}^{202}} {{t}^{11}}+4692 {{s}^{198}} {{t}^{11}}+6645 {{s}^{194}} {{t}^{11}}\\
  +&11871 {{s}^{190}} {{t}^{11}}+13047{{s}^{186}} {{t}^{11}}+17922 {{s}^{182}} {{t}^{11}}+27063 {{s}^{178}} {{t}^{11}}+28869 {{s}^{174}} {{t}^{11}}+37692 {{s}^{170}} {{t}^{11}}\\
  +&50451 {{s}^{166}} {{t}^{11}}+52560 {{s}^{162}} {{t}^{11}}+65478 {{s}^{158}} {{t}^{11}}+79485 {{s}^{154}} {{t}^{11}}+80226 {{s}^{150}} {{t}^{11}}+95511 {{s}^{146}} {{t}^{11}}\\
  +&107586 {{s}^{142}} {{t}^{11}}+105360 {{s}^{138}} {{t}^{11}}+119685{{s}^{134}} {{t}^{11}}+126081 {{s}^{130}} {{t}^{11}}+119397 {{s}^{126}} {{t}^{11}}\\
  +&130233 {{s}^{122}} {{t}^{11}}+129222 {{s}^{118}} {{t}^{11}}+117681 {{s}^{114}} {{t}^{11}}+123306 {{s}^{110}} {{t}^{11}}+115218 {{s}^{106}} {{t}^{11}}\\
  +&100806 {{s}^{102}} {{t}^{11}}+102006 {{s}^{98}} {{t}^{11}}+89097 {{s}^{94}} {{t}^{11}}+74361 {{s}^{90}} {{t}^{11}}+72891 {{s}^{86}} {{t}^{11}}+58998{{s}^{82}} {{t}^{11}}\\
  +&47016 {{s}^{78}} {{t}^{11}}+44601 {{s}^{74}} {{t}^{11}}+32637 {{s}^{70}} {{t}^{11}}+24789 {{s}^{66}} {{t}^{11}}+22941 {{s}^{62}} {{t}^{11}}+14790 {{s}^{58}} {{t}^{11}}\\
  +&10677 {{s}^{54}} {{t}^{11}}+9585 {{s}^{50}} {{t}^{11}}+5172 {{s}^{46}} {{t}^{11}}+3627 {{s}^{42}} {{t}^{11}}+3291 {{s}^{38}} {{t}^{11}}+1350 {{s}^{34}} {{t}^{11}}+879 {{s}^{30}} {{t}^{11}}\\
  +&900 {{s}^{26}} {{t}^{11}}+273 {{s}^{22}} {{t}^{11}}+180 {{s}^{18}} {{t}^{11}}+228 {{s}^{14}} {{t}^{11}}+15 {{s}^{10}} {{t}^{11}}+15 {{s}^{6}} {{t}^{11}}+63 {{s}^{2}} {{t}^{11}}+36 {{s}^{240}} {{t}^{10}}\\
  +&84 {{s}^{236}} {{t}^{10}}+147 {{s}^{228}} {{t}^{10}}+420 {{s}^{224}} {{t}^{10}}+210 {{s}^{220}} {{t}^{10}}+822 {{s}^{216}} {{t}^{10}}+1413 {{s}^{212}} {{t}^{10}}+1113 {{s}^{208}} {{t}^{10}}\\
  +&3033 {{s}^{204}} {{t}^{10}}+4431 {{s}^{200}} {{t}^{10}}+4425 {{s}^{196}} {{t}^{10}}+8568 {{s}^{192}} {{t}^{10}}+11103 {{s}^{188}} {{t}^{10}}+12315 {{s}^{184}} {{t}^{10}}\\
  +&19803 {{s}^{180}} {{t}^{10}}+23745 {{s}^{176}} {{t}^{10}}+26280 {{s}^{172}} {{t}^{10}}+37458 {{s}^{168}} {{t}^{10}}+42306 {{s}^{164}} {{t}^{10}}+46287 {{s}^{160}} {{t}^{10}}\\
  +&60213 {{s}^{156}} {{t}^{10}}+64131 {{s}^{152}} {{t}^{10}}+68169 {{s}^{148}} {{t}^{10}}+83178{{s}^{144}} {{t}^{10}}+83820 {{s}^{140}} {{t}^{10}}+86796 {{s}^{136}} {{t}^{10}}\\
  +&99522 {{s}^{132}} {{t}^{10}}+94959 {{s}^{128}} {{t}^{10}}+95856 {{s}^{124}} {{t}^{10}}+104205 {{s}^{120}} {{t}^{10}}+93978 {{s}^{116}} {{t}^{10}}+92478 {{s}^{112}} {{t}^{10}}\\
  +&94926 {{s}^{108}} {{t}^{10}}+81138 {{s}^{104}} {{t}^{10}}+77859 {{s}^{100}} {{t}^{10}}+75504 {{s}^{96}} {{t}^{10}}+60597 {{s}^{92}}{{t}^{10}}+56637 {{s}^{88}} {{t}^{10}}\\
  +&51753 {{s}^{84}} {{t}^{10}}+38910 {{s}^{80}} {{t}^{10}}+35313 {{s}^{76}} {{t}^{10}}+30255 {{s}^{72}} {{t}^{10}}+20703 {{s}^{68}} {{t}^{10}}+18549 {{s}^{64}} {{t}^{10}}\\
  +&14859 {{s}^{60}} {{t}^{10}}+8958 {{s}^{56}} {{t}^{10}}+8007 {{s}^{52}} {{t}^{10}}+5889 {{s}^{48}} {{t}^{10}}+2901 {{s}^{44}} {{t}^{10}}+2910 {{s}^{40}} {{t}^{10}}+1881 {{s}^{36}} {{t}^{10}}\\
  +&624{{s}^{32}} {{t}^{10}}+846 {{s}^{28}} {{t}^{10}}+456 {{s}^{24}} {{t}^{10}}+105 {{s}^{20}} {{t}^{10}}+231 {{s}^{16}} {{t}^{10}}+75 {{s}^{12}} {{t}^{10}}+45 {{s}^{4}} {{t}^{10}}+21 {{t}^{10}}\\
  +&108 {{s}^{238}} {{t}^{9}}+28 {{s}^{234}} {{t}^{9}}+28 {{s}^{230}} {{t}^{9}}+315 {{s}^{226}} {{t}^{9}}+235 {{s}^{222}} {{t}^{9}}+378 {{s}^{218}} {{t}^{9}}+1017 {{s}^{214}} {{t}^{9}}\\
  +&910 {{s}^{210}} {{t}^{9}}+1394 {{s}^{206}} {{t}^{9}}+3033{{s}^{202}} {{t}^{9}}+3061 {{s}^{198}} {{t}^{9}}+4282 {{s}^{194}} {{t}^{9}}+7435 {{s}^{190}} {{t}^{9}}+7724 {{s}^{186}} {{t}^{9}}\\
  +&10565 {{s}^{182}} {{t}^{9}}+15771 {{s}^{178}} {{t}^{9}}+16058 {{s}^{174}} {{t}^{9}}+20998 {{s}^{170}} {{t}^{9}}+27986 {{s}^{166}} {{t}^{9}}+28137 {{s}^{162}} {{t}^{9}}\\
  +&35301 {{s}^{158}} {{t}^{9}}+42637 {{s}^{154}} {{t}^{9}}+41870 {{s}^{150}} {{t}^{9}}+50368 {{s}^{146}}{{t}^{9}}+56459 {{s}^{142}} {{t}^{9}}+54164 {{s}^{138}} {{t}^{9}}\\
  +&62210 {{s}^{134}} {{t}^{9}}+65202 {{s}^{130}} {{t}^{9}}+60786 {{s}^{126}} {{t}^{9}}+67130 {{s}^{122}} {{t}^{9}}+66314 {{s}^{118}} {{t}^{9}}+59625 {{s}^{114}} {{t}^{9}}\\
  +&63320 {{s}^{110}} {{t}^{9}}+58941 {{s}^{106}} {{t}^{9}}+51040 {{s}^{102}} {{t}^{9}}+52502 {{s}^{98}} {{t}^{9}}+45668 {{s}^{94}} {{t}^{9}}+37766 {{s}^{90}} {{t}^{9}}\\
  +&37806 {{s}^{86}} {{t}^{9}}+30462 {{s}^{82}} {{t}^{9}}+24137 {{s}^{78}} {{t}^{9}}+23490 {{s}^{74}} {{t}^{9}}+17065 {{s}^{70}} {{t}^{9}}+12960 {{s}^{66}} {{t}^{9}}\\
  +&12395 {{s}^{62}} {{t}^{9}}+7942 {{s}^{58}} {{t}^{9}}+5769 {{s}^{54}} {{t}^{9}}+5366 {{s}^{50}} {{t}^{9}}+2893 {{s}^{46}} {{t}^{9}}+2068 {{s}^{42}} {{t}^{9}}+1954 {{s}^{38}} {{t}^{9}}\\
  +&813 {{s}^{34}} {{t}^{9}}+534 {{s}^{30}} {{t}^{9}}+567 {{s}^{26}} {{t}^{9}}+181 {{s}^{22}} {{t}^{9}}+118 {{s}^{18}} {{t}^{9}}+153 {{s}^{14}} {{t}^{9}}+10 {{s}^{10}} {{t}^{9}}+10 {{s}^{6}} {{t}^{9}}\\
  +&45 {{s}^{2}} {{t}^{9}}+45 {{s}^{240}} {{t}^{8}}+63 {{s}^{236}} {{t}^{8}}+84 {{s}^{228}} {{t}^{8}}+228 {{s}^{224}} {{t}^{8}}+105 {{s}^{220}} {{t}^{8}}+417 {{s}^{216}} {{t}^{8}}+654 {{s}^{212}} {{t}^{8}}\\
  +&426 {{s}^{208}} {{t}^{8}}+1308 {{s}^{204}} {{t}^{8}}+1806 {{s}^{200}} {{t}^{8}}+1599 {{s}^{196}} {{t}^{8}}+3354{{s}^{192}} {{t}^{8}}+4134 {{s}^{188}} {{t}^{8}}+4320 {{s}^{184}} {{t}^{8}}\\
  +&7347 {{s}^{180}} {{t}^{8}}+8511 {{s}^{176}} {{t}^{8}}+9072 {{s}^{172}} {{t}^{8}}+13536 {{s}^{168}} {{t}^{8}}+14952 {{s}^{164}} {{t}^{8}}+16017 {{s}^{160}} {{t}^{8}}\\
  +&21600 {{s}^{156}} {{t}^{8}}+22620 {{s}^{152}} {{t}^{8}}+23724 {{s}^{148}} {{t}^{8}}+30036 {{s}^{144}} {{t}^{8}}+29805 {{s}^{140}} {{t}^{8}}+30672 {{s}^{136}} {{t}^{8}}\\
  +&36366 {{s}^{132}} {{t}^{8}}+34164 {{s}^{128}} {{t}^{8}}+34563 {{s}^{124}} {{t}^{8}}+38856 {{s}^{120}} {{t}^{8}}+34440 {{s}^{116}} {{t}^{8}}+34206 {{s}^{112}} {{t}^{8}}\\
  +&36186 {{s}^{108}} {{t}^{8}}+30468 {{s}^{104}} {{t}^{8}}+29739 {{s}^{100}} {{t}^{8}}+29685 {{s}^{96}} {{t}^{8}}+23460 {{s}^{92}} {{t}^{8}}+22440 {{s}^{88}} {{t}^{8}}\\
  +&21117 {{s}^{84}} {{t}^{8}}+15714 {{s}^{80}} {{t}^{8}}+14664{{s}^{76}} {{t}^{8}}+12972 {{s}^{72}} {{t}^{8}}+8784 {{s}^{68}} {{t}^{8}}+8169 {{s}^{64}} {{t}^{8}}+6810 {{s}^{60}} {{t}^{8}}\\
  +&4068 {{s}^{56}} {{t}^{8}}+3804 {{s}^{52}} {{t}^{8}}+2934 {{s}^{48}} {{t}^{8}}+1434 {{s}^{44}} {{t}^{8}}+1530 {{s}^{40}} {{t}^{8}}+1038 {{s}^{36}} {{t}^{8}}+348 {{s}^{32}} {{t}^{8}}\\
  +&492 {{s}^{28}} {{t}^{8}}+279 {{s}^{24}} {{t}^{8}}+66 {{s}^{20}} {{t}^{8}}+147 {{s}^{16}} {{t}^{8}}+48 {{s}^{12}} {{t}^{8}}+30{{s}^{4}} {{t}^{8}}+15 {{t}^{8}}+36 {{s}^{238}} {{t}^{7}}\\
  +&84 {{s}^{226}} {{t}^{7}}+48 {{s}^{222}} {{t}^{7}}+84 {{s}^{218}} {{t}^{7}}+237 {{s}^{214}} {{t}^{7}}+168 {{s}^{210}} {{t}^{7}}+267 {{s}^{206}} {{t}^{7}}+684 {{s}^{202}} {{t}^{7}}\\
  +&597 {{s}^{198}} {{t}^{7}}+804 {{s}^{194}} {{t}^{7}}+1611 {{s}^{190}} {{t}^{7}}+1497 {{s}^{186}} {{t}^{7}}+2085 {{s}^{182}} {{t}^{7}}+3480 {{s}^{178}} {{t}^{7}}+3189 {{s}^{174}} {{t}^{7}}\\
  +&4338{{s}^{170}} {{t}^{7}}+6282 {{s}^{166}} {{t}^{7}}+5859 {{s}^{162}} {{t}^{7}}+7758 {{s}^{158}} {{t}^{7}}+9867 {{s}^{154}} {{t}^{7}}+9150 {{s}^{150}} {{t}^{7}}+11667 {{s}^{146}} {{t}^{7}}\\
  +&13602 {{s}^{142}} {{t}^{7}}+12564 {{s}^{138}} {{t}^{7}}+15207 {{s}^{134}} {{t}^{7}}+16407 {{s}^{130}} {{t}^{7}}+14901 {{s}^{126}} {{t}^{7}}+17364 {{s}^{122}} {{t}^{7}}\\
  +&17616 {{s}^{118}} {{t}^{7}}+15492 {{s}^{114}} {{t}^{7}}+17355 {{s}^{110}} {{t}^{7}}+16569 {{s}^{106}} {{t}^{7}}+14112 {{s}^{102}} {{t}^{7}}+15411 {{s}^{98}} {{t}^{7}}\\
  +&13680 {{s}^{94}} {{t}^{7}}+11142 {{s}^{90}} {{t}^{7}}+11955 {{s}^{86}} {{t}^{7}}+9807 {{s}^{82}} {{t}^{7}}+7722 {{s}^{78}} {{t}^{7}}+8097 {{s}^{74}} {{t}^{7}}+5952 {{s}^{70}} {{t}^{7}}\\
  +&4548 {{s}^{66}} {{t}^{7}}+4722 {{s}^{62}} {{t}^{7}}+3069 {{s}^{58}} {{t}^{7}}+2271 {{s}^{54}} {{t}^{7}}+2283{{s}^{50}} {{t}^{7}}+1266 {{s}^{46}} {{t}^{7}}+933 {{s}^{42}} {{t}^{7}}\\
  +&948 {{s}^{38}} {{t}^{7}}+417 {{s}^{34}} {{t}^{7}}+276 {{s}^{30}} {{t}^{7}}+312 {{s}^{26}} {{t}^{7}}+108 {{s}^{22}} {{t}^{7}}+69 {{s}^{18}} {{t}^{7}}+93 {{s}^{14}} {{t}^{7}}+6 {{s}^{10}} {{t}^{7}}+6 {{s}^{6}} {{t}^{7}}\\
  +&30 {{s}^{2}} {{t}^{7}}+28 {{s}^{224}} {{t}^{6}}+28 {{s}^{216}} {{t}^{6}}+63 {{s}^{212}} {{t}^{6}}+126 {{s}^{204}} {{t}^{6}}+188 {{s}^{200}} {{t}^{6}}+80 {{s}^{196}}{{t}^{6}}+333 {{s}^{192}} {{t}^{6}}\\
  +&412 {{s}^{188}} {{t}^{6}}+342 {{s}^{184}} {{t}^{6}}+820 {{s}^{180}} {{t}^{6}}+952 {{s}^{176}} {{t}^{6}}+866 {{s}^{172}} {{t}^{6}}+1629 {{s}^{168}} {{t}^{6}}+1836 {{s}^{164}} {{t}^{6}}\\
  +&1831 {{s}^{160}} {{t}^{6}}+2908 {{s}^{156}} {{t}^{6}}+3071 {{s}^{152}} {{t}^{6}}+3029 {{s}^{148}} {{t}^{6}}+4501 {{s}^{144}} {{t}^{6}}+4459 {{s}^{140}} {{t}^{6}}+4458 {{s}^{136}} {{t}^{6}}\\
  +&6020{{s}^{132}} {{t}^{6}}+5565 {{s}^{128}} {{t}^{6}}+5633 {{s}^{124}} {{t}^{6}}+7127 {{s}^{120}} {{t}^{6}}+6157 {{s}^{116}} {{t}^{6}}+6240 {{s}^{112}} {{t}^{6}}+7288 {{s}^{108}} {{t}^{6}}\\
  +&6010 {{s}^{104}} {{t}^{6}}+6094 {{s}^{100}} {{t}^{6}}+6623 {{s}^{96}} {{t}^{6}}+5118 {{s}^{92}} {{t}^{6}}+5141 {{s}^{88}} {{t}^{6}}+5249 {{s}^{84}} {{t}^{6}}+3867 {{s}^{80}} {{t}^{6}}\\
  +&3786 {{s}^{76}} {{t}^{6}}+3642 {{s}^{72}} {{t}^{6}}+2429 {{s}^{68}} {{t}^{6}}+2413 {{s}^{64}} {{t}^{6}}+2203 {{s}^{60}} {{t}^{6}}+1297 {{s}^{56}} {{t}^{6}}+1300 {{s}^{52}} {{t}^{6}}\\
  +&1105 {{s}^{48}} {{t}^{6}}+532 {{s}^{44}} {{t}^{6}}+625 {{s}^{40}} {{t}^{6}}+463 {{s}^{36}} {{t}^{6}}+154 {{s}^{32}} {{t}^{6}}+236 {{s}^{28}} {{t}^{6}}+146 {{s}^{24}} {{t}^{6}}+36 {{s}^{20}} {{t}^{6}}\\
  +&82 {{s}^{16}} {{t}^{6}}+27 {{s}^{12}} {{t}^{6}}+18 {{s}^{4}} {{t}^{6}}+10 {{t}^{6}}+21 {{s}^{202}} {{t}^{5}}+21 {{s}^{198}} {{t}^{5}}+66 {{s}^{190}} {{t}^{5}}+45 {{s}^{186}} {{t}^{5}}+45 {{s}^{182}} {{t}^{5}}\\
  +&195 {{s}^{178}} {{t}^{5}}+114 {{s}^{174}} {{t}^{5}}+165 {{s}^{170}} {{t}^{5}}+402 {{s}^{166}} {{t}^{5}}+282 {{s}^{162}} {{t}^{5}}+447 {{s}^{158}} {{t}^{5}}+735 {{s}^{154}} {{t}^{5}}\\
  +&546 {{s}^{150}} {{t}^{5}}+849 {{s}^{146}} {{t}^{5}}+1191 {{s}^{142}} {{t}^{5}}+975 {{s}^{138}} {{t}^{5}}+1347 {{s}^{134}} {{t}^{5}}+1659{{s}^{130}} {{t}^{5}}+1386 {{s}^{126}} {{t}^{5}}\\
  +&1845 {{s}^{122}} {{t}^{5}}+2100 {{s}^{118}} {{t}^{5}}+1695 {{s}^{114}} {{t}^{5}}+2163 {{s}^{110}} {{t}^{5}}+2286 {{s}^{106}} {{t}^{5}}+1815 {{s}^{102}} {{t}^{5}}+2289 {{s}^{98}} {{t}^{5}}\\
  +&2181 {{s}^{94}} {{t}^{5}}+1653 {{s}^{90}} {{t}^{5}}+2094 {{s}^{86}} {{t}^{5}}+1812 {{s}^{82}} {{t}^{5}}+1365 {{s}^{78}} {{t}^{5}}+1683 {{s}^{74}} {{t}^{5}}+1272 {{s}^{70}} {{t}^{5}}\\
  +&960 {{s}^{66}} {{t}^{5}}+1170 {{s}^{62}} {{t}^{5}}+783 {{s}^{58}} {{t}^{5}}+585 {{s}^{54}} {{t}^{5}}+672 {{s}^{50}} {{t}^{5}}+390 {{s}^{46}} {{t}^{5}}+297 {{s}^{42}} {{t}^{5}}\\
  +&342 {{s}^{38}} {{t}^{5}}+162 {{s}^{34}} {{t}^{5}}+105 {{s}^{30}} {{t}^{5}}+135 {{s}^{26}} {{t}^{5}}+54 {{s}^{22}} {{t}^{5}}+33 {{s}^{18}} {{t}^{5}}+48 {{s}^{14}} {{t}^{5}}+3 {{s}^{10}} {{t}^{5}}+3 {{s}^{6}} {{t}^{5}}\\
  +&18 {{s}^{2}} {{t}^{5}}+15 {{s}^{176}} {{t}^{4}}+15 {{s}^{168}}{{t}^{4}}+45 {{s}^{164}} {{t}^{4}}+15 {{s}^{160}} {{t}^{4}}+60 {{s}^{156}} {{t}^{4}}+96 {{s}^{152}} {{t}^{4}}+33 {{s}^{148}} {{t}^{4}}\\
  +&165 {{s}^{144}} {{t}^{4}}+183 {{s}^{140}} {{t}^{4}}+111 {{s}^{136}} {{t}^{4}}+294 {{s}^{132}} {{t}^{4}}+267 {{s}^{128}} {{t}^{4}}+231 {{s}^{124}} {{t}^{4}}+468 {{s}^{120}} {{t}^{4}}\\
  +&369 {{s}^{116}} {{t}^{4}}+360 {{s}^{112}} {{t}^{4}}+582 {{s}^{108}} {{t}^{4}}+456 {{s}^{104}} {{t}^{4}}+480 {{s}^{100}} {{t}^{4}}+660 {{s}^{96}} {{t}^{4}}+480 {{s}^{92}} {{t}^{4}}\\
  +&507 {{s}^{88}} {{t}^{4}}+642 {{s}^{84}} {{t}^{4}}+465 {{s}^{80}} {{t}^{4}}+471 {{s}^{76}} {{t}^{4}}+558 {{s}^{72}} {{t}^{4}}+351 {{s}^{68}} {{t}^{4}}+378 {{s}^{64}} {{t}^{4}}+426 {{s}^{60}} {{t}^{4}}\\
  +&234 {{s}^{56}} {{t}^{4}}+258 {{s}^{52}} {{t}^{4}}+267 {{s}^{48}} {{t}^{4}}+117 {{s}^{44}} {{t}^{4}}+168 {{s}^{40}} {{t}^{4}}+144 {{s}^{36}} {{t}^{4}}+42 {{s}^{32}} {{t}^{4}}+78 {{s}^{28}}{{t}^{4}}\\
  +&57 {{s}^{24}} {{t}^{4}}+15 {{s}^{20}} {{t}^{4}}+36 {{s}^{16}} {{t}^{4}}+12 {{s}^{12}} {{t}^{4}}+9 {{s}^{4}} {{t}^{4}}+6 {{t}^{4}}+10 {{s}^{142}} {{t}^{3}}+10 {{s}^{138}} {{t}^{3}}+28 {{s}^{130}} {{t}^{3}}\\
  +&18 {{s}^{126}} {{t}^{3}}+18 {{s}^{122}} {{t}^{3}}+64 {{s}^{118}} {{t}^{3}}+27 {{s}^{114}} {{t}^{3}}+45 {{s}^{110}} {{t}^{3}}+99 {{s}^{106}} {{t}^{3}}+44 {{s}^{102}} {{t}^{3}}+90 {{s}^{98}} {{t}^{3}}\\
  +&125 {{s}^{94}} {{t}^{3}}+52 {{s}^{90}} {{t}^{3}}+126 {{s}^{86}} {{t}^{3}}+133 {{s}^{82}} {{t}^{3}}+72 {{s}^{78}} {{t}^{3}}+144 {{s}^{74}} {{t}^{3}}+115 {{s}^{70}} {{t}^{3}}+73 {{s}^{66}} {{t}^{3}}\\
  +&135 {{s}^{62}} {{t}^{3}}+98 {{s}^{58}} {{t}^{3}}+65 {{s}^{54}} {{t}^{3}}+99 {{s}^{50}} {{t}^{3}}+63 {{s}^{46}} {{t}^{3}}+48 {{s}^{42}} {{t}^{3}}+72 {{s}^{38}} {{t}^{3}}+37 {{s}^{34}} {{t}^{3}}+\underline{20 {{s}^{30}} {{t}^{3}}}\\
  +&36 {{s}^{26}} {{t}^{3}}+19 {{s}^{22}} {{t}^{3}}+10 {{s}^{18}} {{t}^{3}}+18 {{s}^{14}} {{t}^{3}}+{{s}^{10}} {{t}^{3}}+{{s}^{6}} {{t}^{3}}+9 {{s}^{2}} {{t}^{3}}+6 {{s}^{104}} {{t}^{2}}+6 {{s}^{96}} {{t}^{2}}\\
  +&9 {{s}^{92}} {{t}^{2}}+15 {{s}^{84}} {{t}^{2}}+15 {{s}^{80}} {{t}^{2}}+3 {{s}^{76}} {{t}^{2}}+24 {{s}^{72}} {{t}^{2}}+12 {{s}^{68}} {{t}^{2}}+9 {{s}^{64}} {{t}^{2}}+33 {{s}^{60}} {{t}^{2}}+12 {{s}^{56}} {{t}^{2}}\\
  +&12 {{s}^{52}} {{t}^{2}}+27 {{s}^{48}} {{t}^{2}}+6 {{s}^{44}} {{t}^{2}}+18 {{s}^{40}} {{t}^{2}}+21 {{s}^{36}} {{t}^{2}}+3 {{s}^{32}} {{t}^{2}}+12 {{s}^{28}} {{t}^{2}}+12 {{s}^{24}} {{t}^{2}}+3 {{s}^{20}} {{t}^{2}}\\
  +&9 {{s}^{16}} {{t}^{2}}+3 {{s}^{12}} {{t}^{2}}+3 {{s}^{4}} {{t}^{2}}+3 {{t}^{2}}+3 {{s}^{58}} t+3 {{s}^{46}} t+3 {{s}^{38}} t+3 {{s}^{34}} t+3 {{s}^{26}} t+3 {{s}^{22}} t+3 {{s}^{14}} t\\
  +&3 {{s}^{2}} t+1
\end{align*}

\end{document}